\PassOptionsToPackage{table,xcdraw}{xcolor}

\documentclass[10pt]{amsart}

\usepackage[
	letterpaper,
	twoside=true,
	textheight=22cm,
	textwidth=15cm,
	marginparsep=0.75cm,
	marginparwidth=2.5cm,
	heightrounded,
	centering
]{geometry}

\usepackage[protrusion=true]{microtype}

\usepackage{iftex}
\ifLuaTeX
  \usepackage{fontspec}
  \usepackage[math-style=TeX]{unicode-math}

\else
  \usepackage[bitstream-charter]{mathdesign}
  \usepackage[T1]{fontenc}
\fi

\usepackage{biolinum}

\renewcommand{\mathsf}[1]{\text{\normalfont\sffamily#1}}

\DeclareMathAlphabet{\eur}{U}{zeus}{m}{n}
\renewcommand{\mathcal}[1]{\eur{#1}}

\usepackage{amsmath, amsthm, thmtools, marginnote, ltxcmds, adjustbox, graphicx, stmaryrd}

\usepackage{mathtools, stackengine, scalerel}

\usepackage[nice]{nicefrac}

\usepackage[normalem]{ulem}

\usepackage{xr}

\usepackage{todonotes}

\makeatletter
\newcommand*{\saved@uline}{}
\let\saved@uline\uline

\newcommand*{\mathuline}{%
  \mathpalette{\math@uline\saved@uline}%
}
\newcommand*{\math@uline}[3]{%
  \mbox{#1{$#2#3\m@th$}}%
}

\renewcommand*{\uline}{%
  \relax  
  \ifmmode
    \expandafter\mathuline
  \else
    \expandafter\saved@uline
  \fi
}
\makeatother

\usepackage[style=alphabetic, backref=true, backend=biber, hyperref=true, giveninits=true]{biblatex}

\AtEveryBibitem{
    \clearfield{urlyear}
    \clearfield{urlmonth}
}

\AtBeginBibliography{\small}

\renewbibmacro*{pageref}{%
	\iflistundef{pageref}
	{}
	{%
		\printtext{\addperiod\space $\uparrow${\printlist[pageref][-\value{listtotal}]{pageref}}}}}

\DeclareLabelalphaTemplate{
	\labelelement{
		\field[final]{shorthand}
		\field{label}
		\field[strwidth=1,strside=left,ifnames=1]{labelname}
		\field[strwidth=1,strside=left]{labelname}
	}
}

\definecolor{cite}{HTML}{11871E}
\definecolor{url}{HTML}{698996}
\definecolor{link}{HTML}{912F1B}

\usepackage[pdfencoding=unicode, colorlinks=true, linkcolor=link, citecolor=cite, urlcolor=url, linktocpage]{hyperref}
\usepackage[hyphenbreaks]{breakurl}
\usepackage{xurl}
\hypersetup{breaklinks=true}

\hypersetup{final}

\usepackage{tikz}
\usetikzlibrary{cd}
\usetikzlibrary{arrows, arrows.meta, positioning, calc}
\tikzcdset{arrow style=tikz, diagrams={>={Straight Barb[scale=0.8]}}}

\tikzstyle{arrow} = [-{Straight Barb[scale=0.8]}, line width=0.2mm]

\usepackage{enumitem}
\newenvironment{myenum}[1]{%
\begin{enumerate}[label=#1,topsep=1pt,itemsep=0pt,partopsep=1pt,parsep=1pt]%
}%
{\end{enumerate}%
}

\usepackage[capitalise]{cleveref}

\Crefname{prop}{Proposition}{Propositions}
\Crefname{lem}{Lemma}{Lemmas}
\Crefname{cor}{Corollary}{Corollaries}
\Crefname{thm}{Theorem}{Theorems}

\Crefname{defn}{Definition}{Definitions}
\Crefname{notation}{Notation}{Notations}
\Crefname{conj}{Conjecture}{Conjectures}
\Crefname{ass}{Assumption}{Assumptions}
\Crefname{expt}{Expectation}{Expectations}

\Crefname{rmk}{Remark}{Remarks}
\Crefname{question}{Question}{Questions}
\Crefname{expl}{Example}{Examples}

\Crefname{figure}{Figure}{Figures}

\crefformat{equation}{(#2#1#3)}
\crefmultiformat{equation}{(#2#1#3)}{ and~(#2#1#3)}{, (#2#1#3)}{, and~(#2#1#3)}
\crefformat{section}{\S#2#1#3}
\crefmultiformat{section}{\S\S#2#1#3}{ and~#2#1#3}{, #2#1#3}{, and~#2#1#3}

\theoremstyle{plain}
\newtheorem{prop}[subsubsection]{Proposition}
\newtheorem{lem}[subsubsection]{Lemma}
\newtheorem{cor}[subsubsection]{Corollary}
\newtheorem{thm}[subsubsection]{Theorem}
\newtheorem*{thm*}{Theorem}

\theoremstyle{definition}
\newtheorem{defn}[subsubsection]{Definition}

\theoremstyle{remark}
\newtheorem{rmk}[subsubsection]{Remark}
\newtheorem*{rmk*}{Remark}
\newtheorem{expl}[subsubsection]{Example}

\numberwithin{equation}{subsection}

\newcommand{\teq}{\addtocounter{subsubsection}{1}\tag{\thesubsubsection}}

\DeclareMathOperator{\act}{\mathsf{act}}
\newcommand{\add}{\mathsf{add}}

\newcommand{\all}{\mathsf{all}}
\DeclareMathOperator{\Alg}{\mathsf{Alg}}
\DeclareMathOperator{\BiMod}{\mathsf{BiMod}}
\DeclareMathOperator{\assgr}{\mathsf{assgr}}
\DeclareMathOperator{\can}{can}
\DeclareMathOperator{\Cat}{\mathsf{Cat}}
\newcommand{\CC}{\mathbb{C}}
\DeclareMathOperator{\Cech}{\smash{\textnormal{\textsf{\v{C}ech}}}}
\DeclareMathOperator{\Ch}{\mathsf{Ch}}

\DeclareMathOperator{\cHom}{\mathcal{H}\mathsf{om}}
\DeclareMathOperator{\Chr}{\mathsf{Chr}}
\DeclareMathOperator{\Co}{\mathsf{C}}
\DeclareMathOperator{\coCone}{\mathsf{coCone}}

\DeclareMathOperator*{\colim}{\mathsf{colim}}
\DeclareMathOperator{\ComAlg}{\mathsf{ComAlg}}
\DeclareMathOperator{\Cone}{\mathsf{Cone}}
\DeclareMathOperator{\const}{\mathsf{const}}
\newcommand{\cont}{\mathsf{cont}}
\DeclareMathOperator{\Corr}{\mathsf{Corr}}
\DeclareMathOperator{\cuHom}{\smash{\uline{\smash{\mathcal{H}\mathsf{om}}}}}
\newcommand{\defeq}{\coloneqq}
\newcommand{\DG}{\mathsf{DG}}
\DeclareMathOperator{\DGCat}{\mathsf{DGCat}}
\newcommand{\DGCatidemex}{\DGCat_{\idem,\ex}}
\newcommand{\DGCatprescont}{\DGCat_{\pres,\cont}}
\DeclareMathOperator{\DVer}{\mathsf{D}_\mathsf{Ver}}
\newcommand{\En}{\mathsf{E}}

\newcommand{\enh}{\mathsf{enh}}
\newcommand{\eqdef}{\eqqcolon}
\newcommand{\etale}{\'etale}
\newcommand{\ex}{\mathsf{ex}}
\DeclareMathOperator{\Ext}{\mathsf{Ext}}
\newcommand{\FF}{\mathbb{F}}
\newcommand{\fin}{\mathsf{fin}}
\newcommand{\Fq}{\mathbb{F}_q}
\newcommand{\Fqn}{\mathbb{F}_{q^n}}
\newcommand{\Fqbar}{\overline{\mathbb{F}}_q}
\DeclareMathOperator{\Frob}{\mathsf{Frob}}
\DeclareMathOperator{\Fun}{\mathsf{Fun}}

\newcommand{\Ga}{\mathbb{G}_a}
\newcommand{\Gm}{\mathbb{G}_m}
\newcommand{\gr}{\mathsf{gr}}
\DeclareMathOperator{\Gr}{\mathsf{Gr}}

\newcommand{\ho}{\mathop{}\!\mathsf{h}}
\DeclareMathOperator{\Ho}{\mathsf{H}}
\DeclareMathOperator{\Hom}{\mathsf{Hom}}
\newcommand{\homflypt}{\textsc{homfly-pt}}
\newcommand{\horiz}{\mathsf{horiz}}
\DeclareMathOperator{\id}{\mathsf{id}}
\DeclareMathOperator{\IC}{\mathrm{IC}}
\newcommand{\idem}{\mathsf{idem}}
\DeclareMathOperator{\Idem}{\mathsf{Idem}}
\let\Im\relax
\DeclareMathOperator{\Im}{\mathsf{Im}}
\DeclareMathOperator{\Ind}{\mathsf{Ind}}

\newcommand{\Kuenneth}{K\"unneth}

\newcommand{\mixed}{\mathsf{m}}
\DeclareMathOperator{\Mod}{\mathsf{Mod}}
\DeclareMathOperator{\mult}{\mathsf{mult}}
\newcommand{\Negut}{Neguț}
\DeclareMathOperator{\nsimeq}{\not\simeq}
\DeclareMathOperator{\oblv}{\mathsf{oblv}}
\newcommand{\opp}{\mathsf{op}}

\DeclareMathOperator{\Perv}{\mathsf{Perv}}
\newcommand{\pervt}{{}^p\!t}

\DeclareMathOperator{\pHo}{{}^{\mathit{p}}\mathsf{H}}
\let\Pr\relax
\newcommand{\pr}{\mathsf{pr}}
\newcommand{\Pr}{\mathsf{Pr}}
\newcommand{\pres}{\mathsf{pres}}
\DeclareMathOperator{\PreStk}{\mathsf{PreStk}}
\newcommand{\pt}{\mathrm{pt}}
\DeclareMathOperator{\Pur}{\mathsf{Pur}}

\newcommand{\Qlbar}{\QQbar_{\ell}}
\newcommand{\QQbar}{\overline{\mathbb{Q}}}
\DeclareMathOperator{\rank}{\mathsf{rank}}
\newcommand{\ren}{\mathsf{ren}}
\DeclareMathOperator{\res}{\mathsf{res}}

\newcommand{\SBim}{\mathsf{SBim}}
\DeclareMathOperator{\Sch}{\mathsf{Sch}}
\DeclareMathOperator{\sh}{\mathsf{sh}}
\DeclareMathOperator{\Shv}{\mathsf{Shv}}
\newcommand{\sm}{\mathsf{sm}}
\DeclareMathOperator{\Spc}{\mathsf{Spc}}
\DeclareMathOperator{\Spec}{\mathsf{Spec}}
\DeclareMathOperator{\Sptr}{\mathsf{Sptr}}

\newcommand{\stab}{\mathsf{st}}
\DeclareMathOperator{\Stk}{\mathsf{Stk}}
\DeclareMathOperator{\Sym}{\mathsf{Sym}}
\newcommand{\tH}[1]{{}^\mathit{#1}\!\Ho}
\newcommand{\theart}{\heartsuit\!_t}
\DeclareMathOperator{\Tot}{\mathsf{Tot}}

\newcommand{\UC}{\mathsf{UC}}

\newcommand{\unren}{\mathsf{unren}}
\DeclareMathOperator{\Vect}{\mathsf{Vect}}
\newcommand{\vertc}{\mathsf{vert}}
\newcommand{\weightheart}{\heartsuit\!_w}
\DeclareMathOperator{\wt}{\mathsf{wt}}

\mathchardef\mhyphensymb="2D

\newcommand{\arrdisp}{0.33ex}

\newcommand{\alignsep}{\!\!\!\!\!\!}

\newcommand{\lbar}[1]{\overline{#1}}

\newcommand{\lrangle}[1]{\langle#1\rangle}

\newcommand{\wtilde}[1]{\widetilde{#1}}

\title{Revisiting mixed geometry}
\author{Quoc P. Ho}
\address{Department of Mathematics, The Hong Kong University of Science and Technology (HKUST), Clear Water Bay, Hong Kong}
\email{phuquocvn@gmail.com}
\author{Penghui Li}
\address{YMSC, Tsinghua University, Beijing, China}
\email{lipenghui@mail.tsinghua.edu.cn}
\date{\today}

\keywords{Mixed geometry, graded lifts, weight structures, $\ell$-adic sheaves, perverse sheaves, Hecke categories.}
\subjclass[2020]{Primary 14F08, 14A30. Secondary 14M15.}

\begin{document}
\begin{abstract}
We provide a uniform construction of ``mixed versions'' or ``graded lifts'' in the sense of Beilinson--Ginzburg--Soergel which works for arbitrary Artin stacks. In particular, we obtain a general construction of graded lifts of many categories arising in geometric representation theory and categorified knot invariants. Our new theory associates to each Artin stack of finite type $\mathcal{Y}$ over $\Fqbar$ a symmetric monoidal $\DG$-category $\Shv_{\gr, c}(\mathcal{Y})$ of constructible graded sheaves on $\mathcal{Y}$ along with the six-functor formalism, a perverse $t$-structure, and a weight (or co-$t$-)structure in the sense of Bondarko and Pauksztello, compatible with the six-functor formalism, perverse $t$-structures, and Frobenius weights on the category of (mixed) $\ell$-adic sheaves.

Classically, mixed versions were only constructed in very special cases due to the non-semisimplicity of Frobenius. Our construction sidesteps this issue by semi-simplifying the Frobenius action itself. However, the category $\Shv_{\gr, c}(\mathcal{Y})$ agrees with those previously constructed when they are available. For example, for any reductive group $G$ with a fixed pair $T\subset B$ of a maximal torus and a Borel subgroup, we have an equivalence of monoidal $\DG$ weight categories $\Shv_{\gr, c}(B\backslash G/B) \simeq \Ch^b(\SBim_W)$, where $\Ch^b(\SBim_W)$ is the monoidal $\DG$-category of bounded chain complexes of Soergel bimodules and $W$ is the Weyl group of $G$.
\end{abstract}

\maketitle
\tableofcontents

\setlength\emergencystretch{\hsize}

\section{Introduction}
\subsection{Motivation}
Frobenius weights in the theory of mixed $\ell$-adic sheaves developed by Deligne and Beilinson--Bernstein--Deligne--Gabber~\cites{deligne_conjecture_1980,beilinson_faisceaux_2018} have long played an important role in geometric representation theory. This went back to the works of Beilinson, Ginzburg, and Soergel~\cites{beilinson_mixed_1986,beilinson_koszul_1996} to more recent developments by Achar, Bezrukavnikov, Lusztig, Riche, Rider, Yun, and others~\cites{achar_koszul_2013,bezrukavnikov_koszul_2013,rider_formality_2013,lusztig_endoscopy_2020}, to name a few. In these cases, Frobenius weights provide an extra grading, allowing one to realize combinatorially defined categories with a grading, such as the category of Soergel bimodules, with geometrically defined categories, such as the category of mixed sheaves on $B\backslash G/B$ (up to some modification). Such a grading is indispensable to formality results and Koszul duality phenomena in geometric representation theory. But its importance extends beyond representation theory. For example, in~\cite{khovanov_triply-graded_2007}, Khovanov used Soergel bimodules to construct a triply-graded link homology (also known as \homflypt{} homology), categorifying \homflypt{} knot polynomials. Frobenius weights thus made their way into the theory of categorified knot invariants as well via the works of, for example, Webster--Williamson, Shende--Treumann--Zaslow, Bezrukavnikov--Tolmachov, and Trinh~\cites{webster_geometric_2017,shende_legendrian_2017,bezrukavnikov_monodromic_2022,trinh_hecke_2021}.

In all of these cases, to match geometric and combinatorially defined categories, a ``mixed version'' (a.k.a. ``graded lift'') of the category of sheaves (on the geometric side) has to be constructed via some artificial ``cooking'' (in the words of Romanov--Williamson, see~\cite{romanov_langlands_2021}) which only works in \emph{very special cases}, where the space involved has a \emph{nice} stratification by affine spaces. The ``cooking'' is necessary to kill non-trivial extensions between \emph{pure objects of the same weights}. Due to Frobenius' non-semisimplicity, a fundamental problem already over a point, the mixed version is \emph{not} simply the category of mixed sheaves in the sense of~\cite{beilinson_faisceaux_2018}. Roughly speaking, all known constructions of a mixed version involve picking a collection $\Pur_c(\mathcal{Y})$ of (constructible) semi-simple complexes on a stack $\mathcal{Y}$ of interest and show that the homotopy category of chain complexes $K^b(\Pur_c(\mathcal{Y}))$ recovers a known object previously combinatorially defined.\footnote{For example, when $\mathcal{Y} = B\backslash G/B$, we obtain the homotopy category of bounded chain complexes of Soergel bimodules.} However, to relate it back to the usual sheaf theory, known techniques, which rely on \emph{Tate geometry} and which are subtle and technical,\footnote{We loosely use the term Tate geometry to refer to the situation where the space involved has a nice stratification by affine spaces and the sheaves involved satisfy some strong purity condition.} only work in very special cases.\footnote{This also applies to the more recent work of Soergel, Virk, and Wendt~\cites{soergel_perverse_2018,soergel_equivariant_2018}. See \cref{rmk:compare_motives}.} Moreover, the construction $K^b(\Pur_c(-))$ loses most of the functoriality that makes the geometric language powerful: we no longer know how to pull (resp. push) along non-smooth (resp. non-proper) morphisms. This problem is not easy to fix: preservation of semi-simplicity is part of the Tate conjecture, which is still wide open. 

By semi-simplifying the Frobenius actions at the categorical level, the current paper provides a uniform construction of mixed versions or graded lifts of the usual category of sheaves\footnote{\label{ftn:ignoring_Shv_infty}In fact, we need a slightly smaller category than the category of sheaves. This technical point will, however, be ignored in the introduction. See \cref{subsec:mixed_geometry} for more precise statements.} \emph{in complete generality}. The construction itself is surprisingly simple (one might even say that it is naive), which is rather remarkable considering that it recovers and vastly generalizes all known constructions of mixed versions based on $\ell$-adic sheaves. The main thrust of the paper, which is also most involved and subtle, thus revolves around showing that such a construction does indeed give the correct category.

\begin{rmk}\label{rmk:compare_motives}
Soergel, Wendt, and Virk, in~\cites{soergel_perverse_2018,soergel_equivariant_2018}, and more recently\footnote{after the current paper is available}, Eberhardt and Scholbach, in~\cite{eberhardt_integral_2023}, used the theory of motives and weight structures to construct graded lifts of various categories in geometric representation theory. Their methods also avoid the ad-hoc step of taking $K^b(\Pur_c(-))$. Nonetheless, one crucial advantage of our work compared to theirs is that while their methods (which uses mixed Tate motives) only work for (equivariant) spaces with a Whitney--Tate stratification, our theory works for \emph{any} Artin stacks out of the box. For example, while their theory is sufficient to deal with $B\backslash G/B$ and, with some care, the nilpotent cone, genuinely new ideas are needed (see~\cite[\S1.11]{soergel_equivariant_2018}) to deal with $\frac{G}{B}$ and $\frac{G}{G}$ (with the adjoint actions), a natural object in the theory of character sheaves. This is one of the main motivations for us in writing the current paper. Another advantage is that the general existence of the analog of the perverse $t$-structure on the triangulated categories of mixed motives depends on the standard conjectures on algebraic cycles (which are wide open) while this $t$-structure is available to us in complete generality.
\end{rmk}

\subsection{A sketch of the theory}
Our theory associates to each scheme/stack $\mathcal{Y}$ over $\Fqbar$ a triangulated (in fact, $\DG$-)category $\Shv_{\gr}(\mathcal{Y})^{\ren}$ (resp. $\Shv_{\gr, c}(\mathcal{Y})$) of \emph{graded sheaves} (resp. \emph{constructible graded sheaves}) on $\mathcal{Y}$, along with the six-functor formalism, a perverse $t$-structure, and a weight (or co-$t$-)structure in the sense of Pauksztello and Bondarko~\cites{pauksztello_compact_2008, bondarko_weight_2010}. As part of the definition of a weight structure, pure objects are semi-simple. In particular, we have a decomposition theorem similar to, but in some sense, more convenient than~\cite{beilinson_faisceaux_2018}. Moreover, the weight complex functor of~\cites{bondarko_weight_2010,sosnilo_theorem_2019,aoki_weight_2020} provides a comparison $\wt: \Shv_{\gr, c}(\mathcal{Y}) \to \Ch^b(\ho\Shv_{\gr, c}(\mathcal{Y})^{\weightheart})$, where $\Ch^b(\ho\Shv_{\gr, c}(\mathcal{Y})^{\weightheart})$ is the $\DG$-category of bounded chain complexes in the underlying homotopy category $\ho\Shv_{\gr, c}(\mathcal{Y})^{\weightheart}$ of the weight heart $\Shv_{\gr, c}(\mathcal{Y})^{\weightheart}$, which plays the role of $\Pur_c(\mathcal{Y})$.

More precisely, for any $\Fq$-form $\mathcal{Y}_n$ of $\mathcal{Y}$, $\Shv_{\gr, c}(\mathcal{Y})$ fits into the following diagram
\[
\begin{tikzcd}
	\Shv_{\mixed, c}(\mathcal{Y}_n) \ar{r}{\gr_{\mathcal{Y}_n}} & \Shv_{\gr, c}(\mathcal{Y}) \ar{r}{\oblv_\gr} \ar{d}{\wt} & \Shv_c(\mathcal{Y}) \\
	& \Ch^b(\ho\Shv_{\gr, c}(\mathcal{Y})^{\weightheart})
\end{tikzcd}
\]
where $\gr_{\mathcal{Y}_n}$ and $\oblv_\gr$ are compatible with all of the six-functor formalism, the perverse $t$-structures, and Frobenius weights. In special cases, such as the case of $B\backslash G/B$, $\Shv_{\gr, c}(\mathcal{Y})^{\weightheart}$ is classical, i.e., $\Shv_{\gr, c}(\mathcal{Y})^{\weightheart} \simeq \ho\Shv_{\gr, c}(\mathcal{Y})^{\weightheart}$,\footnote{See \cref{subsubsec:weight_heart} for a more detailed discussion.} and the weight complex functor $\wt$ is an equivalence of categories. In fact, $\Shv_{\gr, c}(\mathcal{Y})^{\weightheart} \simeq \ho\Pur_c(\mathcal{Y})$ in this case,\footnote{Since we work within the framework of $\DG$-/$\infty$-categories, there is a $\Hom$-space rather than a $\Hom$-set between any two objects in our categories. Roughly speaking, under Dold--Kan, ``negative $\Ext$'s'' are part of the $\Hom$-spaces. Moreover, the process of taking the underlying homotopy category means that we forget these negative $\Ext$'s. In this sense, for example, our $\ho\Pur_c$ corresponds to $\Pur_c$ in the literature.} see \cref{prop:criterion_weight_heart_gr_sheaves_classical}, and we obtain a functor
\[
	\Ch^b(\ho \Pur_c(\mathcal{Y})) \xrightarrow{\oblv_\gr \circ \wt^{-1}} \Shv_c(\mathcal{Y})
\]
which realizes the former as a graded lift of the latter, recovering the classical constructions of graded lifts.\footnote{See also \cref{ftn:ignoring_Shv_infty}.}

The condition for $\Shv_{\gr, c}(\mathcal{Y})^{\weightheart}$ to be classical is, by \cref{prop:criterion_weight_heart_gr_sheaves_classical},  equivalent to a certain $\Hom$-purity condition on the category of mixed sheaves. In turn, $\Hom$-purity itself can be established using, for example, Tate geometry when that is available as it has already been done. However, the graded lift $\Shv_{\gr, c}(\mathcal{Y})$ exists in general and when $\Shv_{\gr, c}(\mathcal{Y})^{\weightheart}$ is not classical, it is in fact a genuinely new category not available previously. We do this by circumventing the Frobenius non-semisimplicity problem: instead of picking semi-simple objects in the category, we systematically semi-simplify the Frobenius actions at the categorical level. Consequently, functoriality comes for free and known results in the theory of mixed $\ell$-adic sheaves have direct translations in our setting. 

\begin{rmk}
	We expect that one can obtain a similar theory by replacing mixed $\ell$-adic sheaves with mixed Hodge modules or twistor $D$-modules by tensoring up from the category of mixed Hodge modules over a point using the functor of taking associated graded with respect to the weight filtration. We hope to return to this in a future publication. 
\end{rmk}

\subsection{Applications} \label{subsec:future_work}
As mentioned above, in writing this paper, we are motivated by questions from geometric representation theory and categorified knot invariants. As a proof of concept, we will conclude the paper with a streamlined proof of the fact that $\Shv_{\gr, c}(B\backslash G/B)$ is equivalent to the $\DG$-category of bounded chain complexes of Soergel bimodules $\Ch^b(\SBim_W)$, matching the geometrically and combinatorially defined versions of Hecke categories. More generally, we expect that the theory developed in this paper will have a much wider applicability. In particular, in forthcoming papers (some of which already appeared), we will utilize the robust theory of graded sheaves developed here
\begin{myenum}{(\roman*)}
	\item (already appeared in~\cite{ho_graded_2023}) to compute the categorical trace and Drinfel'd center of the graded Hecke categories (i.e., the $\DG$-categories of bounded chain complexes of Soergel bimodules), using the techniques inspired by Ben-Zvi, Nadler, and Preygel~\cites{ben-zvi_character_2009, ben-zvi_spectral_2017}, and obtain, as a consequence, a proof of
	\begin{enumerate}[label=--]
		\item a conjecture of Gorsky--\Negut--Rasmussen in~\cite{gorsky_flag_2021} on the relation between \homflypt{} link homology and Hilbert schemes of points on $\CC^2$;
		\item a conjecture of Gorsky--Hogancamp--Wedrich in~\cite{gorsky_derived_2021} on the formality of the Hochschild homologies of the graded Hecke categories;
	\end{enumerate}
	\item (forthcoming) to construct a Koszul duality equivalence between the categories of (graded) character sheaves of Langlands dual groups, upgrading Ben-Zvi--Nadler's 2-periodized result; and
	\item (forthcoming) to upgrade Bezrukavnikov's theorem on the two geometric realizations of affine Hecke algebras to the graded setting, as conjectured in~\cite[\S11.1]{bezrukavnikov_two_2016}.
\end{myenum}
Beyond the applications and ongoing work described above, the theory developed in this paper has also appeared crucially in the formulations of the many conjectures by Ben-Zvi, Chen, Helm, Nadler, Sakellaridis, and Venkatesh on relative Langlands program as well as other refinements of the Langlands program, see, for example~\cite{ben-zvi_between_2023,ben-zvi_relative_2023}.

\subsection{Conventions}
\label{subsec:conventions}
Throughout the paper, we fix $k_1 = \Fq$ where $q$ is a fixed power of a fixed prime $p$. More generally, for any positive integer $n$, we use $k_n = \Fqn$ to denote the degree $n$ field extension of $k_1$. We also use $k=\Fqbar$ to denote the algebraic closure of $\Fq$. A scheme/stack defined over $k_n$ will be written as $X_n$. We will use $X_{n'}$ to denote its base change to $k_{n'}$ and $X$ its base change to $k$.

All schemes that appear in the paper are separated and of finite type. All stacks are Artin, of finite type, and have affine stabilizers. For any field $\FF$, we use $\Sch_{\FF}$ (resp. $\Stk_{\FF}$) to denote the category of schemes (resp. stacks) over $\FF$, satisfying the conditions above. Unless otherwise specified, $\FF \in \{k_n, k\}_{n\in \mathbb{Z}} = \{\Fqn, \Fqbar\}_{n\in \mathbb{Z}}$ in this paper. We also use $\pt_{\FF} \in \Sch_{\FF}$ to denote the final object. Following the convention written in the above paragraph, we also use $\pt_n$ to denote $\Spec k_n = \Spec \Fqn$ and $\pt = \Spec k = \Spec \Fqbar$.

Unless otherwise specified, by a category, we always mean an $\infty$-category (or more precisely, an $(\infty, 1)$-category). Given a triangulated/stable $\infty$-category $\mathcal{C}$ equipped with a $t$-structure, the heart of the $t$-structure is denoted by $\mathcal{C}^{\heartsuit}$. Similarly, the heart of a weight structure is denoted by $\mathcal{C}^{\weightheart}$. We will also use $\mathcal{C}^{\theart}$ to denote the heart of a $t$-structure on $\mathcal{C}$ to emphasize the $t$-structure, especially when a weight structure is also present. Lastly, categories and functors are derived by default.

\subsection{The main results}
\label{subsec:main_results}

\subsubsection{Abstract theory}
For each stack $\mathcal{Y}\in \Stk_k$, we construct a symmetric monoidal $\DG$-category $\Shv_{\gr, c}(\mathcal{Y})$ of \emph{constructible graded sheaves} on $\mathcal{Y}$ along with its ``large'' cousin, the category of renormalized graded sheaves $\Shv_\gr(\mathcal{Y})^\ren \simeq \Ind(\Shv_{\gr, c}(\mathcal{Y}))$, where $\Ind$ denotes the process of taking ind-completion, i.e., formally adding all filtered colimits. 

For any $k_n$-form $\mathcal{Y}_n \in \Stk_{k_n}$ of $\mathcal{Y}$, we have
\[
	\Shv_\gr(\mathcal{Y})^\ren \simeq \Shv_\mixed(\mathcal{Y}_n)^\ren \otimes_{\Shv_\mixed(\pt_n)} \Vect^\gr.
\]
Here, $\Vect^\gr \defeq \Fun(\mathbb{Z}, \Vect)$ denotes the category of graded chain complexes of vector spaces over $\Qlbar$ and $\Shv_\mixed(\mathcal{Y}_n)^\ren \defeq \Ind(\Shv_{\mixed, c}(\mathcal{Y}_n))$ where $\Shv_{\mixed, c}(\mathcal{Y}_n)$ is the (derived) category of constructible mixed $\ell$-adic sheaves over $\mathcal{Y}_n$ in the sense of \cites{beilinson_faisceaux_2018,laszlo_perverse_2009}. Moreover, the tensor product is the one defined by Lurie, where $\Shv_\mixed(\mathcal{Y}_n)^\ren$ is tensored over $\Shv_\mixed(\pt_n)$ via pulling back and the symmetric monoidal functor $\Shv_\mixed(\pt_n) \to \Vect^\gr$ is obtained by turning Frobenius weights into a grading. In particular, $\Shv_\gr(\pt) \simeq \Vect^\gr$, which is semi-simple.

In fact, we take this as the definition of $\Shv_\gr(\mathcal{Y})^\ren$, see \cref{defn:Shv_gr}, and prove that this definition is independent of the choice of a $k_n$-form, see \cref{thm:Shv_gr_is_a_sheaf_theory_on_Stk_k}. Moreover, the category $\Shv_{\gr, c}(\mathcal{Y})$ is defined to be the full subcategory of $\Shv_\gr(\mathcal{Y})^\ren$ spanned by compact objects.

\begin{rmk}
While the small category of constructible sheaves is more explicit, dealing with its large cousin is unavoidable since, for example, the push-forward functor between stacks does not preserve constructibility most of the time.
\end{rmk}

\begin{rmk}
For any stack $\mathcal{Y} \in \Stk_k$, there is also the usual category $\Shv(\mathcal{Y})$ obtained, for instance, by using descent via a smooth atlas $S \to \mathcal{Y}$ where $S \in \Sch_k$.\footnote{We ignore Frobenius action/mixedness since it is not relevant for the current discussion.} However, for many purposes, this category is not what we want: even in the simplest case $\mathcal{Y} = B\Gm$, the classifying stack of the multiplicative group $\Gm$, the constant sheaf on $\mathcal{Y}$ is not compact in $\Shv(\mathcal{Y})$. See \cref{expl:constant_sheaves_not_compact}. Note, however, that the distinction between renormalized and usual is only relevant for stacks. 

The functoriality of the renormalized sheaf theory is similar to that of the usual sheaf theory and the two are closely related. To distinguish between the two, we will include the subscript $\ren$ in the notations of the various pull and push functors when working with renormalized sheaves. Note, however, that when applied to constructible sheaves, functors on the two sides are identified in a precise sense. This is reviewed in \cref{subsubsec:ind_constructible_on_stacks} and subsequent subsubsections.
\end{rmk}

By construction, $\Shv_\gr(\mathcal{Y})^\ren$ is tensored over $\Vect^\gr$ and moreover, this structure is compatible with the usual operations of sheaves. We have the following result that compares the graded sheaf theory and the theory of usual/mixed sheaves.

\begin{thm}[\cref{thm:compatibility_functoriality_mixed_graded_non_mixed,lem:oblv_gr_conservative}]
Let $\mathcal{Y} \in \Stk_k$ and $\mathcal{Y}_n \in \Stk_{k_n}$ a $k_n$-form of $\mathcal{Y}$. Then, we have a sequence of symmetric monoidal functors
\[
	\Shv_\mixed(\mathcal{Y}_n)^\ren \xrightarrow{\gr_{\mathcal{Y}_n}} \Shv_\gr(\mathcal{Y})^\ren \xrightarrow{\oblv_\gr} \Shv(\mathcal{Y})^\ren
\]
where $\oblv_\gr$ is conservative. Moreover, these functors are compatible with various pull and push functors along $f_n: \mathcal{Y}_n \to \mathcal{Z}_n$ and its base change to $k$, $f: \mathcal{Y} \to \mathcal{Z}$.
\end{thm}

Even though $\Shv_{\gr,c}(\mathcal{Y})$ is defined in an abstract way, one can understand it quite explicitly. Namely, its objects are direct summands of finite colimits of objects in the essential image of $\gr_{\mathcal{Y}_n}$. Moreover, given $\mathcal{F}_n, \mathcal{G}_n \in \Shv_{\mixed, c}(\mathcal{Y}_n)$,
\[
	\cHom_{\Shv_\gr(\mathcal{Y})^\ren}(\gr_{\mathcal{Y}_n} (\mathcal{F}_n), \gr_{\mathcal{Y}_n} (\mathcal{G}_n)) \in \Vect
\]
is simply the Frobenius weight $0$ part of the mixed-$\cHom$ complex, see \cref{prop:hom_graded_sheaves}. This simple, but important, observation allows us to show that $\Shv_{\gr, c}(\mathcal{Y}_n)$ has a transversal weight and $t$-structures in the sense of~\cite{bondarko_weight_2012}. See \cref{thm:bondarko_main_thm} for a quick review.

\begin{thm}[\cref{thm:weight_t_structures_graded_sheaves}]
For any $\mathcal{Y} \in \Stk_k$, $\Shv_{\gr, c}(\mathcal{Y})$ is equipped with a weight and a $t$-structure such that the $t$-structure is transversal to the weight structure.
	
This $t$-structure will be referred to as the perverse $t$-structure, whose heart is the category of graded perverse sheaves $\Perv_{\gr, c}(\mathcal{Y}) \defeq \Shv_{\gr, c}(\mathcal{Y})^{\theart}$.
\end{thm}

Moreover, the weight and $t$-structures are compatible with various functors as expected. For instance, we have the following results. Many similar results are established in \cref{subsec:formal_properties_weight_Shv_gr}.

\begin{prop}[\cref{prop:t-exactness}]
Let $\mathcal{Y}_n \in \Stk_{k_n}$ and $\mathcal{Y}$ its base change to $k$. Then, the functors
\[
	\Shv_{\mixed, c}(\mathcal{Y}_n) \xrightarrow{\gr} \Shv_{\gr, c}(\mathcal{Y}) \xrightarrow{\oblv_\gr} \Shv_c(\mathcal{Y})
\]
preserves and reflects $t$-structures with respect to the perverse $t$-structures. Namely, for any $n$ and any $\mathcal{F} \in \Shv_{\gr, c}(\mathcal{Y})$, $\mathcal{F} \in \Shv_{\gr, c}(\mathcal{Y})^{t\leq n}$ (resp. $\mathcal{F} \in \Shv_{\gr, c}(\mathcal{Y})^{t\geq n}$) if and only if $\oblv_\gr(\mathcal{F}) \in \Shv_c(\mathcal{Y})^{t\leq n}$ (resp. $\oblv_\gr(\mathcal{F}) \in \Shv_c(\mathcal{Y})^{t\geq n}$). We have a similar statement for $\gr$.
\end{prop}

\begin{prop}[\cref{prop:w-exactness}]
Let $\mathcal{Y}$ and $\mathcal{Y}_n$ be as above. The functor $\gr: \Shv_{\mixed, c}(\mathcal{Y}_n) \to \Shv_{\gr, c}(\mathcal{Y})$ preserves and reflects weights, where we use Frobenius weights on $\Shv_{\mixed, c}(\mathcal{Y}_n)$ and our weight structure on $\Shv_{\gr, c}(\mathcal{Y})$. Namely, for any $k$ and $\mathcal{F} \in \Shv_{\mixed, c}(\mathcal{Y})$, then $\mathcal{F} \in \Shv_{\mixed, c}(\mathcal{Y})^{w\leq k}$ (resp. $\mathcal{F} \in \Shv_{\mixed, c}(\mathcal{Y})^{w\geq k}$) if and only if $\gr(\mathcal{F}) \in \Shv_{\gr, c}(\mathcal{Y})^{w \leq k}$ (resp. $\gr(\mathcal{F}) \in \Shv_{\gr, c}(\mathcal{Y})^{w \geq k}$).
\end{prop}

We will now relate our theory and previous attempts in the literature. Let $\Shv_{\infty,c}(\mathcal{Y})$ denote the smallest full triangulate subcategory of $\Shv_c(\mathcal{Y})$ containing the essential image of $\Shv_{\mixed, c}(\mathcal{Y}_m)$ under pullbacks for all $k_m$-forms of $\mathcal{Y}$ (for all $m$). We show that this category inherits the perverse $t$-structure from $\Shv_c(\mathcal{Y})$ whose heart is denoted by $\Perv_{\infty, c}(\mathcal{Y})$.

\begin{thm}[\cref{prop:gr_shv_grading_ala_BGS,prop:gr_shv_mixed_version_ala_Rider}]
$\Perv_{\gr, c}(\mathcal{Y})$ is a graded version of $\Perv_{\infty, c}(\mathcal{Y})$ in the sense of~\cite[Defn. 4.3.1.(1)]{beilinson_koszul_1996}.

$\Shv_{\gr, c}(\mathcal{Y})$ is a mixed version of $\Shv_{\infty, c}(\mathcal{Y})$ in the sense of~\cite[Defn. 4.2]{rider_formality_2013}.
\end{thm}

\subsubsection{Application to Hecke categories and Soergel bimodules}
We conclude the paper with an application to representation theory. Let $G$ be a reductive group with a fixed pair $T \subseteq B$ of Borel $B$ and a maximal torus. Let $W$ be the Weyl group of $G$ and consider the monoidal $\DG$-category of bounded chain complexes of Soergel bimodules $\Ch^b(\SBim_W)$. We show that this category has a natural geometric incarnation using graded sheaves.

\begin{thm}[\cref{thm:geometric_incarnation_Hecke}]
We have an equivalence of monoidal categories
\[
	\Shv_{\gr, c}(B\backslash G/B) \simeq \Ch^b(\SBim_W).
\]
\end{thm}

As graded sheaves interact seamlessly with the usual theory of sheaves/mixed sheaves, the theorem above is obtained by readily using known results about $B\backslash G/B$ proved earlier in~\cites{soergel_kategorie_1990,bezrukavnikov_koszul_2013}.

\subsection{An outline of the paper}
In \cref{sec:graded_sheaves_construction_formal_properties}, we construct the categories of graded sheaves and establish their formal properties. In \cref{sec:weight_perverse_t_structures_graded_sheaves}, using the theory developed by Bondarko in~\cite{bondarko_weight_2012}, we show that the category of constructible graded sheaves on any Artin stack admits a transversal weight and $t$-structures. We show that the weight structure is compatible with Frobenius weights in the theory of mixed sheaves and that the $t$-structure is compatible with the perverse $t$-structure of~\cite{beilinson_faisceaux_2018}. Finally, in \cref{sec:Hecke_categories}, we show that the category of constructible graded sheaves on the finite Hecke stack $B\backslash G/B$ is equivalent, as a monoidal $\DG$ category, to $\Ch^b(\SBim_W)$, the category of bounded chain complexes of Soergel bimodules.

In \cref{sec:generalities_DGCats}, we review the necessary background regarding $\DG$-categories. It contains mostly known results, except possibly for \cref{subsec:induction_formula_enriched_hom}, where we prove an induction formula for the enriched $\Hom$-spaces, crucial in our study of graded sheaves.

\section{Graded sheaves: construction and formal properties}
\label{sec:graded_sheaves_construction_formal_properties}
In this section, we will construct the category of graded sheaves on schemes or stacks and prove their formal properties (see \cref{subsec:conventions} for our conventions regarding schemes and stacks). In preparation for the construction, we start, in \cref{subsec:ell-adic_sheaves,subsec:mixed_sheaves,subsec:Shv_mixed(pt_n)-module_structures}, with a brief review of the theory of $\ell$-adic sheaves and mixed sheaves on stacks in the form that is useful for our purposes. We construct the category of graded sheaves in \cref{subsec:graded_sheaves_and_functoriality} and study their functoriality in \cref{subsec:functoriality_graded_sheaves}. The streamlined construction of the category of graded sheaves means that six-functor formalism carries over to the graded setting in a straightforward manner. In \cref{subsec:hom_graded_sheaves}, we study the $\Hom$-spaces between graded sheaves. Despite being an easy application of \cref{subsec:induction_formula_enriched_hom}, the result in this subsection serves as the foundation for most of the important results concerning graded sheaves in this paper. Possibly most interestingly, we show in \cref{subsec:invariance_extensions_scalars} that in a precise sense, the category of graded sheaves is invariant under extensions of scalars. In other words, for a stack $\mathcal{Y}_n$ defined over $\pt_n = \Spec \Fqn$, the category of graded sheaves on it only depends on $\mathcal{Y}$, the base change of $\mathcal{Y}_n$ to $\Fqbar$. Finally, in \cref{subsec:graded_sheaves_and_correspondences}, we explain how the results of~\cites{liu_enhanced_2017,liu_enhanced_2012} can be used to upgrade our theory of graded sheaves to a functor out of the category of correspondences on stacks. This captures all the base change isomorphisms in a homotopy coherent way, which is useful when one wants to construct an $\infty$-monoidal category using convolutions, a pattern often seen in geometric representation theory.

\subsection{\texorpdfstring{$\ell$}{ℓ}-adic sheaves}
\label{subsec:ell-adic_sheaves}
We start with a quick review of the theory of $\ell$-adic sheaves. The materials presented here are developed in detail in~\cites{gaitsgory_atiyah-bott_2015,gaitsgory_weils_2019,hemo_constructible_2021} and~\cite[Appendix F.5]{arinkin_stack_2022}. 

\subsubsection{$\ell$-adic sheaves on schemes}
For any scheme $S \in \Sch_{\FF}$ (see \cref{subsec:conventions}), we let $\Shv(S)$ denote the ind-completion of the category $\Shv_c(S)$ of constructible $\ell$-adic ($\Qlbar$) sheaves on $S$. In other words, we have
\[
	\Shv(S) = \Ind(\Shv_c(S)) \qquad\text{and}\qquad \Shv_c(S) = \Shv(S)^c.
\]
For example, for $\pt \in \Sch_k$, where, following our conventions in \cref{subsec:conventions}, $k = \Fqbar$, $\Shv(\pt) = \Vect$ and $\Shv_c(\pt) = \Vect^c$. Here, $\Vect$ denotes the $\DG$-category of chain complexes of vector spaces over $\Qlbar$, and $\Vect^c$ the full subcategory spanned by compact objects, which are perfect complexes.

Note that since all functors in the six-functor formalism respect constructibility, the resulting functors between the large categories (i.e., after ind-completion) are obtained by ind-extensions. Thus, all of these functors are continuous and compactness-preserving.

Using the pullback functors $(-)^*$ and $(-)^!$, we obtain
\[
	\Shv^?: \Sch_{\FF}^{\opp} \to \DGCatprescont
\]
where $?$ is either $*$ or $!$. Similarly, for the \emph{small} versions, we have
\[
	\Shv_c^?: \Sch_{\FF}^{\opp} \to \DGCatidemex.
\]

\subsubsection{$\ell$-adic sheaves on stacks}
\label{eq:subsubsec_ell_adic_stacks}
Right Kan extending $\Shv^?$ along $\Sch_{\FF}^\opp \hookrightarrow \Stk_{\FF}^\opp$, we obtain
\[
	\Shv^?: \Stk_{\FF}^{\opp} \to \DGCatprescont.
\]
More concretely, for $\mathcal{Y} \in \Stk_{\FF}$,
\[
	\Shv^?(\mathcal{Y}) = \lim_{S \in \Sch_{\FF/\mathcal{Y}}^\opp} \Shv(S),
\]
where the transition functors are either $(-)^*$ or $(-)^!$. Since the property of satisfying descent is preserved by right Kan extension~\cite[Prop. 6.4.3]{gaitsgory_ind-coherent_2011}, the new sheaf theories satisfy smooth descent. In particular, if $h: Y\to \mathcal{Y}$ is a smooth atlas where $Y$ is a scheme, then
\[
	\Shv^?(\mathcal{Y}) = \Tot(\Shv^?(\Cech^\bullet(Y/\mathcal{Y}))) \teq\label{eq:descent_sheaves_stacks}
\]
where $\Cech^\bullet(Y/\mathcal{Y})$ is the (simplicial) $\Cech$ nerve of the covering $Y\to \mathcal{Y}$, and $\Tot$ is the procedure of taking totalization, i.e., limit, of a co-simplicial object.

\subsubsection{} As in~\cite[Prop. 11.4.3]{gaitsgory_ind-coherent_2011}, for any $\mathcal{Y} \in \Stk_{\FF}$, $\Shv^!(\mathcal{Y}) \simeq \Shv^*(\mathcal{Y})$. This is essentially due to the fact that the diagram used to compute the totalization \cref{eq:descent_sheaves_stacks} for the two theories are equivalent, given by a shift. Thus, we will simply write $\Shv(\mathcal{Y})$ from now on. In particular, for $f: \mathcal{Y} \to \mathcal{Z}$ a morphism in $\Stk_{\FF}$, we have functors
\[
	f^*, f^!: \Shv(\mathcal{Z}) \to \Shv(\mathcal{Y}). \teq\label{eq:Shv_!_vs_*_stacks}
\]
By construction, $f^*$ is continuous, and thus, has a right adjoint $f_*$. $f^!$ is co-continuous, inherited from same property of $(-)^!$ functors between schemes, and thus has a left adjoint $f_!$.\footnote{$f^!$ is in fact also continuous. Thus, it also has a right adjoint. However, this right adjoint is rarely used and in particular, we do not make use of this functor in this paper.}

\begin{rmk}
Let $\PreStk_{\FF} = \Fun(\Sch_{\FF}^\opp, \Spc)$ be the category of prestacks. The same considerations yield two sheaf theories $\Shv^!$ and $\Shv^*$ on prestacks, which are useful, for instance, in the studies of affine Grassmannians as well as their factorizable versions. In this generality, however, \cref{eq:Shv_!_vs_*_stacks} does not always hold. We will revisit this theory in a subsequent paper.
\end{rmk}

\subsubsection{}
By~\cite[Vol. I, Chap. 1, Prop. 2.5.7]{gaitsgory_study_2017}, we obtain a alternative description of $\Shv^!(\mathcal{Y})$ from \cref{eq:descent_sheaves_stacks}
\[
	\Shv(\mathcal{Y}) = |\Shv(\Cech^\bullet(Y/\mathcal{Y}))|,
\]
where we now use the $(-)_!$-functor to move between different schemes. Moreover, $|-|$ denotes the procedure of taking geometric realization, i.e., colimit, of a simplicial object. Since all the functors used in the colimit preserve compactness, the category $\Shv(\mathcal{Y})$ is compactly generated with a set of compact generators given by $h_!(\Shv_c(Y))$, by~\cite[Cor. 1.9.4]{drinfeld_compact_2015}. It is important that the colimit is taken inside $\DGCatprescont$.

\subsubsection{$\Ind$-constructible sheaves on stacks} \label{subsubsec:ind_constructible_on_stacks}
The constant sheaves on most stacks are not compact, see \cref{expl:constant_sheaves_not_compact} below. On the other hand, it is constructible. The \emph{renormalization} procedure described in this subsubsection enlarges the category of sheaves so that constructible sheaves become compact in this new category. The difference between the \emph{renormalized} version and the usual version is parallel to the difference between coherent complexes and perfect complexes, or more precisely, between ind-coherent sheaves and quasi-coherent sheaves since we are working in the \emph{large} category setting. A more detailed discussion can be found in~\cite[Appendix F.5]{arinkin_stack_2022}.

\begin{defn}
For $\mathcal{Y} \in \Stk_{\FF}$, we let $\Shv_c(\mathcal{Y})$ denote the full-subcategory of $\Shv(\mathcal{Y})$ consisting of \emph{constructible} objects, defined to be those whose pullbacks along $s: S \to \mathcal{Y}$ either via $s^*$ or $s^!$ are compact/constructible for any scheme $S$. 

We use $\Shv(\mathcal{Y})^{\ren} \defeq \Ind(\Shv_c(\mathcal{Y}))$ to denote the renormalized category of sheaves on $\mathcal{Y}$.
\end{defn}

\begin{rmk}
When $\mathcal{Y}$ is a scheme rather than a stack, $\Shv(\mathcal{Y})$ and $\Shv(\mathcal{Y})^\ren$ coincide by construction.
\end{rmk}

\begin{lem} \label{lem:constructiblity_stack_*_!_pullback}
Let $\mathcal{Y} \in \Stk_{\FF}$ and $h: Y \to \mathcal{Y}$ a smooth atlas. Then $\mathcal{F} \in \Shv(\mathcal{Y})$ is constructible if and only if $h^* \mathcal{F}$ or $h^! \mathcal{F}$ is constructible.
\end{lem}
\begin{proof}
We will prove the statement for $h^*$; the one for $h^!$ is done analogously. The only if direction is clear by definition. We will now prove the if direction.

Let $f: S\to \mathcal{Y}$ be an arbitrary morphism, where $S\in \Sch_{\FF}$. Consider the following pullback square
\[
\begin{tikzcd}
	S\times_\mathcal{Y} Y \ar{d}{h_S} \ar{r}{f'} & Y \ar{d}{h} \\
	S \ar{r}{f} & \mathcal{Y}
\end{tikzcd}
\]
where $h_S$ is a smooth surjective map between schemes. Since $h_S^* f^* \mathcal{F} \simeq f'^* h^* \mathcal{F}$ is constructible, so is $f^* \mathcal{F}$ and we are done.
\end{proof}

By definition, $\Shv_c(\mathcal{Y}) \simeq (\Shv(\mathcal{Y})^{\ren})^c$ since $\Shv_c(\mathcal{Y})$ is easily seen to be idempotent complete. Unlike the case of schemes, however, $\Shv_c(\mathcal{Y})$ is, in general, different from $\Shv(\mathcal{Y})^c$, i.e., constructibility and compactness do not necessarily coincide for a general stack $\mathcal{Y}$ (see \cref{expl:constant_sheaves_not_compact} below).

\begin{expl} \label{expl:constant_sheaves_not_compact}
Consider the constant sheaf $\QQbar_{\ell, B\Gm} \in \Shv(B\Gm)$ which is evidently constructible. It is, however, not compact in $\Shv(B\Gm)$. Indeed, to see that, it suffices to show that
\[
	\pi_* = \cHom_{\Shv(B\Gm)}(\QQbar_{\ell, B\Gm}, -): \Shv(B\Gm) \to \Shv(\pt) \simeq \Vect
\]
does not commute with filtered colimits, where $\pi: B\Gm \to \pt$ is the structure map. Consider
\[
	\cHom_{\Shv(B\Gm)}(\QQbar_{\ell, B\Gm}, \QQbar_{\ell, B\Gm}) = \pi_* \QQbar_{\ell, B\Gm} \simeq \QQbar_\ell[\beta]
\]
where $\beta$ is in cohomological degree $2$. Let
\[
	\mathcal{F} = \colim(\QQbar_{\ell, B\Gm} \xrightarrow{\beta} \QQbar_{\ell, B\Gm}[2] \xrightarrow{\beta} \cdots) \in \Shv(B\Gm).
\]

Let $h: \pt \to B\Gm$ be the canonical map. Then, by descent, $h^*$ is continuous and conservative, i.e., it does not kill any object. But we see that
\[
	h^* \mathcal{F} = \colim(\Qlbar \xrightarrow{\beta} \Qlbar[2] \xrightarrow{\beta} \cdots) \simeq 0.
\]
Hence, by conservativity of $h^*$, $\mathcal{F} = 0$, which means
\[
	\pi_* \colim(\Qlbar \xrightarrow{\beta} \Qlbar[2] \xrightarrow{\beta} \cdots) \simeq \pi_* \mathcal{F} \simeq 0.
\]

On the other hand, 
\begin{align*}
	&\alignsep\colim(\pi_* \QQbar_{\ell, B\Gm} \xrightarrow{\beta} \pi_* \QQbar_{\ell, B\Gm}[2] \xrightarrow{\beta} \cdots) 
	\simeq \Qlbar[\beta, \beta^{-1}] \\
	&\not\simeq 0 
	\simeq \pi_* \colim(\QQbar_{\ell, B\Gm} \xrightarrow{\beta} \QQbar_{\ell, B\Gm}[2] \xrightarrow{\beta} \cdots)
\end{align*}
and we are done.
\end{expl}

\subsubsection{$\ren$ and $\unren$}
The categories $\Shv(\mathcal{Y})$ and $\Shv(\mathcal{Y})^\ren$ are related by a pair of adjoint functors $\ren \dashv \unren$ which we will now turn to. For more details, see~\cite[Appendix F.5]{arinkin_stack_2022}.

First, note that compact objects in $\Shv(\mathcal{Y})$ are constructible, i.e., $\Shv(\mathcal{Y})^c \subset \Shv_c(\mathcal{Y})$. Thus, by ind-extension, we obtain a continuous fully faithful functor
\[
	\ren: \Shv(\mathcal{Y}) \to \Shv(\mathcal{Y})^\ren.
\]
Similarly, the functor $\Shv_c(\mathcal{Y}) \hookrightarrow \Shv(\mathcal{Y})$ induces a continuous functor
\[
	\unren: \Shv(\mathcal{Y})^\ren \hookrightarrow \Shv(\mathcal{Y}).
\]
These two functors form a pair of adjoints $\ren \dashv \unren$.

By definition, we have the following commutative diagrams
\[
\begin{tikzcd}
	\Shv(\mathcal{Y})^c \ar[hookrightarrow]{d} \ar[hookrightarrow]{dr} \\
	\Shv(\mathcal{Y}) \ar[hookrightarrow]{r}{\ren} & \Shv(\mathcal{Y})^\ren
\end{tikzcd} \qquad
\begin{tikzcd}
	& \Shv_c(\mathcal{Y}) \ar[hookrightarrow]{dl} \ar[hookrightarrow]{d} \\
	\Shv(\mathcal{Y}) & \Shv(\mathcal{Y})^\ren \ar{l}[swap]{\unren}
\end{tikzcd}
\]

\subsubsection{$t$-structures}
\label{subsubsec:t-structures_on_cats_of_sheaves}
The categories $\Shv(\mathcal{Y})$ and $\Shv(\mathcal{Y})^\ren$ are naturally equipped with a standard and a perverse $t$-structure, both obtained from ind-extending the corresponding $t$-structure on the full subcategory of compact objects, see~\cite[Prop. 2.13]{antieau_k-theoretic_2019} or~\cite[Lem. C.2.4.3]{lurie_spectral_2018}. Moreover, the functor $\unren$, restricted to the bounded below subcategories, induces an equivalence of categories~\cite[Appendix F.5.2]{arinkin_stack_2022}, i.e.
\[
	\unren|_{\Shv(\mathcal{Y})^{\ren, +}}: \Shv(\mathcal{Y})^{\ren, +} \xrightarrow{\simeq} \Shv(\mathcal{Y})^+.
\]

\subsubsection{Functoriality of ind-constructible sheaves on stacks}
Since pulling back along any morphism $f: \mathcal{Y} \to \mathcal{Z}$ in $\Stk_{\FF}$ preserves constructibility, we automatically get continuous functors $f^!_{\ren}$ and $f^*_{\ren}$ obtained by ind-extending the restrictions of $f^!$ and $f^*$ to the constructible full-subcategories.\footnote{Note that on stacks, $f^*$ and $f^!$ do not necessarily preserve compactness. See \cref{expl:constant_sheaves_not_compact}.} By construction, $f^*_{\ren}$ is continuous, and hence, it admits a right adjoint, denoted by $f_{*, \ren}$. Moreover, since $f^*_{\ren}$ preserves compactness by definition, $f_{*, \ren}$ is also continuous.

When $f_!$ preserves constructibility (for example, when $f$ is representable), we have a pair of adjoint functors
\[
	f_!: \Shv_c(\mathcal{Y}) \rightleftarrows \Shv_c(\mathcal{Z}): f^!.
\]
Ind-extending these, we obtain also a pair of adjoint functors
\[
	f_{!, \ren}: \Shv(\mathcal{Y})^\ren \rightleftarrows \Shv(\mathcal{Z})^\ren: f_\ren^!.
\]
Note that in this case, both $f^!_\ren$ and $f_{!, \ren}$ are continuous.

$f^*$ and $f^*_\ren$ are $t$-exact with respect to the standard $t$-structures. Hence, $f_*$ and $f_{*, \ren}$, being the right adjoints, are left $t$-exact with respect to the standard $t$-structures. In particular, $f_*$ and $f_{*, \ren}$ preserve eventual co-connectivity (the property of being bounded below). Since the $(-)^!$ functor for schemes preserve eventual co-connectivity, so is $f^!$ and $f^!_\ren$ (between stacks).

\subsubsection{Renormalized vs. usual functors}
We will now compare the various pullback/pushforward functors with their renormalized counterparts. Let $f: \mathcal{Y} \to \mathcal{Z}$ in $\Stk_{\FF}$. We have the following (a priori not necessarily commutative) diagram, where parallel arrows are adjoints
\[
\begin{tikzcd}[sep=large]
	\Shv(\mathcal{Y}) \ar[shift left=\arrdisp]{r}{\ren} \ar[shift left=\arrdisp]{d}{f_*} & \Shv(\mathcal{Y})^\ren \ar[shift left=\arrdisp]{l}{\unren} \ar[shift left=\arrdisp]{d}{f_{*, \ren}} \\
	\Shv(\mathcal{Z}) \ar[shift left=\arrdisp]{r}{\ren} \ar[shift left=\arrdisp]{u}{f^*} & \Shv(\mathcal{Z})^\ren \ar[shift left=\arrdisp]{l}{\unren} \ar[shift left=\arrdisp]{u}{f^*_\ren}
\end{tikzcd}
\]

\begin{prop}
\label{prop:interaction_*-functors_ren}
We have an equivalence of functors $f^* \circ \unren \simeq \unren \circ f^*_\ren$. Moreover, for any $\mathcal{F} \in \Shv_c(\mathcal{Y}) \subseteq \Shv(\mathcal{Y})^\ren$, we have a natural equivalence
\[
	f_*(\unren(\mathcal{F})) \simeq \unren(f_{*, \ren}(\mathcal{F})).
\]
\end{prop}
\begin{proof}
For the first statement, note that since all functors are continuous, it suffices to prove the statement when restricted to $\Shv_c(\mathcal{Z})$. But then, it follows from the definition of $f^*_\ren$.

For the second, we first define a functor
\[
	f_{\bullet, \ren}: \Shv_c(\mathcal{Y}) \to \Shv(\mathcal{Z})^+ \hookrightarrow \Shv(\mathcal{Z})^\ren,
\]
as the following composition
\[
	\Shv_c(\mathcal{Y}) \xrightarrow{f_*} \Shv(\mathcal{Z})^+ \simeq \Shv(\mathcal{Z})^{\ren, +} \hookrightarrow \Shv(\mathcal{Z})^\ren.
\]
Ind-extending, we obtain an (eponymous) continuous functor
\[
	f_{\bullet, \ren}: \Shv(\mathcal{Y})^\ren \to \Shv(\mathcal{Z})^\ren.
\]
By construction, for any $\mathcal{F} \in \Shv_c(\mathcal{Y})$, $\unren(f_{\bullet, \ren}(\mathcal{F})) \simeq f_*(\unren(\mathcal{F}))$. It thus remains to show that $f_{\bullet, \ren} \simeq f_{*, \ren}$. Or equivalently, it suffices to show that $f_{\bullet, \ren}$ is the right adjoint of $f^*_\ren$.

To do that, we produce a natural equivalence
\[
	\Hom_{\Shv(\mathcal{Z})^\ren}(\mathcal{F}, f_{\bullet, \ren}\mathcal{G}) \simeq \Hom_{\Shv(\mathcal{Y})^\ren}(f^*_\ren \mathcal{F}, \mathcal{G}), \quad \forall \mathcal{F} \in \Shv(\mathcal{Z})^\ren, \forall\mathcal{G} \in \Shv(\mathcal{Y})^\ren.
\]
In fact, since the functors involved are continuous, it suffices to assume that $\mathcal{F}$ and $\mathcal{G}$ are constructible. But now, we conclude using the following sequence of equivalences
\begin{align*}
	\Hom_{\Shv(\mathcal{Z})^\ren}(\mathcal{F}, f_{\bullet, \ren}\mathcal{G}) 
	&\simeq \Hom_{\Shv(\mathcal{Z})^+}(\mathcal{F}, f_* \mathcal{G}) \\
	&\simeq \Hom_{\Shv(\mathcal{Y})^+}(f^* \mathcal{F}, \mathcal{G}) \\
	&\simeq \Hom_{\Shv(\mathcal{Y})^{\ren, +}}(f^*_\ren \mathcal{F}, \mathcal{G}) \\
	&\simeq \Hom_{\Shv(\mathcal{Y})^\ren}(f^*_\ren \mathcal{F}, \mathcal{G})
\end{align*}
where we made use of the fact that $\unren$ induces an equivalence between the bounded below parts of the usual and renormalized categories, see \cref{subsubsec:t-structures_on_cats_of_sheaves}.
\end{proof}

\begin{rmk}
\label{rmk:interaction_*-functors_ren}
If $f^*$ preserves compactness or equivalently, if $f_*$ is continuous (for example, when $f$ is representable), then the right adjoints in the diagram above commute in general, i.e., $f_* \circ \unren \simeq \unren \circ f_{*, \ren}$. Indeed, this follows from \cref{prop:interaction_*-functors_ren}, using continuity of $f_*$.

As a result, the left adjoints also commute, i.e., $\ren\circ f^* \simeq f^*_\ren \circ \ren$. 
\end{rmk}

\subsubsection{}
Now, when $f_!$ preserves constructibility, consider the following (a priori not necessarily commutative) diagram
\[
\begin{tikzcd}[sep=large]
	\Shv(\mathcal{Y}) \ar[shift left=\arrdisp]{r}{\ren} \ar[shift right=\arrdisp]{d}[swap]{f_!} & \Shv(\mathcal{Y})^\ren \ar[shift left=\arrdisp]{l}{\unren} \ar[shift right=\arrdisp]{d}[swap]{f_{!, \ren}} \\
	\Shv(\mathcal{Z}) \ar[shift left=\arrdisp]{r}{\ren} \ar[shift right=\arrdisp]{u}[swap]{f^!} & \Shv(\mathcal{Z})^\ren \ar[shift left=\arrdisp]{l}{\unren} \ar[shift right=\arrdisp]{u}[swap]{f^!_\ren}
\end{tikzcd}
\]

\begin{lem}
\label{lem:interaction_!-functors_ren}
Consider $f$ as in \cref{prop:interaction_*-functors_ren}. Then $\unren \circ f^!_\ren \simeq f^! \circ \unren$.

Moreover, if $f_!$ preserves constructibility (so that $f_{!, \ren}$ is defined), then $f_! \circ \unren \simeq \unren \circ f_{!, \ren}$. In this case, switching to the left adjoints of the previous statement also yields an equivalence $f_{!, \ren} \circ \ren \simeq \ren \circ f_!$.
\end{lem}
\begin{proof}
For the first statement, if $\mathcal{F} \in \Shv_c(\mathcal{Z})$, the definition of $f^!_\ren$ implies that $\unren(f^!_\ren (\mathcal{F})) \simeq f^!(\unren (\mathcal{F}))$. But since all functors involved are continuous, the same conclusion applies to all $\mathcal{F} \in \Shv(\mathcal{Z})$.

The second statement is proved similarly.
\end{proof}

\begin{rmk}
\label{rmk:renormalize_vs_normal_interchangeably}
Due to \cref{prop:interaction_*-functors_ren}, when working with constructible sheaves, there is no ambiguity between $f^*, f_*$ and their renormalized versions $f^*_\ren$, $f_{*, \ren}$, respectively. By~\cref{lem:interaction_!-functors_ren}, the same statement also applies to $f^!$ vs. $f^!_\ren$ in general, and also to $f_!$ vs. $f_{!, \ren}$ when $f_!$ preserves constructibility. Therefore, in these situations, when working with constructible sheaves, we will use the two versions interchangeably.
\end{rmk}

\subsubsection{Symmetric monoidal structures}
For each $S\in \Sch_{\FF}$, $\Shv(S)$ is equipped with a symmetric monoidal structure given by tensor products of sheaves. This tensor product preserves compactness, or equivalently (since we are working with schemes), constructibility. Moreover, for any $f: S \to T$ in $\Sch_{\FF}$, $f^*$ is symmetric monoidal. 

This structure induces symmetric monoidal structures on $\Shv(\mathcal{Y})$ and $\Shv(\mathcal{Y})^\ren$ for any $\mathcal{Y} \in \Stk_{\FF}$. Moreover, for any $f: \mathcal{Y} \to \mathcal{Z}$ in $\Stk_{\FF}$, $f^*$ and $f^*_\ren$ are symmetric monoidal. In particular, $\Shv(\mathcal{Y})$ and $\Shv(\mathcal{Y})^{\ren}$ are equipped with the structures of module categories over $\Shv(\pt_{\FF}) \simeq \Shv(\pt_{\FF})^\ren$.

\subsection{Mixed sheaves}
\label{subsec:mixed_sheaves}
Working over a finite field $k_n = \Fqn$, the theory of $\ell$-adic sheaves has a ``refinement'' using weights of Frobenius. 

\subsubsection{Mixed sheaves on schemes}
Following~\cite[\S5.1.5]{beilinson_faisceaux_2018}, for a scheme $X_n \in \Sch_{k_n}$, we let $\Shv_{\mixed, c}(X_n)$ denote the category of constructible mixed complexes, which is a full subcategory of $\Shv_c(X_n)$.\footnote{The category $\Shv_{\mixed, c}(X_n)$ is written as $D^b_\mixed(X_n)$ in~\cite{beilinson_faisceaux_2018}.} We let
\[
	\Shv_{\mixed}(X_n) \defeq \Ind(\Shv_{\mixed,c}(X_n))
\]
and obtain a fully faithful embedding
\[
	\Shv_\mixed(X_n) \hookrightarrow \Shv(X_n).
\]

\subsubsection{The case of a point}
\label{subsubsec:mixed_sheaves_on_a_pt}
The category $\Shv_\mixed(\pt_n)$ consists of complexes of continuous Frobenius modules whose eigenvalues are algebraic numbers of absolute values $(q^n)^{\nicefrac{w}{2}}$. The category $\Shv_\mixed(\pt_n)$ thus breaks up into a direct sum of categories
\[
	\Shv_\mixed(\pt_n) \simeq \bigoplus_{w \in \mathbb{Z}} \Shv_\mixed(\pt_n)_w.
\]
In particular, for $V\in \Shv_\mixed(\pt_n)$, we have a natural decomposition
\[
	V \simeq \bigoplus_{w \in \mathbb{Z}} V_w
\]
where $V_w$ has \emph{naive} weight $w$, i.e., all eigenvalues have absolute values $(q^n)^{\nicefrac{w}{2}}$. 

\begin{rmk}
We caution the reader here that $V_w$ does not have weight $w$ in the sense of~\cite{beilinson_faisceaux_2018}. For example, consider $\Qlbar \in \Shv_\mixed(\pt_n)$. Then, $\Qlbar[-2]$ has weight $2$ in the sense of~\cite{beilinson_faisceaux_2018} but \emph{naive} weight $0$ in the sense above.
\end{rmk}

\subsubsection{}
The category $\Shv_\mixed(\pt_n)$ is thus equipped with a symmetric monoidal functor, given by forgetting the Frobenius module structure (but still remembering the naive weights)
\[
	\gr: \Shv_\mixed(\pt_n) \to \Vect^{\gr} = \Fun(\mathbb{Z}, \Vect). \teq\label{eq:gr_functor_for_a_point}
\]
In the notation above, $\gr(V) = \bigoplus_{w\in \mathbb{Z}} \gr(V)_w$ where $\gr(V)_w = V_w$ and, as the notation suggests, is put in graded degree $w$.

We can think of this functor as semi-simplifying the Frobenius action. Though simple, it plays a crucial role in our construction of the category of graded sheaves. In fact, we will see that $\Shv_\gr(\pt_n) \simeq \Vect^\gr$.

\subsubsection{Mixed sheaves on stacks}
\label{subsubsec:mixed_sheaves_on_stacks}
For $\mathcal{Y}_n \in \Stk_{k_n}$, we let $\Shv_\mixed(\mathcal{Y}_n)$ denote the full subcategory of $\Shv(\mathcal{Y}_n)$ consisting of those $\mathcal{F}$ whose $*$-, or $!$-, pullback along $s_n: S_n \to \mathcal{Y}_n$ for any $S_n \in \Sch_{k_n}$ is mixed. Since mixedness is preserved under pullbacks~\cite[\S5.1.6]{beilinson_faisceaux_2018}, it suffices to check mixedness on a stack by pulling back to a smooth atlas as in \cref{lem:constructiblity_stack_*_!_pullback}. See also~\cite[Prop. 2.8 and Rmk. 2.12]{sun_l-series_2012}. 

This description also implies that the category $\Shv_\mixed(\mathcal{Y}_n)$ could also be obtained by right Kan extending the corresponding theory for schemes as in \cref{eq:subsubsec_ell_adic_stacks}, where $\Shv$ is replaced by $\Shv_\mixed$.

Similarly to \cref{subsubsec:ind_constructible_on_stacks}, we let $\Shv_{\mixed, c}(\mathcal{Y}_n)$ denote the full subcategory of $\Shv_\mixed(\mathcal{Y}_n)$ consisting of constructible objects. Moreover, we let $\Shv_\mixed(\mathcal{Y}_n)^\ren \defeq \Ind(\Shv_{\mixed, c}(\mathcal{Y}_n))$ be the associated compactly generated category. By construction, we also have a fully faithful embedding $\Shv_\mixed(\mathcal{Y}_n) \hookrightarrow \Shv_\mixed(\mathcal{Y}_n)^\ren$.

\subsubsection{Functoriality}
Tensor products and the various pullback and pushforward functors preserve mixedness. Indeed, it is proved for the case of schemes in~\cite[\S5.1.6]{beilinson_faisceaux_2018} and the case of stacks follows by a spectral sequence argument, see, for example~\cite[Thm. 2.11 and Rmk. 2.12]{sun_l-series_2012}. Thus, $\Shv_\mixed(\mathcal{Y}_n)$, $\Shv_{\mixed, c}(\mathcal{Y}_n)$, and $\Shv_{\mixed}(\mathcal{Y}_n)^\ren$ are symmetric monoidal. Moreover, for any morphism $f: \mathcal{Y}_n' \to \mathcal{Y}_n$ between objects in $\Stk_{k_n}$, the various pullback and pushforward functors coming from $f$ preserve mixedness, whenever these functors are defined. In particular, using the structure map $\pi: \mathcal{Y}_n \to \pt_n = \Spec k_n = \Spec \Fqn$, we obtain symmetric monoidal functors
\[
	\pi^*: \Shv_\mixed(\pt_n) \to \Shv_\mixed(\mathcal{Y}_n) \qquad\text{and}\qquad \pi^*_{\ren}: \Shv_\mixed(\pt_n) \to \Shv_\mixed(\mathcal{Y}_n)^\ren.
\]
These functors equip $\Shv_\mixed(\mathcal{Y}_n)$ and $\Shv_\mixed(\mathcal{Y}_n)^\ren$ with structures of objects in $\ComAlg(\Mod_{\Shv_\mixed(\pt_n)})$.

\begin{rmk}
In the rest of the paper, we will work exclusively with the mixed version. All functors on sheaves are thus understood to operate on the mixed level by default.
\end{rmk}

\subsection{\texorpdfstring{$\Shv_\mixed(\pt_n)$}{Shv_m(pt_n)}-module structures}
\label{subsec:Shv_mixed(pt_n)-module_structures}
We have seen above that $\Shv_\mixed(\mathcal{Y}_n)$ and $\Shv_\mixed(\mathcal{Y}_n)^\ren$ admit $\Shv_\mixed(\pt_n)$-module structures for any $\mathcal{Y}_n \in \Stk_{k_n}$. We will now show that all the usual functors are compatible with this structure and deduce various consequences from this fact. 

We note that a large part of the materials here has been treated more systematically in~\cites{liu_enhanced_2012,liu_enhanced_2017}, which will actually be used in \cref{subsec:graded_sheaves_and_correspondences} to obtain finer homotopy coherence structures for our sheaf theory, Using~\cite{gaitsgory_study_2017}, the discussion below is a slightly different take of~\cites{liu_enhanced_2012,liu_enhanced_2017}, included for the sake of completeness, especially for readers who are not familiar with~\cites{liu_enhanced_2012,liu_enhanced_2017} or are not interested the finer homotopy coherence aspects of correspondences. 

\subsubsection{The case of schemes}
We start with the case of schemes. The arguments used here are the same as~\cite[Vol. I, Chap. 6]{gaitsgory_study_2017}. We only indicate the main ideas here. The interested reader should consult~\cite[Vol. I, Chap. 6]{gaitsgory_study_2017} for more details.

Throughout, we will make use of the following observation. 

\begin{lem}
$\Shv_\mixed(\pt_n)$ is a compactly generated rigid symmetric monoidal category.
\end{lem}
\begin{proof}
$\Shv_\mixed(\pt_n)$ is compactly generated, by definition, see \cref{subsec:ell-adic_sheaves}. Compact objects are perfect $\Qlbar$ complexes with continuous $\hat{\mathbb{Z}}$-action with integral weights. The tensor product clearly preserves compactness. Finally, compact objects are all dualizable, whose duals are simply $\Qlbar$-linear duals. 
\end{proof}

\begin{lem}
\label{lem:strict_Shv_m(pt)-mod_functor_schemes}
Let $f: Y_n \to Z_n$ where $Y_n, Z_n \in \Sch_{k_n}$. Then, the functors $f^*$, $f_*$, $f^!$, and $f_!$ are strict functors of $\Shv_\mixed(\pt_n)$-modules.
\end{lem}
\begin{proof}
By definition, the functor $f^*$ is a strict functor of $\Shv_\mixed(\pt_n)$-modules. By rigidity of $\Shv_\mixed(\pt_n)$, and by the fact that $f_*$ is continuous, we know that the right adjoint $f_*$ is also strict, see \cref{cor:lax_implies_strict_rigid}.

We turn to the pair $f_! \dashv f^!$. We will treat the case $f^!$; the claim for $f_!$ follows from \cref{cor:lax_implies_strict_rigid}. We factor $f$ as $Y_n \xrightarrow{j} \lbar{Y}_n \xrightarrow{\bar{f}} Z_n$ where $j$ is an open embedding and $\bar{f}$ is proper. Then, $f^! \simeq j^! \circ \bar{f}^! \simeq j^* \circ \bar{f}^!$. Now, $\bar{f}^!$ is a right adjoint to $\bar{f}_! = \bar{f}_*$, which is a strict functor of $\Shv_\mixed(\pt_n)$-modules. Moreover, the case of $j^! = j^*$ already follows from the above. Arguing as in~\cite[Vol. I, Chap. 5]{gaitsgory_study_2017} using~\cite[Vol. I, Chap. 7, Thm. 5.2.4]{gaitsgory_study_2017}, we see that this structure is independent of the choice of a factorization and we are done.
\end{proof}

\subsubsection{The case of stacks}
We will now move to the case of stacks.

\begin{prop}
\label{prop:strict_Shv_m(pt)-mod_functor_stacks}
Let $f: \mathcal{Y}_n \to \mathcal{Z}_n$ where $\mathcal{Y}_n, \mathcal{Z}_n \in \Stk_{k_n}$. Then, the functors $f^*$, $f_*$, $f^!$, and $f_!$ have the structures of strict functors of $\Shv_\mixed(\pt_n)$-modules. Moreover, the same statements apply to the renormalized sheaf theory.
\end{prop}
\begin{proof}
Using \cref{lem:strict_Shv_m(pt)-mod_functor_schemes}, we can define $\Shv_\mixed$ on stacks by right Kan extending along $\Sch_{k_n}^{\opp} \to \Stk_{k_n}^{\opp}$ the following functor
\[
	\Shv^?_\mixed: \Sch_{k_n}^{\opp} \to \Mod_{\Shv_\mixed(\pt_n)}
\]
instead of the one in \cref{eq:subsubsec_ell_adic_stacks}; see also the discussion in \cref{subsubsec:mixed_sheaves_on_stacks}. The resulting object agrees with the category $\Shv_\mixed(\mathcal{Y}_n)$ defined above since the forgetful functor $\Mod_{\Shv_\mixed(\pt_n)} \to \DGCatprescont$, being a right adjoint, commutes with limits. In particular, for $f: \mathcal{Y}_n \to \mathcal{Z}_n$ in $\Stk_{k_n}$, $f^*$ and $f^!$ upgrade to strict functors of $\Shv_\mixed(\pt_n)$-modules. By rigidity of $\Shv_\mixed(\pt_n)$ and \cref{cor:lax_implies_strict_rigid}, we obtain the corresponding statements for $f_*$ and $f_!$ as well.

The desired conclusion also holds for the renormalized versions $f^*_\ren$ and $f^!_\ren$ since the $\Shv_\mixed(\pt_n)$-module structures on $\Shv_\mixed(-)^\ren$ is obtained by ind-extending the $\Shv_{\mixed, c}(\pt)$-module structures on $\Shv_{\mixed, c}(-)$ (see \cref{subsubsec:module_structures_large_vs_small}) and $f^*_\ren$ and $f^!_\ren$ are constructed by ind-extending the corresponding functors on the constructible part. By the rigidity of $\Shv_\mixed(\pt_n)$ and the fact that $f_{*, \ren}$ is continuous, we obtain the statement for $f_{*, \ren}$. Finally, when $f_!$ preserves constructibility (for example, when $f$ is representable), we have a pair of adjoint functors $f_{!, \ren} \dashv f^!_\ren$, which implies the same statement for $f_{!, \ren}$ as well.
\end{proof}

\subsubsection{$\Shv_\mixed(\pt_n)$-enriched $\Hom$}
By the discussion in \cref{subsubsec:enriched_Hom}, for any $\mathcal{F}, \mathcal{G} \in \Shv_\mixed(\mathcal{Y}_n)^\ren$ (resp. $\mathcal{F}, \mathcal{G} \in \Shv_\mixed(\mathcal{Y}_n)$), we can consider the internal $\Hom$ (which is usually called the \emph{sheaf} $\Hom$ in the literature), 
\[
	\cuHom_{\Shv_{\mixed}(\mathcal{Y}_n)^\ren}(\mathcal{F}, \mathcal{G}) \in \Shv_\mixed(\mathcal{Y}_n)^\ren \qquad (\text{resp. } \cuHom_{\Shv_{\mixed}(\mathcal{Y}_n)}(\mathcal{F}, \mathcal{G}) \in \Shv_\mixed(\mathcal{Y}_n))
\]
as well as the $\Shv_\mixed(\pt_n)$-enriched $\Hom$. To keep things less cluttered, we will use the following notation to denote the $\Shv_\mixed(\pt_n)$-enriched $\Hom$
\[
	\cHom_{\Shv_\mixed(\mathcal{Y}_n)^\ren}^\mixed(\mathcal{F}, \mathcal{G}) \defeq \cuHom_{\Shv_\mixed(\mathcal{Y}_n)^\ren}^{\Shv_\mixed(\pt_n)}(\mathcal{F}, \mathcal{G}) \in \Shv_\mixed(\pt_n), \teq \label{eq:mixed_Hom}
\]
and similarly for the non-renormalized, i.e., usual, version. 

In what follows, we will mostly focus on the renormalized case due to the applications we have in mind. In most cases, however, the proof for the usual version is verbatim.

\subsubsection{}
Internal $\Hom$ and $\Shv_\mixed(\pt_n)$-enriched $\Hom$ are related in the expected way by the following lemma.

\begin{lem}
\label{lem:mixed_Hom_vs_internal_Hom}
Let $\mathcal{Y}_n \in \Stk_{k_n}$ and $\mathcal{F}, \mathcal{G} \in \Shv_\mixed(\mathcal{Y}_n)^\ren$ (resp. $\mathcal{F}, \mathcal{G} \in \Shv_\mixed(\mathcal{Y}_n)$). Then
\begin{align*}
	\pi_{*, \ren} \cuHom_{\Shv_{\mixed}(\mathcal{Y}_n)^\ren}(\mathcal{F}, \mathcal{G}) &\simeq \cHom_{\Shv_\mixed(\mathcal{Y}_n)^\ren}^\mixed(\mathcal{F}, \mathcal{G}) \\
	(\text{resp.}\qquad \pi_* \cuHom_{\Shv_{\mixed}(\mathcal{Y}_n)}(\mathcal{F}, \mathcal{G}) &\simeq \cHom_{\Shv_\mixed(\mathcal{Y}_n)}^\mixed(\mathcal{F}, \mathcal{G}))
\end{align*}
Here $\pi: \mathcal{Y}_n \to \pt_n$ denotes the structure map.
\end{lem}
\begin{proof}
We will prove the statement for the renormalized sheaf theory. The proof of the other case is the same. For any $V\in \Shv_\mixed(\pt_n)$, we have the following sequence of natural equivalences which yield the desired conclusion by the Yoneda lemma
\begin{align*}
	&\alignsep\Hom_{\Shv_\mixed(\pt_n)}(V, \pi_{*, \ren} \cuHom_{\Shv_\mixed(\mathcal{Y}_n)^\ren}(\mathcal{F}, \mathcal{G})) \\
	&\simeq \Hom_{\Shv_\mixed(\mathcal{Y}_n)^\ren}(\pi_\ren^* V \otimes \mathcal{F}, \mathcal{G}) \\
	&\simeq \Hom_{\Shv_\mixed(\mathcal{Y}_n)^\ren}(V \otimes \mathcal{F}, \mathcal{G}) \teq \label{eq:remove-^*-action} \\
	&\simeq \Hom_{\Shv_\mixed(\pt_n)}(V, \cHom_{\Shv_\mixed(\mathcal{Y}_n)^\ren}^\mixed(\mathcal{F}, \mathcal{G})).
\end{align*}
In \cref{eq:remove-^*-action}, the tensor denotes the action of $\Shv_\mixed(\pt_n)$ on $\Shv_\mixed(\mathcal{Y}_n)^\ren$, and the equivalence there is due to how the action is defined, i.e., via $\pi^*_\ren$.
\end{proof}

We also have the following expected result.

\begin{prop}
\label{prop:pullback_internal_Hom_mixed}
Let $f: \mathcal{Y}_n \to \mathcal{Z}_n$ be a representable smooth morphism in $\Stk_{k_n}$, and $\mathcal{F}, \mathcal{G} \in \Shv_\mixed(\mathcal{Z}_n)^\ren$. Then, we have a natural equivalence
\[
	f^*_{\ren} \cuHom_{\Shv_\mixed(\mathcal{Z}_n)^\ren}(\mathcal{F}, \mathcal{G}) \simeq \cuHom_{\Shv_\mixed(\mathcal{Y}_n)^\ren}(f^*_\ren \mathcal{F}, f^*_\ren \mathcal{G}).
\]
\end{prop}
\begin{proof}
Since $f$ is smooth, $f^*_\ren$ coincides with $f^!_\ren$, up to a cohomological shift. Thus, $f^*_\ren$ admits a left adjoint, $\tilde{f}_{!, \ren}$ that differs from $f_{!, \ren}$ by a cohomological shift. In particular, $\tilde{f}_{!, \ren}$ satisfies the base change theorem and projection formula.\footnote{Base change theorem for $f_{!, \ren}$ follows from the usual setting (i.e., non-renormalized): one restricts to compact objects, which live in the usual setting, and then ind-extend, since all functors are continuous. As in the usual proof, the projection formula is a consequence of the base change theorem and \Kuenneth{} formula. See also~\cite[Cor. 6.2.3]{liu_enhanced_2012}} Note that representability of $f$ is used to guarantee the existence of $f_{!, \ren}$.

Let $\mathcal{T} \in \Shv_\mixed(\mathcal{Y})^\ren$ be any test object. The desired conclusion follows from the Yoneda lemma and the equivalences below
\begin{align*}
	&\alignsep\Hom_{\Shv_\mixed(\mathcal{Y}_n)^\ren}(\mathcal{T}, f^*_{\ren} \cuHom_{\Shv_\mixed(\mathcal{Z}_n)^\ren}(\mathcal{F}, \mathcal{G})) \\
	&\simeq \Hom_{\Shv_\mixed(\mathcal{Z}_n)^\ren}(\tilde{f}_{!, \ren} \mathcal{T}, \cuHom_{\Shv_\mixed(\mathcal{Z}_n)^\ren}(\mathcal{F}, \mathcal{G})) \\
	&\simeq \Hom_{\Shv_\mixed(\mathcal{Z}_n)^\ren}(\tilde{f}_{!, \ren}(\mathcal{T})\otimes \mathcal{F}, \mathcal{G}) \\
	&\simeq \Hom_{\Shv_\mixed(\mathcal{Z}_n)^\ren}(\tilde{f}_{!, \ren}(\mathcal{T} \otimes f^*_\ren \mathcal{F}), \mathcal{G}) \\
	&\simeq \Hom_{\Shv_\mixed(\mathcal{Y}_n)^\ren}(\mathcal{T} \otimes f^*_\ren \mathcal{F}, f^*_\ren \mathcal{G}) \\
	&\simeq \Hom_{\Shv_\mixed(\mathcal{Y}_n)^\ren}(\mathcal{T}, \cuHom_{\Shv_\mixed(\mathcal{Y}_n)^\ren}(f^*_\ren \mathcal{F}, f^*_\ren \mathcal{G})).
\end{align*}
\end{proof}

\begin{rmk}
\label{rmk:pullback_internal_Hom_mixed}
The same statement as in \cref{prop:pullback_internal_Hom_mixed} also holds for the usual sheaf theory, except that we do not need to require $f$ to be representable. This is because $f_!$ always exists.
\end{rmk}

\begin{cor}
Let $\mathcal{Y}_n \in \Stk_{k_n}$ and $\mathcal{F}, \mathcal{G} \in \Shv_{\mixed, c}(\mathcal{Y}_n)$. Then, both 
\[
	\cuHom_{\Shv_\mixed(\mathcal{Y}_n)}(\mathcal{F}, \mathcal{G}) \quad\text{and}\quad \cuHom_{\Shv_\mixed(\mathcal{Y}_n)^\ren}(\mathcal{F}, \mathcal{G})
\]
are also constructible.
\end{cor}
\begin{proof}
When $\mathcal{Y}_n$ is a scheme, this is well-known as it is part of the six-functor formalism. Note that for schemes, the two sheaf theories coincide.

For a general $\mathcal{Y}_n$, by \cref{lem:constructiblity_stack_*_!_pullback}, to show that $\cuHom_{\Shv_\mixed(\mathcal{Y}_n)}(\mathcal{F}, \mathcal{G})$ is constructible, it suffices to show that its pullback to a smooth atlas is constructible. But then, \cref{rmk:pullback_internal_Hom_mixed} allows us to reduce to the scheme case and the desired conclusion follows.

Since $\cuHom_{\Shv_\mixed(\mathcal{Y}_n)}(\mathcal{F}, \mathcal{G}) \in \Shv_{\mixed, c}(\mathcal{Y}_n)$, we can view it as an element of $\Shv_\mixed(\mathcal{Y}_n)^\ren$. To show that $\cuHom_{\Shv_\mixed(\mathcal{Y}_n)^\ren}(\mathcal{F}, \mathcal{G})$ is constructible, it suffices to show that we have an equivalence
\[
	\cuHom_{\Shv_\mixed(\mathcal{Y}_n)^\ren}(\mathcal{F}, \mathcal{G}) \simeq \cuHom_{\Shv_\mixed(\mathcal{Y}_n)}(\mathcal{F}, \mathcal{G})
\]
as objects in $\Shv_\mixed(\mathcal{Y}_n)^\ren$. But now, for any $\mathcal{T} \in \Shv_{\mixed, c}(\mathcal{Y}_n)$, we have
\begin{align*}
	&\alignsep\Hom_{\Shv_\mixed(\mathcal{Y}_n)^\ren}(\mathcal{T}, \cuHom_{\Shv_\mixed(\mathcal{Y}_n)}(\mathcal{F}, \mathcal{G})) \\
	&\simeq \Hom_{\Shv_{\mixed, c}(\mathcal{Y}_n)}(\mathcal{T}, \cuHom_{\Shv_\mixed(\mathcal{Y}_n)}(\mathcal{F}, \mathcal{G})) \\
	&\simeq \Hom_{\Shv_\mixed(\mathcal{Y}_n)}(\mathcal{T}, \cuHom_{\Shv_\mixed(\mathcal{Y}_n)}(\mathcal{F}, \mathcal{G})) \\
	&\simeq \Hom_{\Shv_\mixed(\mathcal{Y}_n)}(\mathcal{T}\otimes \mathcal{F}, \mathcal{G}) \\
	&\simeq \Hom_{\Shv_{\mixed, c}(\mathcal{Y}_n)}(\mathcal{T}\otimes \mathcal{F}, \mathcal{G}) \\
	&\simeq \Hom_{\Shv_\mixed(\mathcal{Y}_n)^\ren}(\mathcal{T}\otimes \mathcal{F}, \mathcal{G}) \\
	&\simeq \Hom_{\Shv_\mixed(\mathcal{Y}_n)^\ren}(\mathcal{T}, \cuHom_{\Shv_\mixed(\mathcal{Y}_n)^\ren}(\mathcal{F}, \mathcal{G})),
\end{align*}
where we have used the fact that $\otimes$ preserves constructibility in the fourth equivalence. This implies that for all $\mathcal{T} \in \Shv_\mixed(\mathcal{Y}_n)^\ren$,
\[
	\Hom_{\Shv_\mixed(\mathcal{Y}_n)^\ren}(\mathcal{T}, \cuHom_{\Shv_\mixed(\mathcal{Y}_n)}(\mathcal{F}, \mathcal{G})) \simeq \Hom_{\Shv_\mixed(\mathcal{Y}_n)^\ren}(\mathcal{T}, \cuHom_{\Shv_\mixed(\mathcal{Y}_n)^\ren}(\mathcal{F}, \mathcal{G})).
\]
We conclude the proof by invoking the Yoneda lemma.
\end{proof}

\subsection{The construction}
\label{subsec:graded_sheaves_and_functoriality}
We now construct the category of graded sheaves.

\subsubsection{The construction} We start with the definition of the category of graded sheaves.

\begin{defn}
\label{defn:Shv_gr}
For $\mathcal{Y}_n \in \Stk_{k_n}$, we let $\Shv_{\gr}(\mathcal{Y}_n) \defeq \Shv_\mixed(\mathcal{Y}_n) \otimes_{\Shv_\mixed(\pt_n)} \Vect^{\gr}$ and $\Shv_{\gr}(\mathcal{Y}_n)^{\ren} \defeq \Shv_\mixed(\mathcal{Y}_n)^{\ren} \otimes_{\Shv_\mixed(\pt_n)} \Vect^{\gr}$ be the categories of graded sheaves and, respectively, renormalized graded sheaves on $\mathcal{Y}_n$. Here, the symmetric monoidal functor $\gr: \Shv_\mixed(\pt_n) \to \Vect^\gr$ of \cref{eq:gr_functor_for_a_point} is used to form the relative tensor.

The full subcategory of graded constructible sheaves $\Shv_{\gr, c}(\mathcal{Y}_n)$ is defined to be the full subcategory spanned by the compact objects $\Shv_{\gr}(\mathcal{Y}_n)^{\ren, c}$.
\end{defn}

Direct from the construction, we have the following observation.
\begin{lem}\label{lem:mixed_to_graded}
Let $\mathcal{Y}_n \in \Stk_{k_n}$. Then, $\Shv_\gr(\mathcal{Y}_n), \Shv_\gr(\mathcal{Y}_n)^\ren \in \ComAlg(\Vect^\gr)$, i.e., they are symmetric monoidal $\DG$-categories whose tensor products are compatible with $\Vect^\gr$-actions. Moreover, we have a natural symmetric monoidal continuous and compact preserving functor $\gr_{\mathcal{Y}_n}: \Shv_\mixed(\mathcal{Y}_n)^\ren \to \Shv_\gr(\mathcal{Y}_n)^\ren$ given by $\gr_{\mathcal{Y}_n}(\mathcal{F}) = \mathcal{F} \boxtimes \Qlbar$ and similarly for the non-renormalized version.

When $\mathcal{Y}_n$ is clear from the context or when $\mathcal{Y}_n = \pt_n$, we simply write $\gr$ in place of $\gr_{\mathcal{Y}_n}$.
\end{lem}

A couple of remarks are in order.

\begin{rmk}
Due to the applications we have in mind, we will, from now on, restrict ourselves to the renormalized case. As we already saw above, this case is equipped with better functoriality compared to the non-renormalized one. Note that many of the results below can also be proved in the same way for the non-renormalized case. And of course, in the case of schemes, there is no distinction between renormalized and non-renormalized theory. Thus, for $S_n \in \Sch_{k_n}$, we can simply write $\Shv_\gr(S_n)$ to mean either one of the two theories.
\end{rmk}

\begin{rmk}
\label{rmk:warning_graded_sheaves_structure_map}
By definition, an object $\mathcal{Y}_n \in \Stk_{k_n}$ is equipped, as part of the definition, with a structure map $\mathcal{Y}_n \to \pt_n = \Spec \Fqn$. It is this map that allows us to equip $\Shv_\mixed(\mathcal{Y}_n)^\ren$ with the structure of a $\Shv_\mixed(\pt_n)$-module used in the definition of $\Shv_\gr(\mathcal{Y}_n)^\ren$. It is important to note that the resulting category of graded sheaves $\Shv_\gr(\mathcal{Y}_n)^\ren$ is sensitive to this structure. 

For example, consider $\pt_{2} = \Spec \mathbb{F}_{q^2}$. It can be viewed as an object of either $\Stk_{k_2}$ or $\Stk_{k_1}$. We thus have two categories
\[
	\Shv_\mixed(\pt_2) \otimes_{\Shv_\mixed(\pt_1)} \Vect^\gr \quad\text{and}\quad \Shv_\mixed(\pt_2) \otimes_{\Shv_\mixed(\pt_2)} \Vect^\gr \simeq \Vect^\gr.
\]
It is easy to see that these two categories are distinct. For example, using \cref{prop:hom_graded_sheaves} below, one can see that the $\Vect^\gr$-enriched $\Hom$ of the constant sheaf in two cases are different: $\Qlbar^2$ for the first case and $\Qlbar$ for the second. This should not be surprising since roughly speaking, the first category is related to $\pt_2 \times_{\pt_1} \pt \simeq \pt \sqcup \pt$.

Because of this, when using the notation $\Shv_\gr(\mathcal{Y}_n)^\ren$, unless specified otherwise, the subscript $n$ in $\mathcal{Y}_n$ is used to indicate the fact that we are viewing $\mathcal{Y}_n$ as an object in $\Stk_{k_n}$ in defining the category of graded sheaves.
\end{rmk}

\begin{rmk}
\label{rmk:omega_graded_sheaves}
Instead of the construction $\Shv_{\gr}(\mathcal{Y}_n)^\ren = \Shv_{\mixed}(\mathcal{Y}_n)^\ren \otimes_{\Shv_\mixed(\pt_n)} \Vect^\gr$, which has the effect of remembering only the weights, we could instead construct the category of $\omega$-graded sheaves $\Shv_\omega(\mathcal{Y}_n)^\ren \defeq \Shv_\mixed(\mathcal{Y}_n)^\ren \otimes_{\Shv_\mixed(\pt_n)} \Vect^\omega$. Here, $\Vect^\omega$ is the category of chain complexes indexed by the set of algebraic integers with (complex) absolute values $(q^n)^{\nicefrac{w}{2}}$ for $w\in \mathbb{Z}$, i.e., those numbers that can appear as eigenvalues of Frobenius actions on objects in $\Shv_\mixed(\pt_n)$, see also~\cite[\S5.1.5]{beilinson_faisceaux_2018}. Moreover, the symmetric monoidal functor $\Shv_\mixed(\pt_n) \to \Vect^\omega$ forgets the Frobenius action but still remembers the eigenvalues. This construction thus literally means semi-simplifying Frobenius actions. Note also that $\Shv_\mixed(\pt_n) \to \Vect^\gr$ factors through this functor.

All results in this paper can be made to work with this variant straightforwardly. Since taking the trace of Frobenius forgets non-semisimplicity behaviors, the usual Grothendieck--Lefschetz trace formula also factors through this construction. For the applications we have in mind, however, we are primarily interested in the weight grading rather than the actual eigenvalues. Moreover, dealing with gradings somewhat simplifies the notation and the exposition. We will, as a result, not pursue this line of investigation in the current paper.
\end{rmk}

\subsubsection{Grading shifts}
\label{subsubsec:shift_of_gradings}
Recall from \cref{expl:Vectgr_rigid} that an object $V \in \Vect^{\gr}$ could be written as a direct sum $V = \oplus_n V_n$ where $V_n \in \Vect$ is placed in graded degree $n$. For any integer $k$, we use $V\lrangle{k}$ to denote a grading shift of $V$, i.e., $V\lrangle{k}_n = V_{n+k}$. For $\mathcal{F} \in \Shv_{\gr}(\mathcal{Y}_n)^\ren$, we write $\mathcal{F} \lrangle{k} \defeq \mathcal{F} \otimes \Qlbar\lrangle{k}$, where $\Qlbar\lrangle{k}$ is $\Qlbar$ placed in graded degree $-k$.
	
This notation is compatible with Tate twist. Namely, if $\mathcal{F} \in \Shv_\mixed(\mathcal{Y})^\ren$, then
\[
	\gr(\mathcal{F}(k)) \simeq \gr(\mathcal{F})\lrangle{2k}
\]
where the factor $2$ is due to the fact that the sheaf $\Qlbar(-1) \in \Shv_\mixed(\pt_n)$ has weight $2$.

From \cref{cor:compact_generators_relative_tensors}, we know that $\Shv_{\gr, c}(\mathcal{Y}_n)$ is compactly generated by $\mathcal{F} \boxtimes \Qlbar\lrangle{k}$ for $\mathcal{F} \in \Shv_{\mixed, c}(\mathcal{Y})$ and $k\in \mathbb{Z}$. But this is equivalent to $\mathcal{F}(k/2) \boxtimes \Qlbar = \gr(\mathcal{F}(k/2))$. Thus, the image of $\Shv_{\mixed, c}(\mathcal{Y}_n)$ under $\gr$ compactly generates $\Shv_\gr(\mathcal{Y}_n)^\ren$.

\subsection{Functoriality}
\label{subsec:functoriality_graded_sheaves}
We will now study functoriality of graded sheaves, which follows in a straightforward manner from that of mixed sheaves.

\subsubsection{Pull and push functors}
Let $f: \mathcal{Y}_n \to \mathcal{Z}_n$ where $\mathcal{Y}_n, \mathcal{Z}_n \in \Stk_{k_n}$. Then, recall that we have functors
\[
	f^*_\ren, f^!_\ren: \Shv_\mixed(\mathcal{Z}_n)^\ren \to \Shv_\mixed(\mathcal{Y}_n)^\ren \qquad\text{and}\qquad f_{*, \ren}: \Shv_\mixed(\mathcal{Y}_n)^\ren \to \Shv_\mixed(\mathcal{Z}_n)^\ren
\]
where $f^*_\ren \dashv f_{*, \ren}$. Moreover, when $f_!$ preserves constructibility, $f^!_\ren$ admits a left adjoint $f_{!, \ren}$.

By \cref{prop:strict_Shv_m(pt)-mod_functor_stacks}, all of these functors are strict functors of $\Shv_\mixed(\pt_n)$-modules. Thus, applying $-\otimes_{\Shv_\mixed(\pt_n)} \Vect^\gr$, we obtain functors
\[
	f^*_\ren, f^!_\ren: \Shv_\gr(\mathcal{Z}_n)^\ren \to \Shv_\gr(\mathcal{Y}_n)^\ren \qquad\text{and}\qquad f_{*, \ren}: \Shv_\gr(\mathcal{Y}_n)^\ren \to \Shv_\gr(\mathcal{Z}_n)^\ren
\]
where $f^*_\ren \dashv f_{*, \ren}$. As above, when $f_!$ (for mixed sheaves) preserves constructibility, $f^!_\ren$ admits a left adjoint $f_{!, \ren}$. Note that to avoid clutter, we do not include in the notation of the various pull and push functors any extra decoration to specify that we are operating with graded sheaves.

\subsubsection{Compatibility with functoriality for mixed/non-mixed sheaves}
\label{subsubsec:functorial_compatibility}
The various pull and push functors for graded sheaves are compatible with those on mixed/non-mixed sheaves in a precise sense. 

\subsubsection{}
Let $\mathcal{Y}_n \in \Stk_{k_n}$ and $\mathcal{Y} \in \Stk_k$ its base change to $k$. Let $h_\mathcal{Y}: \mathcal{Y} \to \mathcal{Y}_n$ and $h: \pt \to \pt_n$ denote the canonical maps, and $\pi_n: \mathcal{Y}_n \to \pt_n$ and $\pi: \mathcal{Y} \to \pt$ denote the structure maps. Then, we have the following commutative diagram in $\ComAlg(\DGCatprescont)$, where in the bottom right, $\Vect \simeq \Shv(\pt)$
\[
\begin{tikzcd}
	\Shv_\mixed(\mathcal{Y}_n)^\ren \ar{rr}{h_{\mathcal{Y}, \ren}^*} && \Shv(\mathcal{Y})^\ren \\
	\Shv_\mixed(\pt_n) \ar{u}{\pi_{n, \ren}^*} \ar[bend right=15]{rr}[swap]{h^*} \ar{r}{\gr} & \Vect^\gr \ar{r}{\oblv_\gr\simeq \oplus} & \Vect \ar{u}{\pi^*_\ren}
\end{tikzcd}
\]
Here, $\oblv_\gr$ denotes the functor of taking the direct sum of all graded components.

This implies the following lemma.
\begin{lem}
\label{lem:gr_and_oblv_gr_for_stacks}
In the situation above, the functor $\gr_{\mathcal{Y}_n}: \Shv_\mixed(\mathcal{Y}_n)^\ren \to \Shv_\gr(\mathcal{Y}_n)^\ren$ of \cref{lem:mixed_to_graded} fits into the following commutative diagram in $\ComAlg(\DGCatprescont)$ where the square on the left is a pushout square in $\ComAlg(\DGCatprescont)$
\[
\begin{tikzcd}
	\Shv_\mixed(\mathcal{Y}_n)^\ren \ar[bend left=15]{rr}{h_{\mathcal{Y}, \ren}^*} \ar{r}[swap]{\gr_{\mathcal{Y}_n}} & \Shv_\gr(\mathcal{Y}_n)^\ren \ar{r}[swap]{\oblv_{\gr, \mathcal{Y}_n}} & \Shv(\mathcal{Y})^\ren \\
	\Shv_\mixed(\pt_n) \ar{u}{\pi_{n, \ren}^*} \ar[bend right=15]{rr}[swap]{h^* \simeq \oblv_{\Frob_n}} \ar{r}{\gr} & \Vect^\gr \ar{u}{\pi_{n, \ren}^*} \ar{r}{\oblv_\gr} & \Vect \ar{u}{\pi^*_\ren}
\end{tikzcd}
\]
As before, when $\mathcal{Y}_n$ is clear from the context, we will also use $\oblv_\gr$ to denote $\oblv_{\gr, \mathcal{Y}_n}$.
\end{lem}

Compatibility between $\gr_{\mathcal{Y}_n}$ and various pull and push functors is given in the following result.

\begin{prop}
\label{thm:compatibility_functoriality_mixed_graded_non_mixed}
Let $f_n: \mathcal{Y}_n \to \mathcal{Z}_n$ where $\mathcal{Y}_n, \mathcal{Z}_n\in \Stk_{k_n}$, $f: \mathcal{Y} \to \mathcal{Z}$ its base change to $k$, $h_\mathcal{Y}: \mathcal{Y} \to \mathcal{Y}_n$ and $h_\mathcal{Z}: \mathcal{Z} \to \mathcal{Z}_n$ the canonical maps. Then we have the following commutative diagrams
\[
\begin{tikzcd}[row sep=large]
	\Shv_\mixed(\mathcal{Y}_n)^\ren \ar[bend left=15]{rr}{h_{\mathcal{Y}, \ren}^*} \ar{r}[swap]{\gr_{\mathcal{Y}_n}} & \Shv_\gr(\mathcal{Y}_n)^\ren \ar{r}[swap]{\oblv_{\gr, \mathcal{Y}_n}} & \Shv(\mathcal{Y})^\ren \\
	\Shv_\mixed(\mathcal{Z}_n)^\ren \ar[bend right=15]{rr}[swap]{h_{\mathcal{Z}, \ren}^*} \ar{r}{\gr_{\mathcal{Z}_n}} \ar{u}{f_{n, \ren}^* \ (\text{resp. } f_{n, \ren}^!)} & \Shv_\gr(\mathcal{Z}_n)^\ren \ar{r}{\oblv_{\gr, \mathcal{Z}_n}} \ar{u}{f_{n, \ren}^* \ (\text{resp. } f_{n, \ren}^!)} & \Shv(\mathcal{Z})^\ren \ar{u}{f^*_\ren \ (\text{resp. } f^!_\ren)}
\end{tikzcd}
\]
and
\[
\begin{tikzcd}[row sep=large]
	\Shv_\mixed(\mathcal{Y}_n)^\ren \ar[bend left=15]{rr}{h_{\mathcal{Y}, \ren}^*} \ar{r}[swap]{\gr_{\mathcal{Y}_n}} \ar{d}[swap]{f_{n*, \ren} \ (\text{resp. } f_{n!,\ren})} & \Shv_\gr(\mathcal{Y}_n)^\ren \ar{r}[swap]{\oblv_{\gr, \mathcal{Y}_n}} \ar{d}[swap]{f_{n*, \ren} \ (\text{resp. } f_{n!,\ren})} & \Shv(\mathcal{Y})^\ren \ar{d}[swap]{f_{*, \ren} \ (\text{resp. } f_{!,\ren})} \\
	\Shv_\mixed(\mathcal{Z}_n)^\ren \ar[bend right=15]{rr}[swap]{h_{\mathcal{Z}, \ren}^*} \ar{r}{\gr_{\mathcal{Z}_n}} & \Shv_\gr(\mathcal{Z}_n)^\ren \ar{r}{\oblv_{\gr, \mathcal{Z}_n}} & \Shv(\mathcal{Z})^\ren
\end{tikzcd}
\]
Note that here, the $(-)_{!, \ren}$ case is only defined when $f_!$ preserves constructibility.
\end{prop}
\begin{proof}
The left squares in the diagrams above commute by naturality of
\[
	\gr_\mathcal{C}: \mathcal{C} \to \mathcal{C}\otimes_{\Shv_\mixed(\pt_n)} \Vect^\gr, \qquad \mathcal{C} \in \Mod_{\Shv_\mixed(\pt_n)}.
\]
Moreover, the outer square also commutes by smooth base change and is compatible with the module structures over $\Shv_\mixed(\pt_n)$ and $\Vect$. By the factorization
\[
	\Shv_\mixed(\pt_n) \to \Vect^\gr \to \Vect,
\]
the commutativity of the squares on the right follows from the universal property of ``tensoring up.''
\end{proof}

\begin{rmk}
Similarly, for any $\mathcal{Y}_n \in \Stk_{k_n}$, we have functors
\[
	\Shv_\mixed(\mathcal{Y}_n)^\ren \xrightarrow{\gr^\omega_{\mathcal{Y}_n}} \Shv_\omega(\mathcal{Y}_n)^\ren \xrightarrow{\oblv_{\gr^\omega_{\mathcal{Y}_n}}} \Shv(\mathcal{Y})^\ren.
\]
where $\Shv_\omega(\mathcal{Y}_n)^\ren$ is described in \cref{rmk:omega_graded_sheaves}.

Beilinson communicated the following observation to us. Let $A_\mixed = \Qlbar[t] \in \ComAlg(\Shv_\mixed(\pt_n))$ where Frobenius acts by $t \mapsto t + 1$ and $A \in \ComAlg(\Shv(\pt))$ its pullback to $\Shv(\pt)$. Then, one can show that $\Shv_\omega(\pt_n) \simeq \Mod_{A_\mixed}(\Shv_\mixed(\pt_n))$, which implies, more generally, that $\Shv_\omega(\mathcal{Y}_n)^\ren \simeq \Mod_{A_\mixed}(\Shv_\mixed(\mathcal{Y}_n)^\ren)$. Under this equivalence, $\gr^\omega_{\mathcal{Y}_n}$ is identified with $\Shv_\mixed(\mathcal{Y}_n)^\ren \xrightarrow{A_\mixed \otimes -} \Mod_{A_\mixed}(\Shv_\mixed(\mathcal{Y}_n)^\ren)$ and $\oblv_{\gr^\omega_{\mathcal{Y}_n}}$ factors as follows
\[
	\Mod_{A_\mixed}(\Shv_\mixed(\mathcal{Y}_n)^\ren) \xrightarrow{\oblv_\mixed} \Mod_A(\Shv(\mathcal{Y})^\ren) \xrightarrow{- \otimes_A \Qlbar} \Shv(\mathcal{Y})^\ren,
\]
where the commutative algebra map $A \to \Qlbar$ is given by sending $t\mapsto 0$. We will not make use of this in the current paper.
\end{rmk}

\subsubsection{Base change results} Other functorial properties of graded sheaves also follow from the corresponding results for mixed sheaves in a straightforward manner.

\begin{thm}
\label{thm:base_change_graded_sheaves}
The theory of graded sheaves satisfies the projection formula and smooth and proper base change. Namely,
\begin{myenum}{(\roman*)}
	\item Projection formula: Let $f: \mathcal{Y}_n \to \mathcal{Z}_n$ be a morphism in $\Stk_{k_n}$ such that $f_{!, \ren}$ is defined (for example, when $f_!$ for mixed sheaves preserves constructibility). Then, for any $\mathcal{F} \in \Shv_\gr(\mathcal{Y}_n)^\ren, \mathcal{G} \in \Shv_\gr(\mathcal{Z}_n)^\ren$, we have a natural equivalence
	\[
		f_{!, \ren}(\mathcal{F} \otimes f^*_\ren \mathcal{G}) \simeq f_{!, \ren} \mathcal{F} \otimes \mathcal{G}.
	\]
	\item Consider the following Cartesian square in $\Stk_{k_n}$
	\[
	\begin{tikzcd}
		\mathcal{Y}'_n \ar{d}{v'} \ar{r}{h'} & \mathcal{Y}_n \ar{d}{v} \\
		\mathcal{Z}'_n \ar{r}{h} & \mathcal{Z}_n
	\end{tikzcd}
	\]
	\begin{myenum}{--}
		\item Proper base change: When $v_{!, \ren}$ and $v'_{!, \ren}$ are defined then we have a natural equivalence of functors $h^*_\ren v_{!, \ren} \simeq v'_{!, \ren} h'^*_\ren$.
		\item Smooth base change: When $h$ is smooth, we have a natural equivalence of functors $h^*_\ren v_{*, \ren} \simeq v'_{*, \ren} h'^*_\ren$.
		\item Smooth base change (variant): We have an equivalence of functors $h^!_\ren v_{*, \ren} \simeq v'_{*, \ren} h'^!_\ren$.
	\end{myenum}
\end{myenum}
\end{thm}
\begin{proof}
Since all functors involved are continuous and all categories involved are compactly generated, it suffices to check on a set of compact objects. It thus remains to show that these statements hold for constructible mixed sheaves, viewed as objects in the category of renormalized mixed sheaves.

All operations in the projection formula and proper base change statement preserve constructibility. Thus, the usual results for mixed sheaves imply those in the renormalized categories. For the last two statements involving smooth base change, we note that even though renormalized $*$-pushforward functors do not preserve constructibility, they do preserve the property of being bounded below in the usual $t$-structure (in fact, they are left $t$-exact). Similarly, the renormalized $*$-pullback functors also preserve the property of being bounded below, since they are $t$-exact. Now, the functor $\unren$ induces an equivalence between the full subcategories of bounded below objects in the usual and renormalized categories of sheaves. \cref{prop:interaction_*-functors_ren,lem:interaction_!-functors_ren} then allow us to apply the usual smooth base change theorem for mixed sheaves and the proof concludes.
\end{proof}

\subsubsection{Smooth descent}
Like $\Shv_\mixed$, the sheaf theory $\Shv_\gr$ satisfies smooth descent.

\begin{prop}
\label{prop:Shv_gr_smooth_descent}
$\Shv_\gr$ satisfies smooth descent. I.e., let $\mathcal{Z}_n \to \mathcal{Y}_n$ be a smooth morphism in $\Stk_{k_n}$ and let $\Cech^\bullet(\mathcal{Z}_n/\mathcal{Y}_n)$ be the associated $\Cech$ nerve. Then the pullback functor (either $*$ or $!$) induces an equivalence of categories
\[
	\Shv_\gr(\mathcal{Y}_n) \simeq \Tot(\Shv_\gr(\Cech^\bullet(\mathcal{Z}_n/\mathcal{Y}_n))).
\]
Similarly, $\Shv_{\gr, c}$ also satisfies smooth descent, i.e., we have
\[
	\Shv_{\gr, c}(\mathcal{Y}_n) \simeq \Tot(\Shv_{\gr, c}(\Cech^\bullet(\mathcal{Z}_n/\mathcal{Y}_n))).
\]
\end{prop}
\begin{proof}
Since $\Vect^\gr$ is compactly generated, it is also dualizable as a presentable stable $\infty$-category, by~\cite[Vol. I, Chap. 1, Prop. 7.3.2]{gaitsgory_study_2017}. Since $\Shv_\mixed(\pt_n)$ is rigid, $\Vect^\gr$ is also dualizable as an object in $\Mod_{\Shv_\mixed(\pt_n)}$, by~\cite[Vol. I, Chap. 1, Prop. 9.4.4]{gaitsgory_study_2017}. But now, the relative tensor $-\otimes_{\Shv_\mixed(\pt_n)} \Vect^\gr$ commutes with limits, by~\cite[Vol. I, Chap. 1, \S4.3.2]{gaitsgory_study_2017}. Hence, descent for mixed sheaves implies that for graded sheaves
\begin{align*}
	&\alignsep\Shv_\gr(\mathcal{Y}_n) \\
	&\simeq \Shv_\mixed(\mathcal{Y}_n) \otimes_{\Shv_\mixed(\pt_n)} \Vect^\gr \\
	&\simeq \Tot(\Shv_\mixed(\Cech^\bullet(\mathcal{Z}_n/\mathcal{Y}_n)))\otimes_{\Shv_\mixed(\pt_n)} \Vect^\gr \\
	&\simeq \Tot(\Shv_\mixed(\Cech^\bullet(\mathcal{Z}_n/\mathcal{Y}_n))\otimes_{\Shv_\mixed(\pt_n)} \Vect^\gr) \\
	&\simeq \Tot(\Shv_\gr(\Cech^\bullet(\mathcal{Z}_n/\mathcal{Y}_n))).
\end{align*}
Since pulling back preserves compactness by construction, the second statement follows from the first.
\end{proof}

\begin{rmk}
The same proof shows that $\Shv_\gr$, as a functor out of $\Stk_{k_n}^\opp$, can be obtained by right Kan extending $\Shv_\gr$ as a functor out of $\Sch_{k_n}^\opp$.
\end{rmk}

\begin{rmk}
$\Shv_\gr(-)^\ren$ does not satisfy smooth descent. In fact, $\Shv(-)^\ren$ already fails to satisfy descent. Consider the canonical morphism $h: \pt \to B\Gm$ as in \cref{expl:constant_sheaves_not_compact}. If $\Shv(-)^\ren$ satisfied smooth descent, $h^*_\ren$ would be conservative, i.e., it would not kill any object. This is, in fact, not the case.

As in \cref{expl:constant_sheaves_not_compact}, let $\pi: B\Gm \to \pt$ denote the structure morphism and 
\[
	\mathcal{F} = \colim(\Qlbar \xrightarrow{\beta} \Qlbar[2] \xrightarrow{\beta} \Qlbar[4] \xrightarrow{\beta} \cdots).
\]
Since $\pi^*_\ren$ preserves compactness, $\pi_{*, \ren}$ is continuous. \cref{prop:interaction_*-functors_ren} and the calculation in \cref{expl:constant_sheaves_not_compact} implies that
\[
	\pi_{*, \ren} \mathcal{F} \simeq \colim(\pi_{*, \ren}\Qlbar \xrightarrow{\beta} \pi_{*, \ren}\Qlbar[2] \xrightarrow{\beta} \pi_{*, \ren}\Qlbar[4] \xrightarrow{\beta} \cdots) \simeq \Qlbar[\beta, \beta^{-1}] \nsimeq 0.
\]
In particular, $\mathcal{F} \nsimeq 0$.

On the other hand, as in \cref{expl:constant_sheaves_not_compact}, $h^*_\ren(\mathcal{F}) \simeq 0$.
\end{rmk}

\subsubsection{Verdier duality}
The category of mixed sheaves is equipped with a duality functor
\[
	\DVer: \Shv_{\mixed, c}(\mathcal{Y}_n) \xrightarrow{\simeq} \Shv_{\mixed, c}(\mathcal{Y}_n)^\opp
\]
which is compatible with the linear dual on a point, which is a symmetric monoidal functor
\[
	(-)^\vee: \Shv_{\mixed, c}(\pt_n) \to \Shv_{\mixed, c}(\pt_n)^\opp.
\]
We thus get an induced functor
\[
	\Shv_{\mixed, c}(\mathcal{Y}_n) \otimes_{\Shv_{\mixed, c}(\pt_n)} \Shv_{\mixed, c}(\pt_n)^\opp \to \Shv_{\mixed, c}(\mathcal{Y}_n)^\opp,
\]
and hence, a functor
\begin{align*}
	&\alignsep\Shv_{\mixed, c}(\mathcal{Y}_n) \otimes_{\Shv_{\mixed, c}(\pt_n)} \Shv_{\mixed, c}(\pt_n)^\opp \otimes_{\Shv_{\mixed, c}(\pt_n)^\opp} \Vect^{\gr, c, \opp} \\
	&\to \Shv_{\mixed, c}(\mathcal{Y}_n)^\opp \otimes_{\Shv_{\mixed, c}(\pt_n)^\opp} \Vect^{\gr, c, \opp}
\end{align*}

Note that the RHS is equivalent to
\[
	(\Shv_{\mixed, c}(\mathcal{Y}_n) \otimes_{\Shv_{\mixed, c}(\pt_n)} \Vect^{\gr, c})^\opp \simeq \Shv_{\gr, c}(\mathcal{Y})^\opp
\]
since $(-)^\opp$ is a symmetric monoidal auto-equivalence of $\DGCatidemex$. Moreover, we have the following sequence of functors to the LHS
\begin{align*}
	&\alignsep\alignsep\Shv_{\gr, c}(\mathcal{Y}) \simeq \Shv_{\mixed, c}(\mathcal{Y}_n) \otimes_{\Shv_{\mixed, c}(\pt_n)} \Vect^{\gr, c} \\
	&\xrightarrow{\id\otimes (-)^\vee} \Shv_{\mixed, c}(\mathcal{Y}_n) \otimes_{\Shv_{\mixed, c}(\pt_n)} \Vect^{\gr, c, \opp} \\
	&\simeq \Shv_{\mixed, c}(\mathcal{Y}_n) \otimes_{\Shv_{\mixed, c}(\pt_n)} \Shv_{\mixed, c}(\pt_n)^\opp \otimes_{\Shv_{\mixed, c}(\pt_n)^\opp} \Vect^{\gr, c, \opp}
\end{align*}
We thus obtain the corresponding Verdier duality functor for graded sheaves
\[
	\DVer: \Shv_{\gr, c}(\mathcal{Y}) \xrightarrow{\simeq} \Shv_{\gr, c}(\mathcal{Y})^\opp,
\]
compatible with the duality functor $(-)^\vee$ on $\Vect^\gr$.

\subsubsection{}
Unwinding the definition, we see that for $\mathcal{F} \boxtimes V \in \Shv_{\gr, c}(\mathcal{Y})$ where $\mathcal{F} \in \Shv_{\mixed, c}(\mathcal{Y}_n)$ and $V\in \Vect^{\gr, c}$, $\DVer(\mathcal{F} \boxtimes V) \simeq \DVer(\mathcal{F}) \boxtimes V^\vee$. It is also easy to see that we get all the expected properties of the Verdier duality functor for graded sheaves from the Verdier duality functor for mixed sheaves.



\subsection{(Graded) \texorpdfstring{$\Hom$}{Hom}-spaces between graded sheaves}
\label{subsec:hom_graded_sheaves}
We will now study the $\Hom$-spaces, both $\Vect$ and $\Vect^\gr$-enriched, inside $\Shv_\gr(\mathcal{Y}_n)^\ren$. Following \cref{lem:mixed_Hom_vs_internal_Hom}, we will employ the following notation to denote the $\Vect^\gr$-enriched $\Hom$
\[
	\cHom_{\Shv_\gr(\mathcal{Y}_n)^\ren}^\gr(\mathcal{F}, \mathcal{G}) \defeq \cuHom_{\Shv_\gr(\mathcal{Y}_n)^\ren}^{\Vect^\gr}(\mathcal{F}, \mathcal{G}) \in \Vect^\gr.
\]
In particular, when $\mathcal{Y}_n = \pt_n$ and $V, W \in \Vect^\gr \simeq \Shv_\gr(\pt_n)$, we write
\[
	\cHom_{\Vect^\gr}^\gr(V, W) \defeq \cuHom_{\Vect^\gr}^{\Vect^\gr}(V, W) \in \Vect^\gr
\]
to denote the internal $\Hom$.

We start with the following counterpart of \cref{lem:mixed_Hom_vs_internal_Hom}, whose proof is identical to \cref{lem:mixed_Hom_vs_internal_Hom} and hence, will be omitted. 

\begin{lem}
Let $\mathcal{Y}_n \in \Stk_{k_n}$ and $\mathcal{F}, \mathcal{G} \in \Shv_\gr(\mathcal{Y}_n)^\ren$. Then,
\[
	\pi_{*, \ren} \cuHom_{\Shv_\gr(\mathcal{Y}_n)^\ren} (\mathcal{F}, \mathcal{G}) \simeq \cHom_{\Shv_\gr(\mathcal{Y}_n)^\ren}^\gr(\mathcal{F}, \mathcal{G}).
\]
\end{lem}

The next result concerns the computation of $\Hom$ between graded sheaves.

\begin{prop}
\label{prop:hom_graded_sheaves}
Let $\mathcal{Y}_n \in \Stk_{k_n}$, $(\mathcal{F}_c, V_c) \in \Shv_{\mixed, c}(\mathcal{Y}_n) \times \Vect^{\gr, c}$, and $(\mathcal{F}, V) \in \Shv_{\mixed}(\mathcal{Y}_n)^\ren \times \Vect^\gr$. Then,
\[
	\cHom_{\Shv_{\gr}(\mathcal{Y}_n)^\ren}^\gr(\mathcal{F}_c \boxtimes V_c, \mathcal{F}\boxtimes V) \simeq \cHom_{\Vect^\gr}^\gr(V_c, V)\otimes \gr(\cHom_{\Shv_\mixed(\mathcal{Y}_n)^\ren}^\mixed(\mathcal{F}_c, \mathcal{F})).
\]
Consequently,\footnote{Recall that for a $\DG$-category $\mathcal{C}$, $\cHom_\mathcal{C}$ denotes the $\Hom$-complex between two objects. See \cref{subsec:DG-cats}.}
\[
	\cHom_{\Shv_\gr(\mathcal{Y}_n)^\ren}(\gr(\mathcal{F}_c), \gr(\mathcal{F})) \simeq \gr(\cHom_{\Shv_\mixed(\mathcal{Y}_n)^\ren}^\mixed(\mathcal{F}_c, \mathcal{F}))_0 \simeq \cHom_{\Shv_\mixed(\mathcal{Y}_n)^\ren}^\mixed(\mathcal{F}_c, \mathcal{F})_0, \teq\label{eq:cHom_Shv_gr}
\]
where the subscript $0$ denotes the graded $0$ part, or equivalently, the \emph{naive} weight $0$ part.\footnote{See also \cref{subsubsec:mixed_sheaves_on_a_pt} for the notation.}
\end{prop}
\begin{proof}
The first part is an application of \cref{prop:induction_Hom_tensor}.

For the second part, from the definition of enriched $\Hom$, we conclude via the following sequence of equivalences
\begin{align*}
	&\alignsep\cHom_{\Shv_\gr(\mathcal{Y}_n)^\ren}(\gr(\mathcal{F}_c), \gr(\mathcal{F})) \\
	&\simeq \cHom_{\Vect^\gr}(\Qlbar, \cHom_{\Shv_\gr(\mathcal{Y}_n)^\ren}^\gr(\gr(\mathcal{F}_c), \gr(\mathcal{F}))) \\
	&\simeq \cHom_{\Vect^\gr}(\Qlbar, \gr\cHom_{\Shv_\mixed(\mathcal{Y}_n)^\ren}^\mixed(\mathcal{F}_c, \mathcal{F})) \\
	&\simeq \gr\cHom_{\Shv_\mixed(\mathcal{Y}_n)^\ren}^\mixed(\mathcal{F}_c, \mathcal{F})_0 \\
	&\simeq \cHom_{\Shv_\mixed(\mathcal{Y}_n)^\ren}^\mixed(\mathcal{F}_c, \mathcal{F})_0,
\end{align*}
where the second equivalence is due to the first part.
\end{proof}

\begin{cor}
\label{cor:graded_Hom_direct_summand_unmixed_Hom}
Let $\mathcal{Y}_n \in \Stk_{k_n}$, $\mathcal{F}_n, \mathcal{G}_n \in \Shv_\mixed(\mathcal{Y}_n)^\ren$ where $\mathcal{F}_n$ is compact, $\mathcal{Y}, \mathcal{F}, \mathcal{G}$ their base changes to $\pt = \Spec k$, and $h: \mathcal{Y} \to \mathcal{Y}_n$ the canonical map. Then, $\oblv_{\gr, \mathcal{Y}_n}$ (see \cref{lem:mixed_to_graded}) induces a map
\[
	\cHom_{\Shv_\gr(\mathcal{Y}_n)^\ren}(\gr(\mathcal{F}_n), \gr(\mathcal{G}_n)) \to \cHom_{\Shv(\mathcal{Y})^\ren}(\mathcal{F}, \mathcal{G})
\]
which realizes the former as a direct summand of the latter. In particular, if
\[
	\alpha \in \Ho^0(\cHom_{\Shv_\gr(\mathcal{Y}_n)^\ren}(\gr(\mathcal{F}_n), \gr(\mathcal{G}_n))),
\]
then it is $0$ if and only if its image in $\cHom_{\Shv(\mathcal{Y})^\ren}(\mathcal{F}, \mathcal{G})$ is $0$.
\end{cor}
\begin{proof}
The second statement is a direct consequence of the first. The first statement is itself a consequence of \cref{eq:cHom_Shv_gr} and the fact that
\[
	\cHom_{\Shv(\mathcal{Y})^\ren}(\mathcal{F}, \mathcal{G}) \simeq \oblv_{\Frob_n} \cHom_{\Shv_\mixed(\mathcal{Y})^\ren}^\mixed(\mathcal{F}_n, \mathcal{G}_n).
\]
\end{proof}

\subsection{Invariance under extensions of scalars}
\label{subsec:invariance_extensions_scalars}
The relative tensor $-\otimes_{\Shv_\mixed(\pt_n)} \Vect^\gr$ used in the definition of $\Shv_\gr(\mathcal{Y}_n)^\ren$ can be thought of as a categorical way to base change $\mathcal{Y}_n$ to $k$ (while still remembering some weight information). Thus, at least intuitively, we should expect that $\Shv_\gr(\mathcal{Y}_n)^\ren$ is invariant under base changing $\mathcal{Y}_n$ to $\mathcal{Y}_m \in \Stk_{k_m}$. This is the content of \cref{prop:gr_sheaves_invariant_base_change} below. By a spreading argument, this implies that $\Shv_\gr(-)^\ren$ is a sheaf theory on $\Stk_k$, see \cref{thm:Shv_gr_is_a_sheaf_theory_on_Stk_k}.

\subsubsection{}
Consider $f: \pt_m \to \pt_n$ where $m\geq n$, which induces a symmetric monoidal functor
\[
	f^*: \Shv_\mixed(\pt_n) \to \Shv_\mixed(\pt_m),
\]
compatible with the functor to $\Vect^\gr$. Namely, we have the following commutative diagram of symmetric monoidal categories
\[
\begin{tikzcd}[column sep=tiny]
	\Shv_\mixed(\pt_n) \ar{dr}[swap]{\gr_{\pt_n}} \ar{rr}{f^*} && \Shv_\mixed(\pt_m) \ar{dl}{\gr_{\pt_m}} \\
	& \Vect^\gr
\end{tikzcd}
\]
This induces a natural functor
\[
	\Shv_\gr(\pt_n) \to \Shv_\gr(\pt_m)
\]
that is an equivalence of categories. In fact, both sides are naturally identified with $\Vect^\gr$ and under this identification, the resulting functor is an identity functor.

\subsubsection{}
We will now consider the general case. Let $\mathcal{Y}_n \in \Stk_{k_n}$ and consider the following pullback square
\[
\begin{tikzcd}
	\mathcal{Y}_m \ar{d}{\pi_m} \ar{r}{g} \ar{d} & \mathcal{Y}_n \ar{d}{\pi_n} \\
	\pt_m \ar{r}{f} & \pt_n
\end{tikzcd}
\]
which induces the following commutative square of categories, where all functors are symmetric monoidal
\[
\begin{tikzcd}
	\Shv_\mixed(\mathcal{Y}_m)^\ren & \Shv_\mixed(\mathcal{Y}_n)^\ren \ar{l}[swap]{g^*_\ren} \\
	\Shv_\mixed(\pt_m) \ar{u}[swap]{\pi_{m, \ren}^*} & \ar{u}[swap]{\pi_{n, \ren}^*} \Shv_\mixed(\pt_n) \ar{l}[swap]{f^*}
\end{tikzcd}
\]

The diagram above induces a morphism in $\ComAlg(\Mod_{\Shv_\mixed(\pt_m)})$
\[
	\tilde{g}^*_\ren: \wtilde{\Shv}_\mixed(\mathcal{Y}_n)^\ren \defeq \Shv_\mixed(\mathcal{Y}_n)^\ren \otimes_{\Shv_\mixed(\pt_n)} \Shv_\mixed(\pt_m) \to \Shv_\mixed(\mathcal{Y}_m)^\ren. \teq\label{eq:g_tilde_^*_ren}
\]

\begin{lem}
\label{lem:tilde{g}^*_ren_fully_faithful}
$\tilde{g}^*_\ren$ is fully faithful.
\end{lem}
\begin{proof}
It suffices to show that for any $\mathcal{F}, \mathcal{G} \in \Shv_{\mixed, c}(\mathcal{Y}_n)$ and $V, W \in \Shv_{\mixed, c}(\pt_m)$, $\tilde{g}^*_\ren$ induces an equivalence
\[
	\cuHom_{\wtilde{\Shv}_{\mixed}(\mathcal{Y}_n)^\ren}^{\Shv_\mixed(\pt_m)}(\mathcal{F}\boxtimes V, \mathcal{G}\boxtimes W) \simeq \cuHom_{\Shv_\mixed(\mathcal{Y}_m)^\ren}^{\Shv_\mixed(\pt_m)}(V\otimes g^*_\ren \mathcal{F}, W \otimes g^*_\ren \mathcal{G})
\]
where the tensors on the right come from the $\Shv_\mixed(\pt_m)$-module structure of $\Shv_\mixed(\mathcal{Y}_m)^\ren$.

We have
\begin{align*}
	&\alignsep\cuHom_{\wtilde{\Shv}_{\mixed}(\mathcal{Y}_n)^\ren}^{\Shv_\mixed(\pt_m)}(\mathcal{F}\boxtimes V, \mathcal{G}\boxtimes W) \\
	&\simeq \cuHom_{\Shv_\mixed(\pt_m)}(V, W) \otimes f^*\cuHom_{\Shv_\mixed(\mathcal{Y}_n)^\ren}^{\Shv_\mixed(\pt_n)}(\mathcal{F}, \mathcal{G}) \tag{\cref{prop:induction_Hom_tensor}} \\
	&\simeq \cuHom_{\Shv_\mixed(\pt_m)}(V, W) \otimes f^* \pi_{n*, \ren} \cuHom_{\Shv_\mixed(\mathcal{Y}_n)^\ren}(\mathcal{F}, \mathcal{G}) \tag{\cref{lem:mixed_Hom_vs_internal_Hom}} \\
	&\simeq \cuHom_{\Shv_\mixed(\pt_m)}(V, W) \otimes \pi_{m*,\ren} g^*_\ren \cuHom_{\Shv_\mixed(\mathcal{Y}_n)^\ren}(\mathcal{F}, \mathcal{G}) \tag{smooth base change} \\
	&\simeq \cuHom_{\Shv_\mixed(\pt_m)}(V, W) \otimes \pi_{m*,\ren} \cuHom_{\Shv_\mixed(\mathcal{Y}_m)^\ren}(g^*_\ren \mathcal{F}, g^*_\ren \mathcal{G}) \tag{\cref{prop:pullback_internal_Hom_mixed}} \\
	&\simeq \cuHom_{\Shv_\mixed(\pt_m)}(V, W) \otimes \cuHom_{\Shv_\mixed(\mathcal{Y}_m)^\ren}^{\Shv_\mixed(\pt_m)}(g^*_\ren \mathcal{F}, g^*_\ren \mathcal{G}) \tag{\cref{lem:mixed_Hom_vs_internal_Hom}} \\
	&\simeq \cuHom_{\Shv_\mixed(\mathcal{Y}_m)^\ren}^{\Shv_\mixed(\pt_m)}(V\otimes g^*_\ren \mathcal{F}, W \otimes g^*_\ren \mathcal{G}). \tag{\cref{prop:induction_Hom_tensor}}
\end{align*}
\end{proof}

\begin{lem}
\label{lem:mixed_tensor_up_essentially_surjective}
$\tilde{g}^*_\ren$ is essentially surjective.
\end{lem}
\begin{proof}
The functor $\tilde{g}^*_\ren$ fits into the following diagram of adjoints
\[
\begin{tikzcd}
	\Shv_\mixed(\mathcal{Y}_n)^\ren \ar[shift left=\arrdisp]{r}{\can} \ar[bend left=18]{rr}{g^*_\ren} & \wtilde{\Shv}_\mixed(\mathcal{Y}_n)^\ren \ar[shift left=\arrdisp]{l}{\can^R} \ar[shift left=\arrdisp, hookrightarrow]{r}{\tilde{g}^*_\ren} & \Shv_\mixed(\mathcal{Y}_m)^\ren \ar[shift left=\arrdisp]{l}{\tilde{g}_{*, \ren}} \ar[bend left=18]{ll}{g_{*, \ren}}
\end{tikzcd} \teq\label{eq:tilde{g}^*_vs_g^*}
\]
Here, $\can$ is used to denote the canonical functor. All the functors above admit right adjoints since they are all continuous; in fact, the right adjoints are also continuous since the left adjoints preserve compactness. We have $g^*_\ren = \tilde{g}^*_\ren \circ \can$ by construction, and hence, passing to right adjoints, we also have $g_{*, \ren} = \can^R \circ \tilde{g}_{*, \ren}$.

To show that $\tilde{g}^*_\ren$ is essentially surjective, it suffices to show that any object $\mathcal{F} \in \Shv_\mixed(\mathcal{Y}_m)^\ren$ can be obtained as a colimit of objects in (the essential image of) $\wtilde{\Shv}_\mixed(\mathcal{Y}_n)^\ren$. 

Now, note that on the one hand, by base change,
\[
	g^*_\ren g_{*, \ren} \mathcal{F} \simeq \mathcal{F}^{\oplus r},
\]
for some $r \in \mathbb{Z}$. In particular, $\mathcal{F}$ is a retract of $g^*_\ren g_{*, \ren} \mathcal{F}$, and hence,
\[
	\mathcal{F} \simeq \colim(g^*_\ren g_{*, \ren} \mathcal{F} \xrightarrow{e} g^*_\ren g_{*, \ren} \mathcal{F} \xrightarrow{e} g^*_\ren g_{*, \ren} \mathcal{F} \xrightarrow{e} \cdots)
\]
where $e$ is an idempotent which projects to one $\mathcal{F}$ factor.

On the other hand, \cref{eq:tilde{g}^*_vs_g^*} implies that
\[
	g^*_\ren g_{*, \ren} \mathcal{F} \simeq \tilde{g}^*_\ren \can \can^R \tilde{g}_{*, \ren} \mathcal{F} = \tilde{g}^*_\ren \mathcal{G}
\]
where $\mathcal{G} = \can \can^R \tilde{g}_{*, \ren} \mathcal{F} \in \wtilde{\Shv}_\mixed(\mathcal{Y}_n)^\ren$. Hence, 
\[
	\mathcal{F} \simeq \colim(\tilde{g}^*_\ren \mathcal{G} \xrightarrow{e} \tilde{g}^*_\ren \mathcal{G}\xrightarrow{e} \tilde{g}^*_\ren \mathcal{G}\xrightarrow{e} \cdots),
\]
and the proof concludes.
\end{proof}

As a consequence of the two lemmas above, we obtain the following statement.
\begin{prop}
\label{prop:tilde{g}^*_ren_iso_mixed}
Let $\mathcal{Y}_n \in \Stk_{k_n}$ and $\mathcal{Y}_m$ its base change to $k_m$. Let $g: \mathcal{Y}_m \to \mathcal{Y}_n$ be the canonical map and $\tilde{g}^*_\ren$ be defined as in \cref{eq:g_tilde_^*_ren}. Then, $\tilde{g}^*_\ren$ is an equivalence of objects in $\ComAlg(\Mod_{\Shv_\mixed(\pt_m)})$
\[
	\tilde{g}^*_\ren: \Shv_\mixed(\mathcal{Y}_n)^\ren \otimes_{\Shv_\mixed(\pt_n)} \Shv_\mixed(\pt_m) \xrightarrow{\simeq} \Shv_\mixed(\mathcal{Y}_m)^\ren,
\]
i.e., it is an equivalence of symmetric monoidal categories, compatible with the $\Shv_\mixed(\pt_m)$-module structures on both sides.
\end{prop}

\subsubsection{}
From $\tilde{g}^*_\ren$, we obtain the following sequence of equivalences in $\ComAlg(\Mod_{\Vect^\gr})$
\begin{align*}
	\Shv_\gr(\mathcal{Y}_n)^\ren 
	&\simeq \Shv_\mixed(\mathcal{Y}_n)^\ren \otimes_{\Shv_\mixed(\pt_n)} \Vect^\gr \\
	&\simeq \Shv_\mixed(\mathcal{Y}_n)^\ren \otimes_{\Shv_\mixed(\pt_n)} \Shv_\mixed(\pt_m) \otimes_{\Shv_\mixed(\pt_m)} \Vect^\gr \\
	&\simeq \wtilde{\Shv}_\mixed(\mathcal{Y}_n)^\ren \otimes_{\Shv_\mixed(\pt_m)} \Vect^\gr \\
	&\xrightarrow[\simeq]{\tilde{g}^*_\ren \otimes \id_{\Vect^\gr}} \Shv_\mixed(\mathcal{Y}_m)^\ren \otimes_{\Shv_\mixed(\pt_m)} \Vect^\gr \tag{\cref{prop:tilde{g}^*_ren_iso_mixed}} \\
	&\simeq \Shv_\gr(\mathcal{Y}_m)^\ren,
\end{align*}
whose composition is denoted by $\bar{g}^*_\ren$. \cref{prop:tilde{g}^*_ren_iso_mixed} implies that $\bar{g}^*_\ren$ is also an equivalence and we have the following result.

\begin{prop} \label{prop:gr_sheaves_invariant_base_change}
Let $\mathcal{Y}_n \in \Stk_{k_n}$ and $\mathcal{Y}_m$ its base change to $k_m$. Let $g: \mathcal{Y}_m \to \mathcal{Y}_n$ be the canonical map and $\bar{g}^*_\ren$ be defined as above. Then, $\bar{g}^*_\ren$ is an equivalence of objects in $\ComAlg(\Mod_{\Vect^\gr})$
\[
	\bar{g}^*_\ren: \Shv_\gr(\mathcal{Y}_n)^\ren \xrightarrow{\simeq} \Shv_\gr(\mathcal{Y}_m)^\ren,
\]
i.e., it is an equivalence of symmetric monoidal categories, compatible with the $\Vect^\gr$-module structures on both sides.
\end{prop}

\begin{rmk}
It is important to note that in forming $\Shv_\gr(\mathcal{Y}_m)^\ren$, we are viewing $\mathcal{Y}_m$ as an object in $\Stk_{k_m}$; see also \cref{rmk:warning_graded_sheaves_structure_map}. In particular, the functor $\bar{g}^*_\ren$ defined above does not fall into the purview of \cref{subsec:functoriality_graded_sheaves} since over there, we pull and push along maps of geometric objects defined over \emph{the same} base.
\end{rmk}

\subsubsection{}
\label{subsubsec:notation_invariant_scalar_extension}
By our finiteness condition, any $\mathcal{Y} \in \Stk_k$ is a pullback of some $\mathcal{Y}_n \in \Stk_{k_n}$, see~\cite[Chap. 4]{laumon_champs_2000} or~\cite[Thm. 2.1.13]{kubrak_hodge--rham_2021}. If $\mathcal{Y}_{n_1} \in \Stk_{k_{n_1}}$ and $\mathcal{Y}'_{n_2} \in \Stk_{k_{n_2}}$ such that they both pullback to $\mathcal{Y}$ over $\pt = \Spec k$, then there exists $m \gg 0$ such that their pullbacks to $\pt_m$ agree. Moreover, any morphism $f: \mathcal{Y} \to \mathcal{Z}$ between objects in $\Stk_{k}$ is already defined over some $k_n$. 

Thus, by \cref{prop:gr_sheaves_invariant_base_change}, we can view $\Shv_\gr(-)^\ren$ as a sheaf theory on $\Stk_k$. In particular, it makes sense to talk about $\Shv_\gr(\mathcal{Y})^\ren$ for any $\mathcal{Y} \in \Stk_k$, equipped with the usual six-functor formalism described above. This can be made precise in the following theorem.

\begin{thm} \label{thm:Shv_gr_is_a_sheaf_theory_on_Stk_k}
We can attach for each $\mathcal{Y} \in \Stk_k$ the category of graded sheaves $\Shv_\gr(\mathcal{Y})^\ren \in \ComAlg(\Mod_{\Vect^\gr})$ on $\mathcal{Y}$. Moreover, for each $f: \mathcal{Y} \to \mathcal{Z}$ where $\mathcal{Y}, \mathcal{Z} \in \Stk_k$, we have the usual functors $f^*_\ren, f_{*, \ren}, f^!_\ren$, where $f^*_\ren \dashv f_{*, \ren}$. When $f_!$ (for $\Shv(-)$) preserves constructibility, $f^!_\ren$ admits a left adjoint $f_{!, \ren}$.
\end{thm}
\begin{proof}
By~\cite[Thm. 2.1.13]{kubrak_hodge--rham_2021},
\[
	\Stk_k \simeq \colim_n \Stk_{k_n} \teq \label{eq:spread_out_stacks}
\]
where functors  $\Stk_{k_n} \to \Stk_k$ and $\Stk_{k_n} \to \Stk_{k_m}$ are given by base changes. Here, the colimit is taken over the partially ordered set of finite extensions of $k_1$, which is a filtered system. Moreover, \cref{prop:gr_sheaves_invariant_base_change} furnishes us with compatible functors
\[
\begin{tikzcd}[sep=small]
	\cdots \ar{r} & \Stk_{k_{n}} \ar{dr} \ar{r} & \Stk_{k_{n'}} \ar{d} \ar{r} & \Stk_{k_{n''}} \ar{dl} \ar{r} & \cdots \\
	&& \Mod_{\Vect^\gr}
\end{tikzcd}
\]
where $\Stk_{k_n} \to \Mod_{\Vect^\gr}$ encodes the $*$-pushforward functor. \cref{eq:spread_out_stacks} then implies that we obtain a functor
\[
	\Stk_k \to \Mod_{\Vect^\gr},
\]
which encodes the $*$-pushforward functor of graded sheaves on stacks over $\pt = \Spec k$. The rest of the pull/push functors are obtained similarly.
\end{proof}

\subsubsection{Change of notation}
The construction of $\Shv_\gr(-)^\ren$ on $\Stk_k$ above implies that all the properties that we have proved earlier for $\Shv_\gr(-)^\ren$ on $\Stk_{k_n}$ automatically carry over. Thus, everywhere $\Shv_\gr(\mathcal{Y}_n)^\ren$ is used, we can replace it by $\Shv_\gr(\mathcal{Y})^\ren$. From this point onward, we will thus uniformly use notations that reflect this. For example, instead of writing $\gr_{\mathcal{Y}_n}: \Shv_\mixed(\mathcal{Y}_n)^\ren \to \Shv_\gr(\mathcal{Y}_n)^\ren$, we will write $\gr_{\mathcal{Y}_n}: \Shv_\mixed(\mathcal{Y}_n)^\ren \to \Shv_\gr(\mathcal{Y})^\ren$. 

\subsubsection{}
We end this subsection with the following useful lemma, which is a direct consequence of \cref{prop:F_M_is_conservative_when_F_is}.

\begin{lem}
\label{lem:oblv_gr_conservative}
The functor $\oblv_\gr: \Shv_{\gr}(\mathcal{Y})^{\ren} \to \Shv(\mathcal{Y})^{\ren}$ is conservative.
\end{lem}

\begin{rmk}
\label{rmk:unmixed_Hom_direct_sum_gr_Hom}
We factor $\oblv_\gr$ as follows
\[
	\Shv_\gr(\mathcal{Y})^\ren \simeq \Shv_\mixed(\mathcal{Y}_n)^\ren \otimes_{\Shv_\mixed(\pt_n)} \Vect^\gr \to \Shv_\mixed(\mathcal{Y}_n)^\ren \otimes_{\Shv_\mixed(\pt_n)} \Vect \hookrightarrow \Shv(\mathcal{Y})^\ren, \teq\label{eq:oblv_gr_conservative}
\]
where the last functor is fully faithful for the same reason as \cref{lem:tilde{g}^*_ren_fully_faithful}. Together with \cref{prop:induction_Hom_tensor}, \cref{eq:oblv_gr_conservative} implies that for $(\mathcal{F}^c, \mathcal{G}) \in \Shv_{\gr, c}(\mathcal{Y})\times \Shv_\gr(\mathcal{Y})^\ren$, we have the following expected equivalences
\begin{align*}
	&\alignsep\bigoplus_{k \in \mathbb{Z}} \cHom_{\Shv_\gr(\mathcal{Y})^\ren}(\mathcal{F}^c, \mathcal{G}\lrangle{k}) \\
	&\simeq \bigoplus_{k \in \mathbb{Z}} \cHom_{\Shv_\gr(\mathcal{Y})^\ren}^\gr(\mathcal{F}^c, \mathcal{G})_k \\
	&\eqdef \oblv_\gr(\cHom_{\Shv_\gr(\mathcal{Y})^\ren}^\gr(\mathcal{F}^c, \mathcal{G})) \\
	&\simeq \cHom_{\Shv(\mathcal{Y})^\ren}(\oblv_\gr(\mathcal{F}^c), \oblv_\gr(\mathcal{G})).
\end{align*}
\end{rmk}

\subsection{Functoriality via correspondences}
\label{subsec:graded_sheaves_and_correspondences}
We will now describe how $\Shv_\gr(-)^\ren$ can be enhanced to a functor out of the category of correspondences in $\Stk_k$. This structure encodes various base change results of \cref{thm:base_change_graded_sheaves} in a homotopy coherent way, which allows us to construct monoidal structures coming from convolutions. This is necessary since we are dealing with $\infty$-categories, where all compatibilities, such as associativity and commutativity, contain an infinite amount of data. The statements in this subsection are thus technical in nature. Fortunately, due to the way the theory of graded sheaves is set up, everything we need follows from the usual theory of $\ell$-adic sheaves and has already been established in~\cites{liu_enhanced_2017,liu_enhanced_2012}. Readers who are only interested in monoidal structures on triangulated categories can safely skip this subsection.

\subsubsection{Category of correspondences}
\label{subsubsec:correspondences}
Let $\mathcal{C}$ be any $\infty$-category. The $\infty$-category $\Corr(\mathcal{C})$ of correspondences in $\mathcal{C}$ is defined in~\cite[Vol. I, Chap. 7]{gaitsgory_study_2017}. We will quickly recall the ideas here. Roughly speaking, $\Corr(\mathcal{C})$ has the same collection of objects as $\mathcal{C}$. Moreover, given $c_1, c_2 \in \mathcal{C}$, a morphism from $c_1$ to $c_2$ is given by the following diagram in $\mathcal{C}$ 
\[
\begin{tikzcd}
	c \ar{d}{v} \ar{r}{h} & c_1 \\
	c_2
\end{tikzcd} \teq\label{eq:correspondence}
\]
where $c\in \mathcal{C}$, and where compositions are given by Cartesian squares.

More generally, let $\vertc$ and $\horiz$ be two collections of morphisms in $\mathcal{C}$ such that $\vertc$ (resp. $\horiz$) is closed under pulling back along a morphism in $\horiz$ (resp. $\vertc$). Then, we let $\Corr(\mathcal{C})_{\vertc,\horiz}$ be the (non-full) subcategory of $\Corr(\mathcal{C})$ containing the same collection of objects but morphisms are given by \cref{eq:correspondence} such that $v \in \vertc$ and $h \in \horiz$. We will also write $\Corr(\mathcal{C})_{\all,\all}$ to denote $\Corr(\mathcal{C})$ where $\all$ means \emph{all} morphisms are allowed.

When $\mathcal{C}$ is closed under finite products such that $\vertc$ and $\horiz$ are stable under these products, $\Corr(\mathcal{C})_{\vertc,\horiz}$ is equipped with a symmetric monoidal structure given by taking products.

We use $\mathcal{C}_\vertc$ and $\mathcal{C}_\horiz$ to denote the (non-full) subcategories of $\mathcal{C}$ where only morphisms in $\vertc$ and $\horiz$, respectively, are allowed. Then, we have natural functors
\[
	\mathcal{C}_\vertc \to \Corr(\mathcal{C})_{\vertc,\horiz} \qquad\text{and}\qquad \mathcal{C}^\opp_\horiz \to \Corr(\mathcal{C})_{\vertc,\horiz}
\]
which send $c_1 \to c_2$ in $\mathcal{C}_{\vertc}$ and $c_1 \to c_2$ in $\mathcal{C}_{\horiz}$ to 
\[
\begin{tikzcd}
	c_1 \ar{d} \ar[equal]{r} & c_1 \\
	c_2
\end{tikzcd}
\qquad\text{and}\qquad
\begin{tikzcd}
	c_1 \ar{r} \ar[equal]{d} & c_2 \\
	c_1
\end{tikzcd}
\]
respectively. Note that the second correspondence is a morphism from $c_2$ to $c_1$ in $\Corr(\mathcal{C})_{\vertc,\horiz}$.

These functors are symmetric monoidal with respect to the finite product monoidal structures if they are available.

\subsubsection{}
Let $\mathcal{S}$ be any $\infty$-category. Then a functor $\Phi: \Corr(\mathcal{C})_{\vertc,\horiz} \to \mathcal{S}$ induces two functors
\[
	\Phi_\vertc: \mathcal{C}_\vertc \to \mathcal{S} \qquad\text{and}\qquad \Phi_\horiz: \mathcal{C}_\horiz^\opp \to \mathcal{S}. \teq\label{eq:restriction_corr_to_vert_horiz}
\]
Moreover, for each Cartesian square in $\mathcal{C}$
\[
\begin{tikzcd}
	c' \ar{r}{h'} \ar{d}{v'} & c \ar{d}{v} \\
	d' \ar{r}{h} & d
\end{tikzcd}
\]
where $v, v' \in \vertc$ and $h, h' \in \horiz$, we are given (as part of the data of $\Phi$) an equivalence
\[
	\Phi_\vertc(v') \circ \Phi_\horiz(h') \xrightarrow{\simeq} \Phi_\horiz(h) \circ \Phi_\vertc(v),
\]
which has the same form as the usual base change results. The functor $\Phi$ encodes this base change equivalence along with all the homotopy coherence data.

\subsubsection{}
Suppose that $\mathcal{S}$ is symmetric monoidal and $\Corr(\mathcal{C})_{\vertc,\horiz}$ is equipped with a symmetric monoidal structure as above. Then, a lax symmetric monoidal functor $\Phi: \Corr(\mathcal{C})_{\vertc,\horiz} \to \mathcal{S}$ induces lax symmetric monoidal structures on $\Phi_\vertc$ and $\Phi_\horiz$. In particular, for any $c_1, c_2 \in \mathcal{C}$, we are given a morphism
\[
	\Phi(c_1) \otimes \Phi(c_2) \xrightarrow{\boxtimes} \Phi(c_1 \times c_2).
\]
Moreover, this morphism is natural in $c_1$ and $c_2$ via both $\Phi_\vertc$ and $\Phi_\horiz$. This is the shape that our sheaf theory will take.

\subsubsection{Mixed sheaves as functors out of the category of correspondences}
The theory developed in~\cite{liu_enhanced_2012,liu_enhanced_2017} provides us with a right-lax symmetric monoidal functor
\[
	\Shv_{\mixed,!}^*: \Corr(\Stk_{k_n})_{\all,\all} \to \Mod_{\Shv_\mixed(\pt_n)} \teq\label{eq:Shv_mixed^*_!_correspondence}
\]
which sends each $\mathcal{Y}_n \in \Stk_{k_n}$ to $\Shv_\mixed(\mathcal{Y}_n)$ and which encodes $*$-pullback (resp. $!$-pushforward) along all maps as well as the proper base change theorem for mixed sheaves. Note that the right-lax symmetric monoidal  structure encodes the procedure of taking box-tensor
\[
	\Shv_\mixed(\mathcal{Y}_1) \otimes \Shv_\mixed(\mathcal{Y}_2) \xrightarrow{\boxtimes} \Shv_\mixed(\mathcal{Y}_1 \times \mathcal{Y}_2), \quad\text{for all } \mathcal{Y}_1, \mathcal{Y}_2 \in \Stk_{k_n}
\]
as well its compatibility with $!$-pushforwards (\Kuenneth{} formula) and $*$-pullbacks. Here, for $\mathcal{F}_i \in \Shv_\mixed(\mathcal{Y}_i), i \in \{1, 2\}$,
\[
	\mathcal{F}_1 \boxtimes \mathcal{F}_2 \defeq p_1^* \mathcal{F}_1 \otimes p_2^* \mathcal{F}_2
\]
where $p_i: \mathcal{Y}_1 \times \mathcal{Y}_2 \to \mathcal{Y}_i$ denotes the projection onto the $i$-th factor.

\begin{rmk}
Note that the theory of $\ell$-adic sheaves developed in \cite{liu_enhanced_2017} is different from the one in \cite{gaitsgory_weils_2019,gaitsgory_atiyah-bott_2015,hemo_constructible_2021}. However, the two theories agree on the subcategories of constructible sheaves, which is what we use to construct the renormalized sheaf theory. Moreover, the general results of \cite{liu_enhanced_2012} can take as input the results of \cite{gaitsgory_weils_2019,gaitsgory_atiyah-bott_2015,hemo_constructible_2021} and yield the desired correspondence-functoriality as already done in~\cite{liu_enhanced_2017}. We thank Y. Liu for pointing this out to us. See also~\cite{chowdhury_motivic_2021} where the general theory of~\cite{liu_enhanced_2017} is used to obtain correspondence-functoriality for motivic homotopy theory.
\end{rmk}

Before continuing, we need the following definition.

\begin{defn}
A morphism $f: \mathcal{Y} \to \mathcal{Z}$ in $\Stk_{k_n}$ is said to be \emph{universally constructible} with respect to the $*$-pushforward functor (resp. $!$-pushforward functor) if for all $\mathcal{Z}' \to \mathcal{Z}$ in $\Stk_{\FF}$, $f_{\mathcal{Z}', *}$ (resp. $f_{\mathcal{Z}', !}$) preserves constructibility, where $f_{\mathcal{Z}'}: \mathcal{Z}'\times_{\mathcal{Z}} \mathcal{Y} \to \mathcal{Z}'$ is the base change of $f$ to $\mathcal{Z}'$. We also say that $f$ is $\UC_*$ (resp. $\UC_!$) in this case.

A morphism $f: \mathcal{Y} \to \mathcal{Z}$ in $\Stk_k$ is said to be $\UC_*$ (resp. $\UC_!$) if $f$ is the pullback of a $\UC_*$ (resp. $\UC_!$) morphism in $\Stk_{k_n}$.
\end{defn}

\begin{rmk}
The property of being $\UC_*$ or $\UC_!$ is stable under pullbacks and finite extensions of scalars. This allows one to make sense of the second part of the definition above.
\end{rmk}

\begin{rmk}
Representable maps are $\UC_*$ and $\UC_!$. However, there are more $\UC_*$ and $\UC_!$ morphisms than just these. For example, consider $B\Ga \to \pt_n$ where $\Ga = \Spec k_n[t]$ is the additive group scheme over $k_n$.
\end{rmk}

\subsubsection{} \label{subsubsec:Shv^*_{mixed, !}}
The functor \cref{eq:Shv_mixed^*_!_correspondence} restricts to a functor
\[
	\Shv^*_{\mixed, !}: \Corr(\Stk_{k_n})_{\UC_!, \all} \to \Mod_{\Shv_\mixed(\pt_n)}
\]
which induces the following right-lax symmetric monoidal functor by restricting to the full subcategories of constructible sheaves
\[
	\Shv_{\mixed, c, !}^*: \Corr(\Stk_{k_n})_{\UC_!, \all} \to \Mod_{\Shv_{\mixed, c}(\pt_n)}.
\]
Taking $\Ind$, which is symmetric monoidal, we obtain a right-lax symmetric monoidal functor
\[
	\Shv_{\mixed, !}^{\ren, *}: \Corr(\Stk_{k_n})_{\UC_!, \all} \to \Mod_{\Shv_\mixed(\pt_n)}, \teq\label{eq:renormalized_mixed_sheaves_corr_!-pushforward}
\]
which encodes the renormalized $!$-pushforward functors along $\UC_!$-morphisms and $*$-pullback functors along all morphisms.

\subsubsection{}
To formulate the dual we need the following result.

\begin{prop} \label{prop:Kuenneth_renormalized_pushforward}
Let $f_i: \mathcal{Y}_i \to \mathcal{Z}_i$ be morphisms in $\Stk_{k_n}$, where $i \in \{1, 2\}$. Then, $f_{i, *, \ren}$ satisfy \Kuenneth; namely, the following diagram,
\[
\begin{tikzcd}
	\Shv_\mixed(\mathcal{Y}_1)^\ren \otimes \Shv_\mixed(\mathcal{Y}_2)^\ren \ar{r}{\boxtimes} \ar{d}{f_{1, *, \ren} \otimes f_{2, *, \ren}} & \Shv_\mixed(\mathcal{Y}_1 \times \mathcal{Y}_2)^\ren \ar{d}{(f_1 \times f_2)_{*, \ren}} \\
	\Shv_\mixed(\mathcal{Z}_1)^\ren \otimes \Shv_\mixed(\mathcal{Z}_2)^\ren \ar{r}{\boxtimes} & \Shv_\mixed(\mathcal{Z}_1 \times \mathcal{Z}_2)^\ren
\end{tikzcd}
\]
which a priori commutes up to a 2-morphism given by 
\[
	f_{1, *, \ren} \mathcal{F}_1 \boxtimes f_{2, *, \ren} \mathcal{F}_2 \to (f_1 \times f_2)_{*, \ren} (\mathcal{F}_1 \boxtimes \mathcal{F}_2), \qquad\text{for all } \mathcal{F}_i \in \Shv_\mixed(\mathcal{Y}_i)^\ren, i\in \{1, 2\},
\]
is actually commutative, i.e., the canonical morphism above is an equivalence.
\end{prop}
\begin{proof}
Since all functors are continuous, it suffices to assume that $\mathcal{F}_i \in \Shv_{\mixed, c}(\mathcal{Y}_i) \subset \Shv_\mixed(\mathcal{Y}_i)^+, i \in \{1, 2\}$. By \cref{prop:interaction_*-functors_ren}, combined with the fact that renormalized *-pushforward functors are left exact and $\unren$ induces an equivalence $\Shv_\mixed(-)^{\ren, +} \xrightarrow{\simeq} \Shv_\mixed(-)^+$, we can work with non-renormalized, i.e., usual, sheaves and functors. By smooth base change and descent, it suffices to prove the statement when $\mathcal{Z}_i$'s are schemes. But then, the result is proved in~\cite[Theorem 3.4.5.1]{gaitsgory_weils_2019}.
\end{proof}

\begin{rmk} \label{rmk:Kuenneth_renormalized_!-pullback}
The analogous result for the renormalized $!$-pullback functors can be shown analogously.
\end{rmk}

Now, using~\cite[Vol. I, Chap. 12]{gaitsgory_study_2017}, we can pass \cref{eq:Shv_mixed^*_!_correspondence} to right adjoints and obtain a weakly right-lax symmetric monoidal functor
\[
	\Shv^!_{\mixed, *}: \Corr(\Stk_{k_n})_{\all, \all} \to \Mod_{\Shv_\mixed(\pt_n)}.
\]
Here, weakly right-lax symmetric monoidal means that the transformation
\[
	\Shv_\mixed(\mathcal{Y}_1) \otimes \Shv_\mixed(\mathcal{Y}_2) \xrightarrow{\boxtimes} \Shv_\mixed(\mathcal{Y}_1 \times \mathcal{Y}_2)
\]
is only natural in $\mathcal{Y}_1$ and $\mathcal{Y}_2$ in a lax way, i.e., up to a 2-morphism, see also \cite[Vol. I, Chap. 10, \S3.2.1]{gaitsgory_study_2017}. Applying the regularization procedure of \cite[Thm. 4.3.6, Lem. 4.6.1]{preygel_ind-coherent_2012} (which is a formal way to formulate the passage from $\Shv_\mixed(-)$ to $\Shv_\mixed(-)^\ren$), we obtain a weakly right-lax symmetric monoidal functor
\[
	\Shv_{\mixed, *}^{\ren, !}: \Corr(\Stk_{k_n})_{\all,\all} \to \Mod_{\Shv_\mixed(\pt_n)}. \teq \label{eq:renormalized_mixed_sheaves_corr_*-pushforward}
\]
By \cref{prop:Kuenneth_renormalized_pushforward,rmk:Kuenneth_renormalized_!-pullback}, we know that $\Shv_{\mixed, *}^{\ren, !}$ is actually a right-lax symmetric monoidal functor.

\begin{rmk}
Without using the regularization machinery of~\cite{preygel_ind-coherent_2012}, we can obtain the following right-lax symmetric monoidal functor in a more straightforward way
\[
	\Shv_{\mixed, *}^{\ren, !}: \Corr(\Stk_{k_n})_{\UC_*,\all} \to \Mod_{\Shv_\mixed(\pt_n)},
\]
following \cref{subsubsec:Shv^*_{mixed, !}}. This suffices for the applications in \cref{sec:Hecke_categories}.
\end{rmk}

\subsubsection{Graded sheaves as functors out of the category of correspondences}
Composing \cref{eq:renormalized_mixed_sheaves_corr_!-pushforward,eq:renormalized_mixed_sheaves_corr_*-pushforward} with
\[
	\Mod_{\Shv_\mixed(\pt_n)} \xrightarrow{-\otimes_{\Shv_\mixed(\pt_n)}\Vect^\gr} \Mod_{\Vect^\gr}
\]
and using \cref{subsec:invariance_extensions_scalars}, we obtain the following corresponding results for graded sheaves.

\begin{thm}
\label{thm:graded_sheaves_correspondence}
We have a right-lax symmetric monoidal functor
\[
	\Shv_{\gr, !}^{\ren, *}: \Corr(\Stk_k)_{\UC_!, \all} \to \Mod_{\Vect^\gr}
\]
which, in particular, encodes the renormalized $!$-pushforward functors along $\UC_!$-morphisms and renormalized $*$-pullback functors along all morphisms.

Similarly, we have a right-lax symmetric monoidal functor
\[
	\Shv_{\gr, *}^{\ren, !}: \Corr(\Stk_k)_{\all, \all} \to \Mod_{\Vect^\gr}
\]
which, in particular, encodes the renormalized $*$-pushforward and renormalized $!$-pullback functors along all morphisms.

We also have ``small category'' versions of the two functors above
\begin{align*}
	\Shv_{\gr, c, !}^*: \Corr(\Stk_k)_{\UC_!, \all} \to \Mod_{\Vect^{\gr, c}}, \\
	\Shv_{\gr, c, *}^!: \Corr(\Stk_k)_{\UC_*, \all} \to \Mod_{\Vect^{\gr, c}}.
\end{align*}
\end{thm}

\subsubsection{}
We also have the following variant.

\begin{thm}
We have a right-lax symmetric monoidal functor
\[
	\Shv_{\gr, *}^{\ren, *}: \Corr(\Stk_k)_{\all, \sm} \to \Mod_{\Vect^\gr} \teq\label{eq:shv_corr_all_sm}
\]
which, in particular, encodes the renormalized $*$-pushforward functors along all morphisms and renormalized $*$-pullback functors along smooth morphisms.

Similarly, we have a right-lax symmetric monoidal functor
\[
	\Shv_{\gr, *}^{\ren, *}: \Corr(\Stk_k)_{\pr, \all} \to \Mod_{\Vect^\gr} \teq\label{eq:shv_corr_pr_all}
\]
which, in particular, encodes the renormalized $*$-pushforward functors along proper morphisms and renormalized $*$-pullback functors along all morphisms.

We also have ``small category'' versions of the two functors above
\begin{align*}
	\Shv_{\gr, c, *}^*: \Corr(\Stk_k)_{\pr, \all} \to \Mod_{\Vect^{\gr, c}}, \\
	\Shv_{\gr, c, *}^*: \Corr(\Stk_k)_{\UC_*, \sm} \to \Mod_{\Vect^{\gr, c}}.
\end{align*}
\end{thm}
\begin{proof}
By the adjunction $f^*_\ren \dashv f_{*, \ren}$ and the smooth and proper base change results of \cref{thm:base_change_graded_sheaves}, the statements in the theorem, without the right-lax symmetric monoidal structures, are direct consequences of the universal property of the categories of correspondences proved in~\cite[Vol. 1, Chap. 7, \S3]{gaitsgory_study_2017} after ignoring non-invertible 2-morphisms.\footnote{Note that non-invertible 2-morphisms do not appear in our definition of the category of correspondences. In contrast, \cite{gaitsgory_study_2017} allows non-invertible morphisms. Thus, to obtain our result from theirs, we just remove non-invertible 2-morphisms from the answer.} A variant of this result but with right-lax symmetric monoidal structures and only invertible 2-morphisms are proved in~\cite[\S3.2]{liu_enhanced_2012}. Note that the latter uses the language of multi-simplicial sets rather than correspondences. However, the two formulations are the same, see~\cite[\S6.1]{liu_enhanced_2012} and~\cite[Vol. 1, Part. III, \S1.3]{gaitsgory_study_2017}.
\end{proof}

\section{Graded sheaves: weight structure and perverse \texorpdfstring{$t$}{t}-structure}
\label{sec:weight_perverse_t_structures_graded_sheaves}

In this section, we construct a perverse $t$-structure and a weight structure on the category of constructible graded sheaves $\Shv_{\gr, c}(\mathcal{Y})$ for any $\mathcal{Y} \in \Stk_{k}$. Unlike previous sections, we work exclusively with \emph{small} categories here. In particular, unless otherwise specified, $\DG$-categories appearing in this section are assumed to be small and idempotent complete; see also \cref{subsec:large_vs_small_cats}. It is mostly due to convenience since most of the available literature on weight structures, including our main source for this section~\cite{bondarko_weight_2012}, operate in this setting. It is expected that this restriction on size can be lifted~\cite[Rmk. 1.2.3]{bondarko_weight_2012}.\footnote{\cite[Lem. C.2.4.3]{lurie_spectral_2018} or~\cite[Prop. 2.13]{antieau_k-theoretic_2019} allows one to $\Ind$-extend a $t$-structure on a small stable $\infty$-category $\mathcal{C}$ to one on $\Ind (\mathcal{C})$. A similar statement but for weight structures can be found in~\cite[Thm. 4.1.2]{bondarko_morphisms_2021}.} However, we will not pursue this direction, as it is not needed for the applications we have in mind.

We will review the basics of weight structures in \cref{subsec:weight_structures_review}. This is followed by technical preparations needed to actually construct a perverse $t$-structure and a weight structure on $\Shv_{\gr, c}(\mathcal{Y})$. More specifically, in \cref{subsec:pure_graded_perv_sheaves,subsec:pure_graded_perv_sheaves_generate}, we will construct categories of pure graded perverse sheaves of a given weight and show that these categories generate the whole category of constructible graded sheaves in a precise sense. The actual construction is given in \cref{subsec:weight_and_perverse_t-str_graded_sheaves}, which follows directly from the work of Bondarko~\cite{bondarko_weight_2012}. In \cref{subsec:formal_properties_weight_Shv_gr}, various expected results regarding the interactions between the weight/$t$-structure and functoriality of graded sheaves are established. These results follow naturally from the standard ones for mixed sheaves. Finally, in \cref{subsec:mixed_geometry}, we describe connections between our construction and various notions and constructions in the mixed geometry literature.

Since we work mostly with constructible sheaves in this section, in situations where there is no difference between the usual pull/push functors and their renormalized versions, we will, for brevity's sake, omit $\ren$ from the notation, see also \cref{rmk:renormalize_vs_normal_interchangeably}.

We note that the term weight is used to refer to either the Frobenius weight structures on mixed sheaves in the sense of~\cite{beilinson_faisceaux_2018} or to weights in the sense of weight structure (on an arbitrary category) in the sense of~\cite{bondarko_weight_2010,pauksztello_compact_2008}. When we want to emphasize the former, we will use the term Frobenius weight.

\subsection{A quick review of weight structures}
\label{subsec:weight_structures_review}
We will now give a brief review of weight structures (or co-$t$-structures) as discovered independently by Pauksztello and Bondarko~\cite{pauksztello_compact_2008,bondarko_weight_2010}. For more modern treatments, using the language of stable $\infty$-categories, the readers may consult~\cites{sosnilo_theorem_2019,aoki_weight_2020,elmanto_nilpotent_2021}, which are our main sources for this subsection. Note that the indexing convention used in these papers is the reverse of that in~\cite{bondarko_weight_2010} but is the same as the one in~\cite{bondarko_weight_2012}. We will follow the convention used in~\cites{bondarko_weight_2012,sosnilo_theorem_2019,aoki_weight_2020,elmanto_nilpotent_2021}. We note that the proofs of all of the results stated here can easily be found in these papers.

\begin{defn}
\label{defn:weight_structure}
A \emph{weight structure} on a stable $\infty$-category $\mathcal{C}$ is the data of two retract-closed full additive subcategories $(\mathcal{C}^{w\leq 0}, \mathcal{C}^{w \geq 0})$ such that
\begin{myenum}{(\roman*)}
\item $\mathcal{C}^{w\geq 0}[1] \subseteq \mathcal{C}^{w\geq 0}$ and $\mathcal{C}^{w\leq 0}[-1] \subseteq \mathcal{C}^{w\leq 0}$. We write
\[
	\mathcal{C}^{w\geq n} \defeq \mathcal{C}^{w\geq 0}[n] \qquad\text{and}\qquad \mathcal{C}^{w\leq n} \defeq \mathcal{C}^{w\leq 0}[n].
\]
\item \label{enum:defn_weight_structure_negativity} If $c\in \mathcal{C}^{w\leq 0}$ and $d\in \mathcal{C}^{w\geq 1}$, then\footnote{When $\mathcal{C}$ is a $\DG$-category, this condition is equivalent to saying that $\Ho^0(\cHom_\mathcal{C}(c, d)) = 0$.}
\[
	\pi_0 \Hom_\mathcal{C}(c, d) \simeq 0.
\]
\item For any object of $c \in \mathcal{C}$, we have a cofiber sequence
\[
	w_{\leq 0} c \to c \to w_{\geq 1} c
\]
where $w_{\leq 0} c \in \mathcal{C}^{w\leq 0}$ and $w_{\geq 1}c \in \mathcal{C}^{w\geq 1}$. Such a sequence is called \emph{a} weight truncation of $c$.\footnote{Note that unlike the case of $t$-structures, weight truncations are not canonical.}
\end{myenum}

We will use the term \emph{weight category} to refer to a stable $\infty$-category equipped with a weight structure. We say that a weight structure on $\mathcal{C}$ is \emph{bounded} if
\[
	\mathcal{C} \simeq \bigcup_{n} (\mathcal{C}^{w\geq -n} \cap \mathcal{C}^{w\leq n}).
\]

We let $\Cat^w_\infty$ and $\Cat^{w, b}_\infty$ denote the $\infty$-categories of weight categories and bounded weight categories, respectively.
\end{defn}

\begin{rmk}
Note that the definition of a weight structure on a stable $\infty$-category $\mathcal{C}$ does not use any $\infty$-categorical data. Thus, we could equivalently define what it means to have a weight structure on a triangulated category and then state that a weight structure on a stable $\infty$-category $\mathcal{C}$ is that on its homotopy category $\ho\mathcal{C}$. This approach is taken, for example, in~\cite{aoki_weight_2020}, and is parallel to how~\cite{lurie_higher_2017} defines a $t$-structure.
\end{rmk}

\subsubsection{The heart of a weight structure}
\label{subsubsec:weight_heart}
Let $\mathcal{C}$ be a stable $\infty$-category equipped with a weight structure. We let $\mathcal{C}^{\weightheart} \defeq \mathcal{C}^{w\leq 0} \cap \mathcal{C}^{w\geq 0}$ denote the weight heart of the weight structure. An object $c\in \mathcal{C}$ is said to be \emph{pure of weight $n$} if $c \in \mathcal{C}^{\weightheart}[n] \simeq \mathcal{C}^{w=n}$.

Unlike the case of $t$-structures described in~\cite[Rmk. 1.2.1.12]{lurie_higher_2017}, $\mathcal{C}^{\weightheart}$ is not necessarily \emph{classical}, i.e., $\mathcal{C}^{\weightheart}$ is different from its homotopy category $\ho\mathcal{C}^{\weightheart}$. In more concrete terms, let $c, d\in \mathcal{C}^{\weightheart}$. Then, $\Hom_\mathcal{C}(c, d)$ might have non-trivial higher homotopy groups. In the setting of $\DG$-categories, being non-classical means $\cHom_\mathcal{C}(c, d)$ might have non-vanishing negative\footnote{Negative because we are using cohomological indexing convention.} cohomology groups. On the other hand, $\cHom_{\mathcal{C}}(c, d)$ can only concentrate in non-positive degrees in general. Indeed, for any $n \geq 1$, $d[n] \in \mathcal{C}^{w=n} \subseteq \mathcal{C}^{\geq 1}$, and hence, by \cref{defn:weight_structure}.\ref{enum:defn_weight_structure_negativity},
\[
	\Ho^n(\cHom_\mathcal{C}(c, d)) \simeq \Ho^0(\cHom_\mathcal{C}(c, d[n])) \simeq \pi_0 \Hom_\mathcal{C}(c, d[n]) \simeq 0.
\]
The situation is thus \emph{dual} to the case of $t$-structures: elements in the weight heart have no ``positive'' homomorphism whereas elements in the $t$-heart, which is always classical, have no ``negative'' homomorphism.

\subsubsection{}
The weight heart $\mathcal{C}^{\weightheart}$ is an additive $\infty$-category in the sense that it has all finite products and co-products and moreover, its homotopy category $\ho\mathcal{C}^{\weightheart}$ is an additive category in the usual sense. In particular, finite products and co-products in $\mathcal{C}^{\weightheart}$ coincide. See~\cite[\S2]{gepner_universality_2016} and~\cite[Appx. C.1.5]{lurie_spectral_2018} for a more detailed discussion of additive $\infty$-categories.

We use $\Cat^\add_\infty$ to denote the $\infty$-category of additive $\infty$-categories.

\begin{rmk}
The statements above regarding $\DG$-categories could have also been stated more generally for stable $\infty$-categories using $\Hom$-spectra instead of $\Vect$-enriched ones. However, this generality is not needed in the paper and we expect that the readers are more likely to be familiar with $\Vect$.
\end{rmk}

\subsubsection{From additive $\infty$-categories to weight categories}
The procedure of taking the weight heart forms a functor
\[
	(-)^{\weightheart}: \Cat^w_\infty \to \Cat^\add_\infty.
\]
It is possible to go the other direction as well using the $(-)^\fin$ construction, which we will now briefly review for the reader's convenience. We note that for our purposes, it is enough to know the existence of such a functor. For more details, see~\cite[\S2.2.7]{elmanto_nilpotent_2021}. 

Let $\mathcal{A} \in \Cat^\add_\infty$. Consider the stable $\infty$-category $\hat{\mathcal{A}} \defeq \Fun^\times(\mathcal{A}^\opp, \Sptr)$ consisting of functors that preserves finite products, i.e., those that turn finite co-products in $\mathcal{A}$ to products in the category $\Sptr$ of spectra. The Yoneda lemma furnishes a fully faithful embedding $\mathcal{A} \to \hat{\mathcal{A}}$. Let $\mathcal{A}^\fin$ be the smallest stable $\infty$-subcategory of $\hat{\mathcal{A}}$ containing the image of $\mathcal{A}$. The category $\mathcal{A}^\fin$ is equipped with a natural weight structure.

\begin{thm}[{\cite[Thm. 2.2.9]{elmanto_nilpotent_2021}}]
\label{thm:adjoint_functor_add_w}
We have an adjoint pair
\[
	(-)^\fin: \Cat^\add_\infty \rightleftarrows \Cat^{w, b}_\infty: (-)^{\weightheart}.
\]
Moreover,
\begin{myenum}{(\roman*)}
	\item the right adjoint $(-)^{\weightheart}$ is fully faithful, and
	\item the adjoint pair restricts to a pair of mutually inverse equivalences of $\infty$-categories of idempotent complete $\infty$-categories on both sides.
\end{myenum}
\end{thm}

\subsubsection{Weight complex functor}
\label{subsubsec:weight_complex_functor}
Let $\mathcal{A}$ be a classical additive category. Then, we can form a stable $\infty$-category of bounded chain complexes $\Ch^b(\mathcal{A})$, see~\cite[\S1.3.1]{lurie_higher_2017}. Its homotopy category is the homotopy category $K^b(\mathcal{A})$ of bounded chain complexes in $\mathcal{A}$. The category $\Ch^b(\mathcal{A})$ is equipped with a natural weight structure where weight truncations are given by stupid/brutal truncations of complexes,~\cite[\S1.1]{bondarko_weight_2010}. This is a proto-typical example of a weight structure. Note that for the axioms of \cref{defn:weight_structure} to hold, the differentials in $\Ch^b(\mathcal{A})$ is \emph{homological}, i.e., they \emph{decrease} the indices.

\subsubsection{}
For any $\mathcal{C} \in \Cat^{w, b}_\infty$, there exists a \emph{weight complex functor}
\[
	\wt: \mathcal{C} \to \Ch^b(\ho\mathcal{C}^{\weightheart}), \teq\label{eq:wt_complex_functor}
\]
see~\cite[Cor. 3.5]{sosnilo_theorem_2019} and~\cite[Expl. 5.1.7]{elmanto_nilpotent_2021}. This functor is the image of the natural functor $\mathcal{C}^{\weightheart} \to \ho\mathcal{C}^{\weightheart}$ under the following equivalence, coming from \cref{thm:adjoint_functor_add_w}
\[
	\Hom_{\Cat^\add_\infty}(\mathcal{C}^{\weightheart}, \ho\mathcal{C}^{\weightheart}) \simeq \Hom_{\Cat^{w, b}_\infty}(\mathcal{C}, \Ch^b(\ho\mathcal{C}^{\weightheart})).
\]

In particular, when $\mathcal{C}^{\weightheart}$ is classical, i.e., $\mathcal{C}^{\weightheart} \simeq \ho\mathcal{C}^{\weightheart}$, then the weight complex functor induces an equivalence of categories $\mathcal{C} \simeq \Ch^b(\ho\mathcal{C}^{\weightheart})$.

\begin{rmk}
\label{rmk:weight_complex_functor_explicit}
The construction of the weight complex functor given in~\cite[Cor. 3.5]{sosnilo_theorem_2019} is an $\infty$-categorical enhancement of the one given in~\cite[\S3]{bondarko_weight_2010}, which has a more concrete description. Given an object $c\in \mathcal{C}$ where $\mathcal{C}$ is a bounded weight category, by \cref{defn:weight_structure}.(iii), there exists a (non-canonical) finite filtration $c_\bullet$ of $c$, such that the $i$-th associated graded piece $\assgr_i (c_\bullet)$ is pure of weight $i$. Thus, $\assgr_i(c_\bullet)[-i]$ is pure of weight $0$, i.e., $\assgr_i(c_\bullet)[-i] \in \mathcal{C}^{\weightheart}$. A standard construction in homological algebra,~\cite[Defn. 1.2.2.2 and Rmk. 1.2.2.3]{lurie_higher_2017}, gives us a chain complex
\[
	\cdots \to \assgr_{i+1}(c_\bullet)[-(i+1)] \to \assgr_i(c_\bullet)[-i] \to \assgr_{i-1}(c_\bullet)[-(i-1)] \to \cdots
\]
The content of the weight complex functor is that this complex (up to homotopy) is canonical in $c$ even though the weight filtration is \emph{not} canonical. 

We note that this construction takes the same shape as the construction of the chromatographic complex in~\cite[\S3.5]{webster_geometric_2017}, which is not known to be a functor. As we will see, their construction can be realized as the composition of the canonical functor $\Shv_{\mixed, c}(\mathcal{Y}_n) \to \Shv_{\gr, c}(\mathcal{Y})$ and the weight complex functor on $\Shv_{\gr, c}(\mathcal{Y})$. In particular, this shows that the procedure of taking chromatographic complex is a functor.
\end{rmk}

\subsubsection{Compatibility with monoidal structures}
Let $\mathcal{C}$ be a bounded weight category, equipped with a compatible monoidal structure in the sense that $\mathcal{C}^{w\leq 0}$ and $\mathcal{C}^{w\geq 0}$ are closed under tensor products. Here, the monoidal structure can be symmetric or not. In fact, more generally, $\mathcal{C}$ can be $\En_n$-monoidal in the sense of~\cite[\S5.4]{lurie_higher_2017}, with $\En_1$-monoidal, $\En_2$-monoidal, and $\En_\infty$-monoidal being the usual, braided, and symmetric monoidal structures, respectively. Such an $\En_n$-monoidal structure gives rise to an $\En_n$-monoidal structure on $\mathcal{C}^{\weightheart}$ and hence, also on $(\mathcal{C}^{\weightheart})^\fin$ and $(\ho\mathcal{C}^{\weightheart})^\fin \simeq \Ch^b(\ho\mathcal{C}^{\weightheart})$. We have the following result.

\begin{thm}[\cite{aoki_weight_2020}]
\label{thm:aoki_monoidal_wt_complex_functor}
Let $\mathcal{C}$ be a bounded weight category, equipped with an $\En_n$-monoidal structure such that $\mathcal{C}^{w\leq 0}$ and $\mathcal{C}^{w\geq 0}$ are preserved under tensor products. Then, $\mathcal{C}^{\weightheart}$, and hence $(\mathcal{C}^{\weightheart})^\fin$ and $(\ho\mathcal{C}^{\weightheart})^\fin$ are equipped with natural $\En_n$-monoidal structures. Moreover, the following natural functors
\[
	\mathcal{C} \to (\mathcal{C}^{\weightheart})^\fin \to (\ho\mathcal{C}^{\weightheart})^\fin \simeq \Ch^b(\ho\mathcal{C}^{\weightheart})
\]
are $\En_n$-monoidal. In particular, the weight complex functor is $\En_n$-monoidal. 
\end{thm}

\begin{rmk}
The main result of~\cite{aoki_weight_2020} is stated only for symmetric monoidal categories. However, the same proof works more generally. Indeed, the main tool used is~\cite{nikolaus_stable_2016}, which works more generally; see also~\cite[Rmk. 6.11]{nikolaus_stable_2016}.
\end{rmk}

\subsubsection{The case of $\Vect^{\gr, c}$}
As mentioned earlier, our goal is to equip $\Shv_{\gr, c}(\mathcal{Y})$ with a weight structure and a perverse $t$-structure for any $\mathcal{Y} \in \Stk_k$. The remainder of this section will be devoted to this goal. As a warm-up exercise, we will now equip $\Vect^{\gr, c} \simeq \Shv_{\gr, c}(\pt)$ with a bounded weight structure and a $t$-structure. 

The $t$-structure is just the usual one, obtained from the usual $t$-structure on $\Vect^c$. The weight structure is defined as follows: $\Vect^{\gr, c, w\geq 0}$ (resp. $\Vect^{\gr, c, w\leq 0}$) consists of graded perfect chain complexes $A = (A_i)$ where the $i$-th graded piece $A_i$ concentrates in cohomological degrees $\leq i$ (resp. $\geq i$). It is easy to check that the pair $\Vect^{\gr, c, w\leq 0}$ and $\Vect^{\gr, c, w\geq 0}$ satisfies the conditions given in \cref{defn:weight_structure}.

From the description above, we see that $\Vect^{\gr, c, \weightheart} \simeq \bigoplus_i \Vect^{c, \theart}[-i]$. Namely, given any $A \in \Vect^{\gr, c, \weightheart}$, we have an equivalence $A \simeq \bigoplus_i \Ho^i(A_i)\lrangle{-i}[-i]$. This is the ``baby case'' of the decomposition theorem for pure graded sheaves we will establish in \cref{subsec:weight_and_perverse_t-str_graded_sheaves}.

From the description above, it is clear that $\Vect^{\gr, c, \weightheart}$ is classical. Thus, the weight complex functor provides an equivalence of symmetric monoidal categories 
\[
	\Vect^{\gr, c} \simeq \Ch^b(\Vect^{\gr, c, \weightheart}) \simeq \bigoplus_i \Ch^b(\Vect^{c, \theart}[-i]) \simeq \bigoplus_i \Ch^b(\Vect^{c, \theart}),
\]
which is evident. Note that the last equivalence is just a re-indexing.

\subsection{Pure graded perverse sheaves}
\label{subsec:pure_graded_perv_sheaves}
This subsection and the next make the necessary preparation to apply the results of~\cite{bondarko_weight_2012} to construct perverse $t$-structure and a weight structure on the category $\Shv_{\gr, c}(\mathcal{Y})$. Throughout, we let $\mathcal{Y}_n \in \Stk_{k_n}$ whose base changes to $k$ and $k_m$ are $\mathcal{Y}$ and $\mathcal{Y}_m$, respectively, for any $m\in n\mathbb{Z}_{>0}$. 

More specifically, we will now construct various categories $\Perv_{\gr, c}(\mathcal{Y}_m)^{w=\nu}$ of pure graded perverse sheaves of a fixed weight $w=\nu$ and study their formal properties. In the next subsection, we will show that these categories together generate $\Shv_{\gr, c}(\mathcal{Y})$ in the appropriate sense.

\subsubsection{}
For any $\nu \in \mathbb{Z}$, we define the category of graded perverse sheaves of pure weight $\nu$ to be
\[
	\Perv_{\gr, c}(\mathcal{Y}_m)^{w=\nu} \defeq \Idem(\Im(\Perv_{\mixed, c}(\mathcal{Y}_m)^{w=\nu} \xrightarrow{\gr_{\mathcal{Y}_m}} \Shv_{\gr, c}(\mathcal{Y}))) \teq\label{eq:defn_pure_gr_perv_sheaves}
\]
where $\Idem$ is the functor of taking idempotent completion and $\Perv_{\mixed, c}(\mathcal{Y}_m)^{w=\nu}$ is the full subcategory of the category of constructible mixed perverse sheaves $\Perv_{\mixed, c}(\mathcal{Y}_m)$ of (Frobenius) weight $w=\nu$, in the sense of~\cites{beilinson_faisceaux_2018,laszlo_perverse_2009}.\footnote{What we write $\Perv_{\mixed, c}(\mathcal{Y}_m)$ is more usually written as $\Perv_{\mixed}(\mathcal{Y}_m)$.} Since $\Shv_{\gr, c}(\mathcal{Y})$ is idempotent complete, $\Perv_{\gr, c}(\mathcal{Y}_m)^{w=\nu}$ is naturally a full subcategory of $\Shv_{\gr, c}(\mathcal{Y})$. The main goal of this subsection is to have a more explicit description of $\Perv_{\gr, c}(\mathcal{Y}_m)^{w=\nu}$ for any $m$ and $\nu$.

\begin{rmk}
Directly from the construction, we see that the endofunctor $\lrangle{k}: \Shv_\gr(\mathcal{Y})^\ren \to \Shv_\gr(\mathcal{Y})^\ren$ described in \cref{subsubsec:shift_of_gradings} induces an equivalence of categories
\[
	\Perv_{\gr, c}(\mathcal{Y}_m)^{w=\nu} \to \Perv_{\gr, c}(\mathcal{Y}_m)^{w=\nu-k}.
\]
\end{rmk}

\begin{rmk}
\label{rmk:preview_Perv_gr_c}
Note that the term ``pure graded perverse sheaves'' does not a priori have a meaning. However, in \cref{subsec:weight_and_perverse_t-str_graded_sheaves}, we will construct a perverse $t$-structure and a weight structure on $\Shv_{\gr, c}(\mathcal{Y})$ that make sense of this term. Namely, the $t$-heart, $\Perv_{\gr, c}(\mathcal{Y})$, is obtained by assembling $\Perv_{\gr, c}(\mathcal{Y}_m)^{w=\nu}$ for various $\nu$ together. Moreover, $\Perv_{\gr, c}(\mathcal{Y}_m)^{w=\nu}$ consists precisely of objects in $\Perv_{\gr, c}(\mathcal{Y})$ of weight $\nu$ in the new weight structure.
\end{rmk}

\subsubsection{Invariance under of extensions of scalars}
As \cref{rmk:preview_Perv_gr_c} suggests, $\Perv_{\gr, c}(\mathcal{Y}_m)^{w=\nu}$ is also independent of $m$. This is indeed the case.

\begin{lem} \label{lem:gr_pure_perv_sheaves_invariant_base_change}
Let $\mathcal{Y}$, $\mathcal{Y}_n$, and $\mathcal{Y}_m$ as above. Then, for any $\nu\in \mathbb{Z}$, $\Perv_{\gr, c}(\mathcal{Y}_n)^{w=\nu}$ and $\Perv_{\gr, c}(\mathcal{Y}_m)^{w=\nu}$ coincide as full subcategories of $\Shv_{\gr, c}(\mathcal{Y})$.
\end{lem}
\begin{proof}
The proof follows essentially the same strategy as that of \cref{prop:tilde{g}^*_ren_iso_mixed}. We have the following factorization where the horizontal arrow is given by pulling back along $\mathcal{Y}_m \xrightarrow{g} \mathcal{Y}_n$, which preserves perversity and Frobenius weights since it is a finite \etale{} map\footnote{Since we are working with constructible sheaves only, $g^*$ and $g^*_\ren$ are the same.}
\[
\begin{tikzcd}
	\Perv_{\mixed, c}(\mathcal{Y}_n)^{w=\nu} \ar{r}{g^*} \ar{dr}[swap]{\gr_{\mathcal{Y}_n}} & \Perv_{\mixed, c}(\mathcal{Y}_m)^{w=\nu} \ar{d}{\gr_{\mathcal{Y}_m}} \\
	& \Shv_{\gr, c}(\mathcal{Y})
\end{tikzcd} \teq\label{eq:factorization_gr_perv}
\]
which implies the inclusion $\Perv_{\gr, c}(\mathcal{Y}_n)^{w=\nu} \subseteq \Perv_{\gr, c}(\mathcal{Y}_m)^{w=\nu}$ as full subcategories of $\Shv_{\gr, c}(\mathcal{Y})$.

To show that the two coincide, it suffices to show that $\gr_{\mathcal{Y}_m}(\mathcal{F}) \in \Perv_{\gr, c}(\mathcal{Y}_n)^{w=\nu}$ for any $\mathcal{F} \in \Perv_{\mixed, c}(\mathcal{Y}_m)^{w=\nu}$. As in \cref{lem:mixed_tensor_up_essentially_surjective}, $\gr_{\mathcal{Y}_m}(\mathcal{F})$ is a direct summand of $\gr_{\mathcal{Y}_n}(g_* \mathcal{F})$ where $g_* \mathcal{F} \in \Perv_{\mixed, c}(\mathcal{Y}_n)^{w=\nu}$ since $g$ is finite \etale. We are thus done since $\Perv_{\gr, c}(\mathcal{Y}_n)^{w=\nu}$ is idempotent complete, by definition.
\end{proof}

As in the case of graded sheaves explained in \cref{subsubsec:notation_invariant_scalar_extension}, for any $\mathcal{Y} \in \Stk_n$, we can define
\[
	\Perv_{\gr, c}(\mathcal{Y})^{w=\nu} \defeq \Perv_{\gr, c}(\mathcal{Y}_n)^{w=\nu} \subset \Shv_{\gr, c}(\mathcal{Y}_n) \simeq \Shv_{\gr, c}(\mathcal{Y})
\]
where $\mathcal{Y}_n \in \Stk_{k_n}$ is any choice such that its base change to $\pt$ is $\mathcal{Y}$. Indeed, \cref{lem:gr_pure_perv_sheaves_invariant_base_change} guarantees that this is well-defined.

\subsubsection{$\Hom$ estimates and semisimplicity} We will now study morphisms between objects in $\Perv_{\gr, c}(\mathcal{Y})^{w=\nu}$ for various $\nu$. As a consequence, we will obtain the fact that $\Perv_{\gr, c}(\mathcal{Y})^{w=\nu}$ is classical and semi-simple.

\begin{prop}
\label{prop:Hom_Shv_gr_estimates_using_weights}
Let $\mathcal{F} \in \Perv_{\gr, c}(\mathcal{Y})^{w=k}$ and $\mathcal{G} \in \Perv_{\gr, c}(\mathcal{Y})^{w=l}$. Then, $\cHom_{\Shv_\gr(\mathcal{Y})^\ren}(\mathcal{F}, \mathcal{G})$ concentrates in cohomological degrees $[0, k-l]$. In particular, when $k-l<0$, then $\cHom_{\Shv_\gr(\mathcal{Y})^\ren}(\mathcal{F}, \mathcal{G})\simeq 0$.
\end{prop}
\begin{proof}
By construction, objects in $\Perv_{\gr, c}(\mathcal{Y})^{w=\nu}$ are direct summands of objects of the form $\gr_{\mathcal{Y}_n}(\mathcal{F})$ where $\mathcal{F} \in \Perv_{\mixed, c}(\mathcal{Y}_n)^{w=\nu}$. It thus suffices to show the statement above for $\mathcal{H}_\gr \defeq \cHom_{\Shv_\gr(\mathcal{Y})^\ren}(\gr(\mathcal{F}_n), \gr(\mathcal{G}_n))$ where $\mathcal{F}_n \in \Perv_{\mixed, c}(\mathcal{Y}_n)^{w=k}$ and  $\mathcal{G}_n \in \Perv_{\mixed, c}(\mathcal{Y}_n)^{w=l}$. We will use $\mathcal{F}$ and $\mathcal{G}$ to denote their pullbacks to $\mathcal{Y}$.

By~\cite[Prop. 3.9]{sun_decomposition_2012}, we know that $\mathcal{H}_\mixed \defeq \cHom_{\Shv_\mixed(\mathcal{Y}_n)^\ren}^\mixed(\mathcal{F}_n, \mathcal{G}_n) \in \Shv_\mixed(\pt_n)$ has Frobenius weights $\geq l-k$; see also~\cite[Prop. 5.1.15]{beilinson_faisceaux_2018} for the scheme version. By definition, this means $\Ho^i(\mathcal{H}_\mixed)$ has Frobenius weights $\geq i + l - k$. By \cref{prop:hom_graded_sheaves}, $\mathcal{H}_\gr$ is the \emph{naive} weight $0$ part of $\mathcal{H}_\mixed$. Thus, $\Ho^i(\mathcal{H}_\gr) \simeq 0$ when $i+l-k>0$ or equivalently, when $i>k-l$. In other words, we have shown that $\mathcal{H}_\gr$ is cohomologically supported on $(-\infty, k-l]$.

We know that $\mathcal{H} \defeq \cHom_{\Shv(\mathcal{Y})^\ren}(\mathcal{F}, \mathcal{G})$ concentrates in cohomological degrees $\geq 0$ since $\mathcal{F}$ and $\mathcal{G}$ lie in the heart of a $t$-structure, namely, the perverse $t$-structure, of $\Shv(\mathcal{Y})^\ren$. By \cref{cor:graded_Hom_direct_summand_unmixed_Hom}, $\mathcal{H}_\gr$ is a direct summand of $\mathcal{H}$. Thus, $\mathcal{H}_\gr$ is also supported in cohomological degrees $[0, \infty)$. Combined with the above, we know that $\mathcal{H}_\gr$ is supported in cohomological degrees $[0, k-l]$.
\end{proof}

The following result is a direct consequence of \cref{prop:Hom_Shv_gr_estimates_using_weights} above.

\begin{cor}
\label{cor:Perv_gr_w=v_is_classical_semisimple}
Let $\mathcal{F}, \mathcal{G} \in \Perv_{\gr, c}(\mathcal{Y})^{w=\nu}$. Then, $\cHom_{\Shv_\gr(\mathcal{Y})^\ren}(\mathcal{F}, \mathcal{G})$ concentrates in cohomological degree $0$. In particular, $\Perv_{\gr, c}(\mathcal{Y})^{w=\nu}$ is classical (as opposed to being a genuine $\infty$-category) and semi-simple. 

Here, semi-simple means that any exact triangle $A \to B \to C$ splits, where $A, C \in \Perv_{\gr, c}(\mathcal{Y})^{w=\nu}$. This, in particular, implies that $B \in \Perv_{\gr, c}(\mathcal{Y})^{w=\nu}$ as well since $B \simeq A \oplus C$.
\end{cor}

\cref{prop:Hom_Shv_gr_estimates_using_weights} can be slightly strengthened. We start with a lemma involving mixed sheaves.

\begin{lem}
\label{lem:no_weight_zero_Hom_perv_sheaves_diff_weights}
Let $\mathcal{F}, \mathcal{G} \in \Perv_{\mixed, c}(\mathcal{Y}_n)$ of Frobenius weights $k$ and $l$ respectively such that $k\neq l$. Then, $\Ho^0(\cHom_{\Shv_\mixed(\mathcal{Y}_n)^\ren}^\mixed(\mathcal{F}, \mathcal{G})) \in \Shv_\mixed(\pt_n)^{\heartsuit}$ has no weight $0$ component.
\end{lem}
\begin{proof}
Suppose $\Ho^0(\cHom_{\Shv_\mixed(\mathcal{Y}_n)^\ren}^\mixed(\mathcal{F}, \mathcal{G}))$ has non-zero weight $0$-component. Without changing the weight of $\mathcal{G}$, we can twist it by a sheaf of weight $0$ using a Frobenius eigenvalue of $\Ho^0(\cHom_{\Shv_\mixed(\mathcal{Y}_n)^\ren}^\mixed(\mathcal{F}, \mathcal{G}))$. We can thus assume that $\Ho^0(\cHom^\mixed_{\Shv_\mixed(\mathcal{Y}_n)^\ren}(\mathcal{F}, \mathcal{G})^\ren)^F \neq 0$ where $F$ is the Frobenius. We thus obtain a non-zero morphism $\varphi: \mathcal{F} \to \mathcal{G}$ between perverse sheaves of different weights. We will show that such a $\varphi$ cannot exist.

Since $\mathcal{F}$ and $\mathcal{G}$ have finite filtrations whose associated graded pieces are simple perverse sheaves, we can assume that $\mathcal{F}$ and $\mathcal{G}$ themselves are simple, without loss of generality. But then, $\varphi$ being non-zero implies that $\varphi$ is an isomorphism. This is not possible since they have different weights. The proof thus concludes.
\end{proof}

\begin{cor}
\label{cor:orthogonality_perv_gr_diff_weights}
Let $\mathcal{F} \in \Perv_{\gr, c}(\mathcal{Y})^{w=k}$ and $\mathcal{G} \in \Perv_{\gr, c}(\mathcal{Y})^{w=l}$ such that $k$ and $l$ are distinct. Then, $\Ho^0(\cHom_{\Shv_\gr(\mathcal{Y})^\ren}(\mathcal{F}, \mathcal{G})) \simeq 0$.
\end{cor}
\begin{proof}
As in the proof of \cref{prop:Hom_Shv_gr_estimates_using_weights}, we can assume that $\mathcal{F}$ and $\mathcal{G}$ are of the form $\gr(\mathcal{F})$ and $\gr(\mathcal{G})$ respectively, where $\mathcal{F}, \mathcal{G} \in \Perv_{\mixed, c}(\mathcal{Y}_n)$ are of weights $k$ and $l$ respectively. \cref{prop:hom_graded_sheaves} implies that $\Ho^0(\cHom_{\Shv_\gr(\mathcal{Y})^\ren}(\gr(\mathcal{F}), \gr(\mathcal{G})))$ is the weight $0$ part of $\Ho^0(\cHom_{\Shv_\mixed(\mathcal{Y}_n)^\ren}^\mixed(\mathcal{F}, \mathcal{G}))$, which vanishes by \cref{lem:no_weight_zero_Hom_perv_sheaves_diff_weights}.
\end{proof}

Combining \cref{prop:Hom_Shv_gr_estimates_using_weights,cor:orthogonality_perv_gr_diff_weights}, we obtain the following result.

\begin{thm}
\label{thm:refined_Hom_Shv_gr_estimates_using_weights}
Let $\mathcal{F} \in \Perv_{\gr, c}(\mathcal{Y})^{w=k}$ and $\mathcal{G} \in \Perv_{\gr, c}(\mathcal{Y})^{w=l}$. Then, $\cHom_{\Shv_\gr(\mathcal{Y})^\ren}(\mathcal{F}, \mathcal{G})$ concentrates in cohomological degrees $(0, k-l]$ when $k\neq l$ and $\{0\}$ when $k=l$. In particular, when $k-l<0$, then $\cHom_{\Shv_\gr(\mathcal{Y})^\ren}(\mathcal{F}, \mathcal{G})\simeq 0$.
\end{thm}

\subsubsection{Orthogonality}
\cref{cor:orthogonality_perv_gr_diff_weights} can be interpreted as a statement about orthogonality between $\Perv_{\gr, c}(\mathcal{Y})^{w=\nu}$ for different $\nu$. We will now study the same question but within the same $\nu$, see \cref{prop:Hom_simple_graded} below. We start with the following conservativity statement.

\begin{lem}
\label{lem:oblv_gr_isom_iff_for_pure_perv}
Let $\mathcal{F}_\gr, \mathcal{G}_\gr \in \Perv_{\gr, c}(\mathcal{Y})^{w=\nu}$, $\mathcal{F} = \oblv_\gr(\mathcal{F}_\gr)$, and $\mathcal{G} = \oblv_\gr(\mathcal{G}_\gr)$. Then $\mathcal{F}_\gr \simeq \mathcal{G}_\gr$ if and only if $\mathcal{F} \simeq \mathcal{G}$.
\end{lem}
\begin{proof}
The only if direction is clear. For the if direction, suppose $\mathcal{F} \simeq \mathcal{G}$. Then, by \cref{rmk:unmixed_Hom_direct_sum_gr_Hom}, 
\[
	\bigoplus_{k\in \mathbb{Z}} \Ho^0(\cHom_{\Shv_\gr(\mathcal{Y})^\ren}(\mathcal{F}_\gr, \mathcal{G}_\gr\lrangle{k})) \simeq \Ho^0(\cHom_{\Shv(\mathcal{Y})^\ren}(\mathcal{F}, \mathcal{G})).
\]
But unless $k=0$, $\mathcal{F}_\gr$ and $\mathcal{G}_\gr\lrangle{k}$ can have no non-trivial morphism between them, by \cref{thm:refined_Hom_Shv_gr_estimates_using_weights}. Thus, $k=0$, and we obtain an isomorphism
\[
	\Ho^0(\cHom_{\Shv_\gr(\mathcal{Y})^\ren}(\mathcal{F}_\gr, \mathcal{G}_\gr))) \simeq \Ho^0(\cHom_{\Shv(\mathcal{Y})^\ren}(\mathcal{F}, \mathcal{G})). \teq\label{eq:H^0_Hom_same_oblv_gr}
\]
In particular, there exists a non-zero $\varphi_\gr: \mathcal{F}_\gr \to \mathcal{G}_\gr$ such that $\oblv_\gr(\varphi_\gr): \mathcal{F} \to \mathcal{G}$ is an isomorphism. The proof concludes by the conservativity of $\oblv_\gr$, see \cref{lem:oblv_gr_conservative}.
\end{proof}

\begin{prop}
\label{prop:Hom_simple_graded}
In the situation of \cref{lem:oblv_gr_isom_iff_for_pure_perv}, suppose further that $\mathcal{F}, \mathcal{G} \in \Perv_c(\mathcal{Y})$ are simple. Then
\[
	\cHom_{\Shv_\gr(\mathcal{Y})^\ren}(\mathcal{F}_\gr, \mathcal{G}_\gr) \simeq 
	\begin{cases}
		\Qlbar, &\text{when }\mathcal{F}_\gr \simeq \mathcal{G}_\gr \text{ or equivalently, } \mathcal{F} \simeq \mathcal{G}, \\
		0, &\text{otherwise}.
	\end{cases}
\]
\end{prop}
\begin{proof}
By \cref{thm:refined_Hom_Shv_gr_estimates_using_weights}, we know that $\cHom_{\Shv_\gr(\mathcal{Y})^\ren}(\mathcal{F}_\gr, \mathcal{G}_\gr)$ concentrates in cohomological degree $0$. In particular, 
\[
	\cHom_{\Shv_\gr(\mathcal{Y})^\ren}(\mathcal{F}_\gr, \mathcal{G}_\gr) \simeq \Ho^0(\cHom_{\Shv_\gr(\mathcal{Y})^\ren}(\mathcal{F}_\gr, \mathcal{G}_\gr)).
\]
Thus we only need to deal with the cohomological degree $0$ part.

By \cref{lem:oblv_gr_isom_iff_for_pure_perv}, $\mathcal{F}_\gr \simeq \mathcal{G}_\gr$ if and only if $\mathcal{F} \simeq \mathcal{G}$. Moreover, simplicity of $\mathcal{F}$ and $\mathcal{G}$ implies that 
\[
	\Ho^0(\cHom_{\Shv(\mathcal{Y})^\ren}(\mathcal{F}, \mathcal{G})) \simeq 
	\begin{cases}
		\Qlbar, & \text{when } \mathcal{F} \simeq \mathcal{G}, \\
		0, &\text{otherwise}.
	\end{cases}
\]
We thus conclude using \cref{eq:H^0_Hom_same_oblv_gr}.
\end{proof}

\subsubsection{Generators} We will now describe the generators of $\Perv_{\gr, c}(\mathcal{Y})^{w=\nu}$, which results in an explicit description of the category. To start, let $\wtilde{\Perv}_{\gr, c}(\mathcal{Y})^{w=\nu}$ be the full subcategory of $\Shv_{\gr, c}(\mathcal{Y})$ spanned by finite direct sums of objects of the form $\gr_{\mathcal{Y}_n}(\mathcal{F}_n)$ for $\mathcal{F}_n \in \Perv_{\mixed, c}(\mathcal{Y}_n)^{w=\nu}$ for some $n$ such that the pullback of $\mathcal{F}_n$ to $\mathcal{Y}$ is simple. By construction and \cref{lem:gr_pure_perv_sheaves_invariant_base_change}, $\wtilde{\Perv}_{\gr, c}(\mathcal{Y})^{w=\nu}$ is a full subcategory of $\Perv_{\gr, c}(\mathcal{Y})^{w=\nu}$. By \cref{prop:Hom_simple_graded}, we obtain
\[
	\wtilde{\Perv}_{\gr, c}(\mathcal{Y})^{w=\nu} = \bigoplus_{s\in S} \Vect^{c, \heartsuit} 
\]
where $\Vect^{c, \heartsuit}$ is the \emph{abelian} category of finite dimensional vector spaces and $S$ is the set of simple perverse sheaves $\mathcal{F} \in \Perv_c(\mathcal{Y})$ such that $\mathcal{F}$ is the pullback of some (necessarily simple) perverse sheaf $\mathcal{F}_n$ over $\mathcal{Y}_n$ for some $n$. It is clear from the description above that $\wtilde{\Perv}_{\gr, c}(\mathcal{Y})^{w=\nu}$ is a semi-simple abelian category. In particular, it is also idempotent complete.

\begin{thm}
\label{thm:explicit_description_Perv_gr_c_w=nu}
$\Perv_{\gr, c}(\mathcal{Y})^{w=\nu} = \wtilde{\Perv}_{\gr, c}(\mathcal{Y})^{w=\nu}$ as full subcategories of $\Shv_{\gr, c}(\mathcal{Y})$. In particular,
\[
	\Perv_{\gr, c}(\mathcal{Y})^{w=\nu} \simeq \bigoplus_{s\in S} \Vect^{c, \heartsuit}
\]
is a semi-simple abelian category, where $S$ is the set of simple perverse sheaves $\mathcal{F} \in \Perv_c(\mathcal{Y})$ such that $\mathcal{F}$ is the pullback of some (necessarily simple) perverse sheaf $\mathcal{F}_n$ over $\mathcal{Y}_n$ for some $n$.
\end{thm}

Before proving \cref{thm:explicit_description_Perv_gr_c_w=nu}, we recall~\cite[Prop. 5.3.9.(ii)]{beilinson_faisceaux_2018} below. While the result is stated for schemes there, the same proof works for stacks. Indeed, the only ingredient is~\cite[Prop. 5.1.2]{beilinson_faisceaux_2018}, which is also valid for stacks, see also~\cite[Proof of Thm. 3.11]{sun_decomposition_2012}.

\begin{prop}[{\cite[Prop. 5.3.9.(ii)]{beilinson_faisceaux_2018}}]
\label{prop:simple_perv_are_pushforward}
Let $\mathcal{K}_n \in \Perv_{\mixed, c}(\mathcal{Y}_n)$ be a simple perverse sheaf. Then, there exists $m=nd$ for some $d$ and $\mathcal{L}_m \in \Perv_{\mixed, c}(\mathcal{Y}_m)$ such that $\mathcal{K}_n$ is the pushforward of $\mathcal{L}_m$ under $f_{mn}: \mathcal{Y}_m \to \mathcal{Y}_n$. Moreover, the pullback of $\mathcal{L}_m$ to $\mathcal{Y}$ is simple.
\end{prop}

\begin{proof}[Proof of \cref{thm:explicit_description_Perv_gr_c_w=nu}]
It remains to show that $\Perv_{\gr, c}(\mathcal{Y})^{w=\nu} \subseteq \wtilde{\Perv}_{\gr, c}(\mathcal{Y})^{w=\nu}$. Since the latter is idempotent complete, it suffices to show that $\gr(\mathcal{F}_n) \in \wtilde{\Perv}_{\gr, c}(\mathcal{Y})^{w=\nu}$ for any $\mathcal{F}_n \in \Perv_{\mixed, c}(\mathcal{Y}_n)^{w=\nu}$. Since $\Perv_{\gr, c}(\mathcal{Y})^{w=\nu}$ is semi-simple by \cref{cor:Perv_gr_w=v_is_classical_semisimple} and since any such $\mathcal{F}_n$ has a finite filtration whose associated graded pieces are simple, $\gr(\mathcal{F}_n)$ is a finite direct sum of objects of the form $\gr(\mathcal{K}_n)$ where $\mathcal{K}_n \in \Perv_{\mixed, c}(\mathcal{Y}_n)$ is simple. It remains to show that $\gr(\mathcal{K}_n) \in \wtilde{\Perv}_{\gr, c}(\mathcal{Y})^{w=\nu}$.

Let $\mathcal{L}_m \in \Perv_{\mixed, c}(\mathcal{Y}_m)$ be as in \cref{prop:simple_perv_are_pushforward}. Then, by definition, $\gr_{\mathcal{Y}_m}(\mathcal{L}_m) \in \wtilde{\Perv}_{\gr, c}(\mathcal{Y})$. Moreover, $f_{mn}^* \mathcal{K}_n \simeq \mathcal{L}_m^{\oplus d}$. By the factorization \cref{eq:factorization_gr_perv}, we obtain that
\[
	\gr_{\mathcal{Y}_n}(\mathcal{K}_n) \simeq \gr_{\mathcal{Y}_m}(\mathcal{L}_m)^{\oplus d}.
\]
Thus, $\gr_{\mathcal{Y}_n}(\mathcal{K}_n) \in \wtilde{\Perv}_{\gr, c}(\mathcal{Y})^{w=\nu}$ as desired.
\end{proof}

\subsection{Generation}
\label{subsec:pure_graded_perv_sheaves_generate}
We will now show that $\Perv_{\gr, c}(\mathcal{Y})^{w=\nu}$ for all $\nu$ together generate $\Shv_{\gr, c}(\mathcal{Y})$, i.e., $\Shv_{\gr, c}(\mathcal{Y})$ is the smallest full $\DG$-subcategory (or equivalently, full stable $\infty$-subcategory) of $\Shv_\gr(\mathcal{Y})^\ren$ containing $\Perv_{\gr, c}(\mathcal{Y})^{w=\nu}$ for all $\nu$.

Let $\mathcal{C}$ be a triangulated/stable $\infty$-/$\DG$-category and $\{S_i\}_{i \in I}$, where $I$ is an indexing set, a family of collections of objects in $\mathcal{C}$ or subcategories $\mathcal{C}$. Then, following~\cite{bondarko_weight_2012}, we use $\lrangle{S_i}_{i\in I}$ to denote the smallest full triangulated/stable $\infty$-/$\DG$-subcategory of $\mathcal{C}$ containing $S_i$ for all $i$. When $I$ is a singleton, i.e., we simply have one $S$, then we simply write $\lrangle{S}$.

The rest of the current subsection will be devoted to the proof of the following result.

\begin{thm}
\label{thm:Shv_gr_c_generated_pure_graded_perv_sheaves}
The categories $\Perv_{\gr, c}(\mathcal{Y})^{w=\nu}$ together generate $\Shv_{\gr, c}(\mathcal{Y})$. More precisely,
\[
	\Shv_{\gr, c}(\mathcal{Y}) \simeq \lrangle{\Perv_{\gr, c}(\mathcal{Y})^{w=\nu}}_{\nu \in \mathbb{Z}}.
\]
\end{thm}

\subsubsection{}
\label{subsubsec:perv_gr_c_w=nu_generates_steps}
By definition, $\lrangle{\Perv_{\gr, c}(\mathcal{Y})^{w=\nu}} \subseteq \Shv_{\gr, c}(\mathcal{Y})$. To prove the other inclusion, it suffices to show that
\begin{myenum}{(\roman*)}
	\item \label{item:a} $\lrangle{\Perv_{\gr, c}(\mathcal{Y})^{w=\nu}}_{\nu \in \mathbb{Z}}$ contains $\gr(\mathcal{F}_n)$ for all $\mathcal{F}_n \in \Shv_{\mixed, c}(\mathcal{Y}_n)$ for some fixed (and hence, all) $n$, and
	\item $\lrangle{\Perv_{\gr, c}(\mathcal{Y})^{w=\nu}}_{\nu\in \mathbb{Z}}$ is idempotent complete.
\end{myenum}

\subsubsection{} 
We will now prove the first item in \cref{subsubsec:perv_gr_c_w=nu_generates_steps}. Since $\lrangle{\Perv_{\gr, c}(\mathcal{Y})^{w=\nu}}_{\nu \in \mathbb{Z}}$ is closed under extensions, using the perverse $t$-structure on $\Shv_{\mixed, c}(\mathcal{Y})$, we reduce to the case where $\mathcal{F}_n \in \Perv_{\mixed, c}(\mathcal{Y}_n)$. By~\cite[Theorem 9.2]{laszlo_perverse_2009}, $\mathcal{F}_n$ has a weight filtration whose associated graded pieces are pure. This allows us to further reduce to the case where $\mathcal{F}_n$ is a pure perverse sheaf of weight $\nu$ for some $\nu$. But then, we are done since, $\gr(\mathcal{F}_n) \in \Perv_{\gr, c}(\mathcal{Y})^{w=\nu}$.

\subsubsection{}
We will prove the second item in \cref{subsubsec:perv_gr_c_w=nu_generates_steps} in the remainder of this subsection. It suffices to show that $\lrangle{\Perv_{\gr, c}(\mathcal{Y})^{w=\nu}}_{\nu\in[k, l]}$ is idempotent complete for any $k\leq l$ since $\lrangle{\Perv_{\gr, c}(\mathcal{Y})^{w=\nu}}_{\nu\in\mathbb{Z}}$ is the union of categories of this form. We will prove this by induction on the length $l-k$.

\subsubsection{Base case}
For the base case $k=l=\nu$, consider $\lrangle{\Perv_{\gr, c}(\mathcal{Y})^{w=\nu}}$. \cref{thm:explicit_description_Perv_gr_c_w=nu,prop:Hom_simple_graded} imply that
\[
	\lrangle{\Perv_{\gr, c}(\mathcal{Y})^{w=\nu}} \simeq \bigoplus_{s \in S} \Vect^c,
\]
which is idempotent complete.

\subsubsection{Induction step}
Suppose we know that $\lrangle{\Perv_{\gr, c}(\mathcal{Y})^{w=\nu}}_{\nu\in [k, l]}$ is idempotent complete. We will show that $\lrangle{\Perv_{\gr, c}(\mathcal{Y})^{w=\nu}}_{\nu\in [k, l+1]}$ is also idempotent complete. But first, we will need some preparation. 

We start with the following lemma, which is a $\DG$-categorical counterpart of~\cite[Lem. 1.1.5]{bondarko_weight_2012}.

\begin{lem}
\label{lem:semi_orthogonality_adjoint}
Let $\mathcal{C}$ be a $\DG$-category and $\mathcal{C}_1$, $\mathcal{C}_2$ full $\DG$-subcategories of $\mathcal{C}$. Let $\mathcal{C}^\circ = \lrangle{\mathcal{C}_1, \mathcal{C}_2}$ be the smallest full $\DG$-subcategory of $\mathcal{C}$ containing $\mathcal{C}_1$ and $\mathcal{C}_2$. Suppose that for any $c_1 \in \mathcal{C}_1$ and $c_2 \in \mathcal{C}_2$, $\cHom_\mathcal{C}(c_1, c_2) \simeq 0$. Then, we have the following adjoint pairs $F_1 \dashv G_1$ and $F_2 \dashv G_2$ fitting into the following diagram
\[
\begin{tikzcd}
	\mathcal{C}_1 \ar[shift left=\arrdisp, hookrightarrow]{r}{F_1} & \ar[shift left=\arrdisp]{l}{G_1} \mathcal{C}^\circ \ar[shift left=\arrdisp]{r}{F_2} & \ar[shift left=\arrdisp, hookrightarrow]{l}{G_2} \mathcal{C}_2
\end{tikzcd}
\]
such that for any $c\in \mathcal{C}^\circ$, we have an exact triangle
\[
	F_1G_1 c \to c \to G_2F_2 c.
\]
\end{lem}
\begin{proof}
Let $\mathcal{D}^\circ$ be the full subcategory of $\mathcal{C}^\circ$ spanned by objects $d$ such that the functor
\begin{align*}
	\mathcal{C}_1^\opp &\to \Vect \\
	c &\mapsto \cHom_\mathcal{C}(F_1(c), d)
\end{align*}
is representable. Denoting the representing object $G_1(d)$, we obtain a functor $G_1: \mathcal{D}^\circ \to \mathcal{C}_1$ which is a partial right adjoint to $F_1$. It is easy to see that $\mathcal{D}^\circ$ contains $\mathcal{C}_1$ and $\mathcal{C}_2$: $G_1(d) = d$ when $d\in \mathcal{C}_1$ and $G_1(d) = 0$ when $d\in \mathcal{C}_2$. It is also easy to check that $\mathcal{D}^\circ$ is closed under finite direct sums, shifts, and cones, or equivalently, $G_1$ is compatible with these operations. In particular, $\mathcal{D}^\circ = \mathcal{C}^\circ$ and hence, $G_1$ is defined on $\mathcal{C}^\circ$ as desired.

We have the following pairs of adjoint functors
\[
\begin{tikzcd}
	\mathcal{C}_1 \ar[shift left=\arrdisp, hookrightarrow]{r}{F_1} & \ar[shift left=\arrdisp]{l}{G_1} \mathcal{C}^\circ \ar[shift left=\arrdisp]{r}{F_2} & \ar[shift left=\arrdisp, hookrightarrow]{l}{G_2} \ker{G_1}
\end{tikzcd}
\]
where $\ker G_1$ is the full subcategory of $\mathcal{C}^\circ$ consisting of objects $c \in \mathcal{C}^\circ$ such that $G_1(c) \simeq 0$, and $F_2(c) = \Cone(F_1G_1 c \to c)$. We will now show that $\ker G_1 = \mathcal{C}_2$, which will conclude the proof. 

As already seen above, $\mathcal{C}_2 \subseteq \ker G_1$. For the other inclusion, it suffices to show that $F_2 c \in \mathcal{C}_2$ for any $c\in \mathcal{C}^\circ$. For that, note that if we let $\mathcal{E}^\circ$ be the full subcategory of $\mathcal{C}^\circ$ consisting of objects $c$ such that $F_2 c \in \mathcal{C}_2$, then $\mathcal{E}^\circ$ contains $\mathcal{C}_1, \mathcal{C}_2$ and is closed under finite direct sums, shifts, and cones. Thus, $\mathcal{E}^\circ = \mathcal{C}^\circ$ and we are done.
\end{proof}

\begin{lem}
\label{lem:idempotent_complete_generation}
Consider the situation of \cref{lem:semi_orthogonality_adjoint}. Suppose that $\mathcal{C}_1$ and $\mathcal{C}_2$ are idempotent complete. Then, so is $\mathcal{C}^\circ = \lrangle{\mathcal{C}_1, \mathcal{C}_2}$.
\end{lem}
\begin{proof}
Replacing $\mathcal{C}$ by its idempotent completion if necessary, we can assume that $\mathcal{C}$ is also idempotent complete without changing $\mathcal{C}_1$, $\mathcal{C}_2$, or $\mathcal{C}^\circ = \lrangle{\mathcal{C}_1, \mathcal{C}_2}$. Let $\Idem$ denote the category given in~\cite[Defn. 4.4.5.2]{lurie_higher_2017-1}. Consider $\rho: \Idem \to \mathcal{C}^\circ$, which determines an object $c \in \mathcal{C}$ and a map $e: c\to c$ such that $e^2$ is homotopic to $e$. Choosing a left cofinal map $\mathbb{Z}_{\geq 0} \to \Idem$,~\cite[Prop. 4.4.5.17]{lurie_higher_2017-1}, we are reduced to showing that the following diagram
\[
	c \xrightarrow{e} c \xrightarrow{e} c \xrightarrow{e} \cdots
\]
has a colimit in $\mathcal{C}^\circ$.

From \cref{lem:semi_orthogonality_adjoint}, we obtain functors $\mathbb{Z}_{\geq 0} \to \Idem \to \mathcal{C}_1$ and $\mathbb{Z}_{\geq 0} \to \Idem \to \mathcal{C}_2$ given by idempotents $e_1=F_1G_1e$ on $F_1G_1c$ and $e_2[-1] = G_2F_2e[-1]$ on $G_2F_2c[-1]$, fitting into the following diagram
\[
\begin{tikzcd}
	G_2F_2 c[-1] \ar{d} \ar{r}{e_2[-1]} & G_2F_2 c[-1] \ar{d} \ar{r}{e_2[-1]} & G_2F_2 c[-1] \ar{d} \ar{r}{e_2[-1]} & \cdots\\
	F_1G_1 c \ar{d} \ar{r}{e_1} & F_1G_1 c \ar{d} \ar{r}{e_1} & F_1G_1 c \ar{d} \ar{r}{e_1} & \cdots \\
	c \ar{r}{e} & c \ar{r}{e} & c \ar{r}{e} & \cdots
\end{tikzcd}
\]
where the columns are exact triangles. The colimits of the first two rows are in $\mathcal{C}^\circ$ and hence, so is the colimit of the last row. Thus, we are done.
\end{proof}

\begin{rmk}
The lemmas above hold equally true for stable $\infty$-categories in general. The only change is that instead of $\cHom_\mathcal{C}(c_1, c_2)$, we need to formulate our statement in terms of spectra-enriched $\Hom$ instead. We do not need this in the current paper.
\end{rmk}

Back to our induction step, by \cref{thm:refined_Hom_Shv_gr_estimates_using_weights}, it is easy to see that if we take $\mathcal{C}_1 = \lrangle{\Perv_{\gr, c}(\mathcal{Y})^{w=\nu}}_{\nu\in [k, l]}$ and $\mathcal{C}_2 = \lrangle{\Perv_{\gr, c}(\mathcal{Y})^{w=l+1}}$, then $\mathcal{C}_1, \mathcal{C}_2$ and $\mathcal{C} = \Shv_{\gr, c}(\mathcal{Y})$ satisfy the conditions of \cref{lem:idempotent_complete_generation}. Consequently, $\lrangle{\Perv_{\gr, c}(\mathcal{Y})^{w=\nu}}_{\nu\in [k, l+1]}$ is idempotent complete and the induction step concludes.

\subsubsection{} We have thus finished proving \cref{thm:Shv_gr_c_generated_pure_graded_perv_sheaves}.

\subsection{Weight structure and perverse \texorpdfstring{$t$}{t}-structure}
\label{subsec:weight_and_perverse_t-str_graded_sheaves}

In this subsection, we will finally construct a weight structure and a perverse $t$-structure on $\Shv_{\gr, c}(\mathcal{Y})$ for any $\mathcal{Y} \in \Stk_k$. The preparation done in \cref{subsec:pure_graded_perv_sheaves,subsec:pure_graded_perv_sheaves_generate} allows us to apply~\cite{bondarko_weight_2012} directly, yielding a \emph{transversal weight} and $t$-structure on $\Shv_{\gr, c}(\mathcal{Y})$.

\subsubsection{Transversal weight and $t$-structures}
We begin by recalling various results from~\cite{bondarko_weight_2012}.

\begin{defn}[{\cite[Defn. 1.1.4]{bondarko_weight_2012}}]
\label{defn:semi_orthogonal_generating_family}
\begin{myenum}{(\roman*)}
\item Let $\mathcal{C}$ be a triangulated category. We say that a family $\{A_i\}_{i\in \mathbb{Z}}$ of full subcategories $A_i \subseteq \mathcal{C}$ is \emph{semi-orthogonal} if for any $i, j$, $\Hom_\mathcal{C}(A_i, A_j[*])$ is supported on $*\in [0, i-j]$. We will say that $\{A_i\}_{i\in \mathbb{Z}}$ is \emph{strongly semi-orthogonal} if, in addition to the above, $\Hom_\mathcal{C}(A_i, A_j) = 0$ when $i \neq j$.\footnote{Note that since we are working with a triangulated category $\mathcal{C}$, $\Hom_\mathcal{C}(-, -)$ is a set, rather than a space.}

\item We will say that $\{A_i\}_{i \in \mathbb{Z}}$ is \emph{generating} (in $\mathcal{C}$) if $\lrangle{A_i}_{i\in \mathbb{Z}} = \mathcal{C}$. See \cref{subsec:pure_graded_perv_sheaves_generate} for the definition of $\lrangle{-}$.\footnote{Note that this notion is different from the one discussed in \cref{subsubsec:compact_generation}.}
\end{myenum}
\end{defn}

\begin{thm}[{\cite[Thm. 1.2.1]{bondarko_weight_2012}}]
\label{thm:bondarko_main_thm}
Fix a triangulated category $\mathcal{C}$.

Let $\{A_i\}_{i\in \mathbb{Z}}$ be a family of full subcategories of $\mathcal{C}$. Then the following are equivalent (all weight and $t$-structures appearing below are bounded)
\begin{myenum}{(\roman*)}
	\item Each $A_i$ is abelian semi-simple and $\{A_i\}_{i\in \mathbb{Z}}$ is a strongly semi-orthogonal generating family (see \cref{defn:semi_orthogonal_generating_family}) in $\mathcal{C}$.
	\item $\{A_i\}_{i\in \mathbb{Z}}$ is a semi-orthogonal family in $\mathcal{C}$ such that if we let $\mathcal{C}^{t\leq 0}$ (resp. $\mathcal{C}^{t\geq 0}$) be the smallest extension-closed full subcategory of $\mathcal{C}$ containing $\cup_{i \in \mathbb{Z}, j \geq 0} A_i[j]$ (resp. $\cup_{i \in \mathbb{Z}, j \leq 0} A_i[j]$), then $(\mathcal{C}^{t\leq 0}, \mathcal{C}^{t\geq 0})$ yields a $t$-structure on $\mathcal{C}$.
	\item \label{item:bondarko_main_thm:t-heart} $\{A_i\}_{i\in \mathbb{Z}}$ is a semi-orthogonal family in $\mathcal{C}$ such that the smallest extension-closed full subcategory of $\mathcal{C}$ containing $\cup_{i\in \mathbb{Z}} A_i$ is the heart of a $t$-structure.
	\item $\{A_i\}_{i\in \mathbb{Z}}$ is a semi-orthogonal family in $\mathcal{C}$ and $\mathcal{C}$ has a $t$-structure whose heart $\mathcal{C}^{\theart}$ contains $A_i$'s such that for all $X \in \mathcal{C}^{\theart}$ there exists an exhaustive separated increasing filtration by subobjects $W_{\leq i} X$ such that $W_{\leq i} X/W_{\leq i-1}X \in A_i$ for all $i$.
\end{myenum}

The above conditions are equivalent to the following equivalent conditions
\begin{myenum}{(\alph*)}
	\item $\mathcal{C}$ has a weight and a $t$-structure, such that for all $i\in \mathbb{Z}$ and $X\in \mathcal{C}^{\theart}$ there exists a weight truncation $w_{\leq i}X \to X \to w_{\geq i+1}X$ such that all terms are in $\mathcal{C}^{\theart}$. We call such a truncation a \emph{nice decomposition}.
	\item $\mathcal{C}$ has a weight and a $t$-structure such that nice decompositions exist for any object in $\mathcal{C}^{\theart}$. Moreover, they are functorial in $X \in \mathcal{C}^{\theart}$ (if we fix $i$) and the corresponding functors $X \mapsto w_{\leq i} X$ and $X\mapsto w_{\geq i+1} X$ are exact (as endo-functors on $\mathcal{C}^{\theart}$).
	\item $\mathcal{C}$ has a weight and a $t$-structure such that for any $X \in \mathcal{C}^{\theart}$ and any $w_{\leq i} X$ (that is part of a weight truncation of $X$), $\Im(\tH{t}^0(w_{\leq i}X)\to X) \to X$ extends to a nice decomposition of $X$.
\end{myenum}

Moreover the $\{A_i\}_{i\in \mathbb{Z}}$ in the first set of equivalent conditions (indexed by Roman numerals) can be obtained from the weight and $t$-structures in the second set of equivalent conditions (indexed by Latin characters) as follows: $A_i = \mathcal{C}^{\theart} \cap \mathcal{C}^{w=i}$. Namely, $A_i$ consists of objects in the $t$-heart that have weight $i$.
\end{thm}

\begin{defn}
Let $\mathcal{C}$ be a triangulated category equipped with a weight $w$ and a $t$-structure. If $w$ and $t$ satisfy the (equivalent) conditions of \cref{thm:bondarko_main_thm}, we will say that $t$ is transversal to $w$.
\end{defn}

\begin{rmk}
\label{rmk:transversal_weight_t_structures}
\begin{myenum}{(\roman*)}
\item Note that since weight and $t$-structures on a stable $\infty$-category is defined as structures on the underlying triangulated category, all theorems about weight and $t$-structures on triangulated category (\cref{thm:bondarko_main_thm} above and \cref{prop:properties_transversal_weight_t} below, in particular) apply equally well to stable $\infty$-categories.
\item From the proof of \cref{thm:bondarko_main_thm}, we see that the weight structure on $\mathcal{C}$ is given by the the pair $(\mathcal{C}^{w\leq 0}, \mathcal{C}^{w\geq 0})$ where $\mathcal{C}^{w\leq 0}$ (resp. $\mathcal{C}^{w\geq 0}$) is the idempotent-closure in $\mathcal{C}$ of the smallest extension-closed full subcategory of $\mathcal{C}$ containing finite direct sum of objects in $\cup_{i\in \mathbb{Z}, j, i+j\leq 0} A_i[j]$ (resp. $\cup_{i\in \mathbb{Z}, j, i+j\geq 0} A_i[j]$).
\item \label{item:decomposition_theorem} $\mathcal{C}^{\weightheart}$ is the full subcategory of $\mathcal{C}$ spanned by finite direct sums of objects in $A_i[-i]$'s. In particular, any $X \in \mathcal{C}^{\weightheart}$ is equivalent to $\bigoplus_i \Ho^i(X)[-i]$ where $\Ho^i(X)[-i] \in A_i[-i]$, see~\cite[Rmk. 1.2.3.2]{bondarko_weight_2012}. This is an analog of the decomposition theorem of~\cite{beilinson_faisceaux_2018}.
\end{myenum}
\end{rmk}

\begin{prop}[{\cite[Prop. 1.2.4]{bondarko_weight_2012}}]
\label{prop:properties_transversal_weight_t}
Let $\mathcal{C}$ be a triangulated category equipped with a weight structure and a $t$-structure such that $t$ is transversal to $w$. Then,
\begin{myenum}{(\roman*)}
	\item $t$-truncations $X \mapsto \tau^{\leq i} X$ and $X \mapsto \tau^{\geq i} X$ are weight-exact for any $i$, i.e., they preserve $\mathcal{C}^{w\geq 0}$, $\mathcal{C}^{w\leq 0}$, and hence, $\mathcal{C}^{\weightheart}$.
	\item Let $X \in \mathcal{C}$ and $i \in \mathbb{Z}$. Then, $X \in \mathcal{C}^{w\leq i}$ (resp. $X \in \mathcal{C}^{w\geq i}$) if and only if for all $j \in \mathbb{Z}$, $W_{\leq i+j} \tH{t}^j(X) \simeq \tH{t}^j(X)$ (resp. $W_{\leq i+j-1} \tH{t}^j(X) \simeq 0$).
	\item For $X \in \mathcal{C}^{\theart}$, denote $W_{\geq i+1} X \defeq X/W_{\leq i}X$. We have the following pairs of adjoint functors
	\[
	\begin{tikzcd}
		\mathcal{C}^{w\leq i} \cap \mathcal{C}^{\theart}\ar[shift left=\arrdisp, hookrightarrow]{r} & \ar[shift left=\arrdisp]{l}{W_{\leq i}} \mathcal{C}^{\theart} \ar[shift left=\arrdisp]{r}{W_{\geq i+1}} & \ar[shift left=\arrdisp, hookrightarrow]{l} \mathcal{C}^{w\geq i+1} \cap \mathcal{C}^{\theart}
	\end{tikzcd}
	\]
	Moreover, both $W_{\leq i}$ and $W_{\geq i+1}$ are exact.
	\item The functors $X \mapsto W_{\leq i}(W_{\geq i} X)$ and $X \mapsto W_{\geq i}(W_{\leq i}X)$ are canonically isomorphic as functors $\mathcal{C}^{\theart} \to A_i$. For $X \in \mathcal{C}^{\theart}$, we write $\Gr^W_i X \defeq W_{\leq i} X/W_{\leq i-1} X \in \mathcal{C}^{\theart}$.
	\item For $X \in \mathcal{C}^{\theart}$, $X\in \mathcal{C}^{w\leq i}$ (resp. $X \in \mathcal{C}^{w\geq j}$) if and only if $\Gr_j^W X = 0$ for all $j>i$ (resp. $j<i$).
\end{myenum}
\end{prop}

\subsubsection{} We will now apply the discussion above to the case of $\Shv_{\gr, c}(\mathcal{Y})$ where $\mathcal{Y} \in \Stk_k$.

\begin{thm}
\label{thm:weight_t_structures_graded_sheaves}
For any $\mathcal{Y} \in \Stk_k$, $\Shv_{\gr, c}(\mathcal{Y})$ is equipped with a weight and a $t$-structure such that the $t$-structure is transversal to the weight structure.

This $t$-structure will be referred to as the perverse $t$-structure, whose heart is the category of graded perverse sheaves $\Perv_{\gr, c}(\mathcal{Y}) \defeq \Shv_{\gr, c}(\mathcal{Y})^{\theart}$.
\end{thm}
\begin{proof}
Take $A_i = \Perv_{\gr, c}(\mathcal{Y})^{w=i}$ as in \cref{eq:defn_pure_gr_perv_sheaves}. \cref{thm:refined_Hom_Shv_gr_estimates_using_weights,thm:explicit_description_Perv_gr_c_w=nu,thm:Shv_gr_c_generated_pure_graded_perv_sheaves} imply that the family $\{A_i\}_{i \in \mathbb{Z}}$ satisfies the conditions of \cref{thm:bondarko_main_thm}. The proof thus concludes.
\end{proof}

\begin{rmk}
\label{rmk:pure_graded_perverse_sheaves_are_indeed_pure}
\cref{thm:bondarko_main_thm} gives another characterization of of $\Perv_{\gr, c}(\mathcal{Y})^{w=k}$ defined in \cref{eq:defn_pure_gr_perv_sheaves}. Namely, $\Perv_{\gr, c}(\mathcal{Y})^{w=k} = \Perv_{\gr, c}(\mathcal{Y}) \cap \Shv_{\gr, c}(\mathcal{Y})^{w=k}$ is precisely the category of graded perverse sheaves of pure weight $k$. We will also use $\Perv_{\gr, c}(\mathcal{Y})^{\weightheart}$ to denote $\Perv_{\gr, c}(\mathcal{Y})^{w=0}$, the category of pure graded perverse sheaves of weight $0$.
\end{rmk}

\begin{rmk}
The action of $\Vect^{\gr, c} \simeq \Shv_{\gr, c}(\pt)$ on $\Shv_{\gr, c}(\mathcal{Y})$ is compatible with the weight and $t$-structures involved. Indeed, it suffices to check for a single vector space concentrated in one graded and cohomological degree. But now, it is easy to see since the action by such an object is simply given by a combination of taking finite direct sums, cohomological degree shifts, and graded degree shifts. As a result, $\Vect^{\gr, c, \weightheart}$ acts on $\Shv_{\gr, c}(\mathcal{Y})^{\weightheart}$ and $\Vect^{\gr, c, \theart}$ acts on $\Perv_{\gr, c}(\mathcal{Y})$. Thus, $\Vect^{\gr, c} \simeq \Ch^b(\Vect^{\gr, c, \weightheart})$ also acts on $\Ch^b(\ho\Shv_{\gr, c}(\mathcal{Y})^{\weightheart})$. Moreover, this action is compatible with the weight complex functors.
\end{rmk}

\subsection{Formal properties}
\label{subsec:formal_properties_weight_Shv_gr}
Since $\Shv_{\gr, c}(\mathcal{Y})$ is equipped with a perverse $t$-structure, transversal to a weight structure, all statements in \cref{thm:bondarko_main_thm,prop:properties_transversal_weight_t,rmk:transversal_weight_t_structures} apply to $\Shv_{\gr, c}(\mathcal{Y})$ as well. In this subsection, we are interested in the interaction between these structures and those on the categories of mixed sheaves on $\mathcal{Y}_n$ and constructible sheaves on $\mathcal{Y}$.

The general meta-theorem is that the perverse $t$-structure on $\Shv_{\gr, c}(\mathcal{Y})$ is compatible with those on $\Shv_{\mixed, c}(\mathcal{Y}_n)$ and $\Shv_c(\mathcal{Y})$. Similarly, the weight structure on $\Shv_{\gr, c}(\mathcal{Y})$ is compatible with the notion of weight on $\Shv_{\mixed, c}(\mathcal{Y}_n)$ as defined in~\cite{beilinson_faisceaux_2018,laszlo_perverse_2009}. Finally, results regarding weights in the theory of mixed sheaves have natural analogs in the theory of graded sheaves. We will collect some of these properties below. Our goal is not to be exhaustive; rather, we aim demonstrate how straightforward it is to adapt known results from the theory of (mixed) sheaves to the graded sheaf setting.

\begin{prop}
\label{prop:t-exactness}
Let $\mathcal{Y}_n \in \Stk_{k_n}$ and $\mathcal{Y}$ its base change to $k$. Then, the functors
\[
	\Shv_{\mixed, c}(\mathcal{Y}_n) \xrightarrow{\gr} \Shv_{\gr, c}(\mathcal{Y}) \xrightarrow{\oblv_\gr} \Shv_c(\mathcal{Y})
\]
preserves and reflects $t$-structures with respect to the perverse $t$-structures. Namely, for any $n$ and any $\mathcal{F} \in \Shv_{\gr, c}(\mathcal{Y})$, $\mathcal{F} \in \Shv_{\gr, c}(\mathcal{Y})^{t\leq n}$ (resp. $\mathcal{F} \in \Shv_{\gr, c}(\mathcal{Y})^{t\geq n}$) if and only if $\oblv_\gr(\mathcal{F}) \in \Shv_c(\mathcal{Y})^{t\leq n}$ (resp. $\oblv_\gr(\mathcal{F}) \in \Shv_c(\mathcal{Y})^{t\geq n}$). We have a similar statement for $\gr$.
\end{prop}
\begin{proof}
We will now show that $\gr$ is $t$-exact, i.e., it preserves the $t$-structures involved; $t$-exactness for $\oblv_{\gr}$ can be argued similarly. It suffices to show that $\gr$ preserves the $t$-heart. By~\cite[Thm. 9.2]{laszlo_perverse_2009}, any mixed perverse sheaf admits a finite filtration whose associated graded pieces are pure perverse sheaves. Thus, it suffices to show that $\gr$ sends pure perverse sheaves to (pure) graded perverse sheaves. But this follows from the description of the $t$-heart given in \cref{thm:bondarko_main_thm}.\ref{item:bondarko_main_thm:t-heart} and the definition of $A_i = \Perv_{\gr, c}(\mathcal{Y})^{w=i}$ given in \cref{eq:defn_pure_gr_perv_sheaves}.

Note the general fact that a conservative $t$-exact functor necessarily reflects the $t$-structure as well. The desired conclusion then follows from the conservativity of $\gr$ and $\oblv_\gr$, see \cref{prop:F_M_is_conservative_when_F_is,lem:oblv_gr_conservative}.
\end{proof}

\begin{prop}
\label{prop:w-exactness}
Let $\mathcal{Y}$ and $\mathcal{Y}_n$ be as above. The functor $\gr: \Shv_{\mixed, c}(\mathcal{Y}_n) \to \Shv_{\gr, c}(\mathcal{Y})$ preserves and reflects weights, where we use Frobenius weights on $\Shv_{\mixed, c}(\mathcal{Y}_n)$ and our weight structure on $\Shv_{\gr, c}(\mathcal{Y})$. Namely, for any $k$ and $\mathcal{F} \in \Shv_{\mixed, c}(\mathcal{Y})$, we have $\mathcal{F} \in \Shv_{\mixed, c}(\mathcal{Y})^{w\leq k}$ (resp. $\mathcal{F} \in \Shv_{\mixed, c}(\mathcal{Y})^{w\geq k}$) if and only if $\gr(\mathcal{F}) \in \Shv_{\gr, c}(\mathcal{Y})^{w \leq k}$ (resp. $\gr(\mathcal{F}) \in \Shv_{\gr, c}(\mathcal{Y})^{w \geq k}$).
\end{prop}
\begin{proof}
Perverse $t$-truncation is compatible with Frobenius weights on mixed sheaves, by~\cite[Thm. 3.5]{sun_decomposition_2012}, and is compatible with the transversal weight structure on graded sheaves by \cref{prop:properties_transversal_weight_t}. Moreover, $\gr$ is $t$-exact, by \cref{prop:t-exactness}. It, therefore, suffices to consider the case where $\mathcal{F}$ is perverse.

Now, $\mathcal{F}$ has a finite filtration whose associated graded pieces are pure perverse sheaves, by~\cite[Thm. 9.2]{laszlo_perverse_2009}. In particular, $\mathcal{F}$ has Frobenius weight $\leq k$ (resp. $\geq k$) if and only if all the pure pieces have Frobeinus weights $\leq k$ (resp. $\geq k$). This allows us to further reduce to the case where $\mathcal{F}$ is pure. But then, this follows from the definition of $\Perv_{\gr, c}(\mathcal{Y})^{w=k}$ given in \cref{eq:defn_pure_gr_perv_sheaves} and \cref{rmk:pure_graded_perverse_sheaves_are_indeed_pure}.
\end{proof}

\begin{rmk}
Pre-composing the weight complex functor $\wt: \Shv_{\gr, c}(\mathcal{Y}) \to \Ch^b(\ho\Shv_{\gr, c}(\mathcal{Y})^{\weightheart})$ (see \cref{eq:wt_complex_functor}) with $\gr$, we obtain a functor $\Shv_{\mixed, c}(\mathcal{Y}_n) \to \Ch^b(\ho\Shv_{\gr, c}(\mathcal{Y})^{\weightheart})$. Due to weight exactness of $\gr$, \cref{prop:w-exactness}, any (Frobenius) weight filtration on $\mathcal{F} \in \Shv_{\mixed, c}(\mathcal{Y}_n)$ yields a weight filtration on $\gr(\mathcal{F})$. Hence, using $\Chr(\mathcal{F}) \in \Ch^b(\ho\Shv_{\mixed, c}(\mathcal{Y}_n)^{w=0})$ to denote the chromatographic complex of $\mathcal{F}$ in the sense of~\cite[\S3.5]{webster_geometric_2017}, we have
\[
	\gr(\Chr(\mathcal{F})) \simeq \wt(\gr(\mathcal{F})),
\]
see also \cref{rmk:weight_complex_functor_explicit} for an explicit description of the weight complex functor. This is the precise sense in which the weight complex functor is compatible with the chromatographic complex construction, answering the question posed in~\cite[Rmk. 2]{webster_geometric_2017}.
\end{rmk}

\begin{cor}
The weight structure and perverse $t$-structure on $\Shv_{\gr, c}(\mathcal{Y})$ is symmetric with respect to the Verdier duality functor for graded sheaves. Namely, for any integer $n$, Verdier duality induces an equivalence of categories
\begin{align*}
	\Shv_{\gr, c}(\mathcal{Y})^{t\leq n} &\xrightarrow[\simeq]{\DVer} \Shv_{\gr, c}(\mathcal{Y})^{t\geq -n} \\
	\Shv_{\gr, c}(\mathcal{Y})^{w\leq n} &\xrightarrow[\simeq]{\DVer} \Shv_{\gr, c}(\mathcal{Y})^{w\geq -n}.
\end{align*}
\end{cor}

\begin{prop}
\label{prop:functoriality_weights}
Let $f: \mathcal{Y} \to \mathcal{Z}$ be a morphism in $\Stk_k$. Then, for any $k \in \mathbb{Z}$,
\begin{myenum}{(\roman*)}
	\item $f^* \ (\text{resp. } f^!): \Shv_{\gr, c}(\mathcal{Z}) \to \Shv_{\gr, c}(\mathcal{Y})$ preserves the property of having weight $\leq k$ (resp. $\geq k$). Moreover, $f^*$ and $f^!$ are weight exact when $f$ is smooth.
	\item $f_* \ (\text{resp. } f_!): \Shv_{\gr, c}(\mathcal{Y}) \to \Shv_{\gr, c}(\mathcal{Z})$ preserves the property of having weight $\geq k$ (resp. $\leq k$). In particular, $f_* \simeq f_!$ is weight-exact when $f$ is proper.
	\item $f_*$ and $f_!$ are $t$-exact when $f$ is an affine open embedding.
	\item $f^*[d]\lrangle{d} \simeq f^![-d]\lrangle{-d}: \Shv_{\gr, c}(\mathcal{Z}) \to \Shv_{\gr, c}(\mathcal{Y})$ is weight and $t$-exact when $f$ is smooth of relative dimension $d$.
\end{myenum}
\end{prop}
\begin{proof}
We will show that $f^*$ is weight exact when $f$ is smooth. The rest can be argued in a similar manner. As before, by \cref{prop:properties_transversal_weight_t}, it suffices to show that $f^* \mathcal{F}$ is pure when $\mathcal{F}$ is a pure graded perverse sheaf. But we have a complete classification of pure graded perverse sheaves given in \cref{thm:explicit_description_Perv_gr_c_w=nu}. The desired result then follows from the corresponding statement about Frobenius weight under smooth pullbacks, using compatibility between functoriality of graded sheaves and mixed sheaves, \cref{thm:compatibility_functoriality_mixed_graded_non_mixed}.
\end{proof}

\begin{prop}
\label{prop:artinian_simple_objects}
The category $\Perv_{\gr, c}(\mathcal{Y})$ is Artinian and Noetherian. Simple graded perverse sheaves are necessarily pure and moreover, $\mathcal{F} \in \Perv_{\gr, c}(\mathcal{Y})$ is simple if and only if $\oblv_\gr(\mathcal{F}) \in \Perv_c(\mathcal{Y})$ is simple.
\end{prop}
\begin{proof}
As $\oblv_\gr: \Shv_{\gr, c}(\mathcal{Y}) \to \Shv_c(\mathcal{Y})$ is $t$-exact by \cref{prop:t-exactness} and conservative by \cref{lem:oblv_gr_conservative}, Artinian-ness and Noetherian-ness of $\Perv_{\gr, c}(\mathcal{Y})$ follow from the same properties of $\Perv_c(\mathcal{Y})$.

Simple objects are necessarily pure since otherwise, the weight filtration will provide a non-trivial filtration.

Now, let $\mathcal{F} \in \Perv_{\gr, c}(\mathcal{Y})$. By conservativity of $\oblv_\gr$, \cref{lem:oblv_gr_conservative}, $\mathcal{F}$ is simple if $\oblv_\gr(\mathcal{F})$ is. Conversely, suppose $\mathcal{F}$ is simple, we would like to show that $\oblv_\gr(\mathcal{F})$ is also simple. As seen above, $\mathcal{F}$ is necessarily pure. But now, the desired statement follows from the complete description of pure graded perverse sheaves given in \cref{thm:explicit_description_Perv_gr_c_w=nu}.
\end{proof}

We also have a version of the decomposition theorem of~\cite{beilinson_faisceaux_2018,sun_decomposition_2012} in the graded setting.

\begin{thm}
\label{thm:decomposition_theorem}
Let $\mathcal{F} \in \Shv_{\gr, c}(\mathcal{Y})^{\weightheart}$. Then
\[
	\mathcal{F} \simeq \bigoplus_i \tH{p}^i(\mathcal{F})[-i] \simeq \bigoplus_i \bigoplus_k \mathcal{G}_{ik}[-i] \teq\label{eq:decomposition_theorem_graded_sheaves}
\]
where the direct sums are finite and $\mathcal{G}_{ik}$ are simple objects in $\Perv_{\gr, c}(\mathcal{Y})^{w=i}$. 

Moreover, if $\mathcal{F} \simeq \gr \mathcal{F}'$ for some $\mathcal{F}' \in \Shv_{\mixed, c}(\mathcal{Y}_n)^{w=0}$, for some $n\in \mathbb{Z}_{>0}$. Then, 
\[
	\oblv_\gr(\mathcal{F}) \simeq \bigoplus_i \bigoplus_k \oblv_\gr \mathcal{G}_{ik}[-i] \teq\label{eq:compatibility_with_usual_decomposition_thm}
\]
where $\mathcal{G}_{ik}$ are simple perverse sheaves on $\mathcal{Y}$. In other words, the decomposition above is compatible with the usual decomposition theorem of~\cite{beilinson_faisceaux_2018,sun_decomposition_2012}.
\end{thm}
\begin{proof}
The first equivalence of \cref{eq:decomposition_theorem_graded_sheaves} is by \cref{rmk:transversal_weight_t_structures}.\ref{item:decomposition_theorem}. The second equivalence of \cref{eq:decomposition_theorem_graded_sheaves} is by the complete description of simple graded perverse sheaves given in \cref{thm:explicit_description_Perv_gr_c_w=nu}. Finally, \cref{eq:compatibility_with_usual_decomposition_thm} is a consequence of \cref{eq:decomposition_theorem_graded_sheaves,prop:artinian_simple_objects}.
\end{proof}

\begin{rmk}
The decomposition in \cref{thm:decomposition_theorem} above remembers the weights and thus could be thought of as an enhancement of the usual decomposition theorem.
\end{rmk}

\subsubsection{Intermediate extensions}
Intermediate extensions are one of the main reasons that make perverse sheaves special. The same construction can be applied to graded perverse sheaves as well. Let $j: \mathcal{U} \to \mathcal{X}$ be an open embedding in $\Stk_k$ and $\mathcal{F} \in \Perv_{\gr, c}(U)$. Then, we can form
\[
	j_{!*} \mathcal{F} = \Im(\tH{p}^0(j_! \mathcal{F})\to \tH{p}^0(j_* \mathcal{F})).
\]
Due to $t$-exactness of $\oblv_\gr$ and $\gr$, and the compatibility between the functoriality of graded sheaves and that of (mixed) sheaves, $j_{!*}$ is also compatible with the usual intermediate extension functor for (mixed) sheaves. Namely,
\[
	\oblv_\gr \circ j_{!*} \simeq j_{!*} \circ \oblv_\gr \qquad\text{and}\qquad j_{!*} \circ \gr \simeq \gr  \circ j_{!*}.
\]

Moreover, standard properties of the intermediate extension functor such as purity preserving and simplicity preserving etc. also hold in the graded setting. Moreover, these properties can be obtained easily from corresponding results in the theory of (mixed) sheaves.

\subsubsection{Simple graded perverse sheaves and graded intersection complex}
Suppose $\mathcal{U}$ is smooth and $\mathcal{L} \in \Perv_{\gr, c}(\mathcal{U})$ is a graded local system, i.e., $\oblv_\gr(\mathcal{L})$ is a local system on $\mathcal{U}$, appropriately cohomologically shifted to make it perverse. Then, we also write $\IC_\gr(\mathcal{L}) \defeq j_{!*} \mathcal{L}$ to denote the graded intersection complex. 

More generally, when $\mathcal{U}$ is locally closed in $\mathcal{X}$ with closure $\lbar{\mathcal{U}}$, and $\mathcal{L}$ is as before, we also write $\IC_\gr(\mathcal{L}) \defeq \bar{j}_* j_{!*} \mathcal{L}$, where $j$ and $\bar{j}$ fit into the following diagram
\[
	\mathcal{U} \xrightarrow{j} \lbar{\mathcal{U}} \xrightarrow{\bar{j}} \mathcal{X}.
\]

\cref{prop:artinian_simple_objects} implies that $\mathcal{L}$ is irreducible if and only if $\oblv_\gr(\mathcal{L})$ is. We call such an object an irreducible graded local system. More generally, \cref{thm:explicit_description_Perv_gr_c_w=nu} implies the following expected statement.

\begin{prop}
All simple perverse sheaves $\Perv_{\gr, c}(\mathcal{X})$ are of the form $\IC_\gr(\mathcal{L})$ where $\mathcal{L} \in \Perv_{\gr, c}(\mathcal{U})$ is an irreducible graded local system, appropriately shifted.
\end{prop}

\subsection{Mixed geometry}
\label{subsec:mixed_geometry}
As mentioned in the introduction, our theory of graded sheaves is designed to give a uniform construction of mixed categories in the sense of~\cites{beilinson_koszul_1996,rider_formality_2013}. In this subsection, we will explain why this is the case. More precisely, \cref{prop:gr_shv_grading_ala_BGS,prop:gr_shv_mixed_version_ala_Rider} state that our construction indeed provides a mixed version in the sense of Beilinson--Ginzburg--Soergel and Rider. Moreover, \cref{prop:criterion_weight_heart_gr_sheaves_classical} shows that under some purity condition that is satisfied in those situations considered classically, our construction gives the same answers as those previously constructed.

\subsubsection{Graded sheaves as ``mixed versions''}
For $\mathcal{Y} \in \Stk_k$, we let $\Shv_{\infty, c}(\mathcal{Y})$ be the smallest full $\DG$-subcategory of $\Shv_{c}(\mathcal{Y})$ containing the essential images under pullbacks of $\Shv_{\mixed, c}(\mathcal{Y}_m)$ for all $k_m$-forms $\mathcal{Y}_m$ of $\mathcal{Y}$, for all $m$. Moreover, let $\Pur_c(\mathcal{Y})$ be the additive full subcategory of $\Shv_c(\mathcal{Y})$ consisting of semi-simple complexes $\mathcal{F} \simeq \bigoplus_i \IC_i[k_i]$ where $\IC_i$'s are simple perverse sheaves. Similarly, we let $\Pur_{\infty, c}(\mathcal{Y})$ be the full subcategory of $\Pur_c(\mathcal{Y})$ consisting of those $\mathcal{F}$ such that $\IC_i$'s come from $\Perv_{\mixed, c}(\mathcal{Y}_n)$ for some $n$. We note that when $\mathcal{Y}$ is a finite orbit stack, $\Shv_{\infty, c}(\mathcal{Y})$ and $\Pur_{\infty, c}(\mathcal{Y})$ coincide with $\Shv_c(\mathcal{Y})$ and $\Pur_c(\mathcal{Y})$, respectively.\footnote{The assertion still holds in the case of ind-finite orbit stacks. However, strictly speaking, we have only discussed the finite type situation in this paper.}

We have the following characterization of $\Shv_{\infty, c}(\mathcal{Y})$.
\begin{prop}
\label{prop:characterization_Shv_infty}
Let $\mathcal{F} \in \Shv_c(\mathcal{Y})$. Then the following conditions are equivalent
\begin{myenum}{(\roman*)}
	\item \label{enum:prop_characterization_Shv_infty:defn} $\mathcal{F} \in \Shv_{\infty, c}(\mathcal{Y})$;
	\item \label{enum:prop_characterization_Shv_infty:simple_constituents} the simple constituents of $\pHo^i(\mathcal{F})$ belongs to $\Pur_{\infty, c}(\mathcal{Y})$, i.e., they are simple perverse sheaves coming from $\Shv_{\mixed, c}(\mathcal{Y}_n)$ for some $n$.
\end{myenum}
As a consequence, $\Shv_{\infty, c}(\mathcal{Y})$ is closed under the perverse truncations of $\Shv_c(\mathcal{Y})$, and hence, it inherits the perverse $t$-structure on $\Shv_c(\mathcal{Y})$.
\end{prop}
\begin{proof}
Assuming \ref{enum:prop_characterization_Shv_infty:simple_constituents} is satisfied. Then, we can build $\mathcal{F}$ from successive extensions of simple $\IC$'s coming from $\Shv_{\mixed, c}(\mathcal{Y}_n)$ for some $n$. Thus, $\mathcal{F} \in \Shv_{\infty, c}(\mathcal{Y})$ by definition.

Conversely, let $\mathcal{F} \in \Shv_{\infty, c}(\mathcal{Y})$. Using \cref{prop:simple_perv_are_pushforward} and the fact that the pullback functor $\Shv_{\mixed, c}(\mathcal{Y}_n) \to \Shv_c(\mathcal{Y})$ is $t$-exact with respect to the perverse $t$-structure, we know that the essential image of $\Shv_{\mixed, c}(\mathcal{Y}_n)$ satisfies \ref{enum:prop_characterization_Shv_infty:simple_constituents}. We conclude by observing that the condition \ref{enum:prop_characterization_Shv_infty:simple_constituents} is closed under finite direct sums, shifts, and cones.
\end{proof}

\begin{cor}
$\Pur_{\infty, c}(\mathcal{Y}) = \Pur_c(\mathcal{Y}) \cap \Shv_{\infty, c}(\mathcal{Y})$ as full subcategories of $\Shv_c(\mathcal{Y})$.
\end{cor}
\begin{proof}
This follows directly from \cref{prop:characterization_Shv_infty}.
\end{proof}

\begin{cor}
\label{cor:Shv_infty_c_idempotent_complete}
The category $\Shv_{\infty, c}(\mathcal{Y})$ is idempotent complete.
\end{cor}
\begin{proof}
Since $\Shv_{\infty, c}(\mathcal{Y})$ has a bounded $t$-structure, it is idempotent complete, by~\cite[Cor. 2.14]{antieau_k-theoretic_2019}.
\end{proof}

We have yet another characterization of $\Shv_{\infty, c}(\mathcal{Y})$.

\begin{prop}
\label{prop:2nd_characterization_Shv_infty}
The functor $\oblv_\gr: \Shv_{\gr, c}(\mathcal{Y}) \to \Shv_c(\mathcal{Y})$ factors through $\Shv_{\infty, c}(\mathcal{Y})$. Moreover, the resulting functor induces an equivalence of category
\[
	\Shv_{\gr, c}(\mathcal{Y}) \otimes_{\Vect^{\gr, c}} \Vect^c \xrightarrow{\simeq} \Shv_{\infty, c}(\mathcal{Y}).
\]
\end{prop}
\begin{proof}
By definition, $\Shv_{\gr, c}(\mathcal{Y})$ is generated as an idempotent complete $\DG$-category by the essential image of $\Shv_{\mixed, c}(\mathcal{Y}_m)$ for some (in fact, any/all) $\mathbb{F}_{q^m}$-form $\mathcal{Y}_m$ of $\mathcal{Y}$ for some (in fact, any/all) $m$. Thus, $\oblv_\gr$ factors through $\Shv_{\infty, c}(\mathcal{Y})$ and we obtain a functor $\Shv_{\gr, c}(\mathcal{Y}) \to \Shv_{\infty, c}(\mathcal{Y})$. This induces a functor
\[
	\Shv_{\gr, c}(\mathcal{Y}) \otimes_{\Vect^{\gr,c}} \Vect^{c} \to \Shv_{\infty, c}(\mathcal{Y}),
\]
which is fully faithful, see also \cref{eq:oblv_gr_conservative}. Since both the source and target are generated, as idempotent complete $\DG$-categories, by the same collection of objects, the resulting functor is an equivalence of categories and the proof concludes.
\end{proof}

\begin{prop}
\label{prop:gr_shv_mixed_version_ala_Rider}
$\Shv_{\gr, c}(\mathcal{Y})$ is a \emph{mixed version} of $\Shv_{\infty, c}(\mathcal{Y})$  in the sense of~\cite[Defn. 4.2]{rider_formality_2013} with a degree $1$ Tate twist $\lrangle{1}$.
\end{prop}
\begin{proof}
	The first condition of~\cite[Defn. 4.2]{rider_formality_2013} concerning the $t$-exactness of the autoequivalence $\lrangle{1}$ and how $\lrangle{1}$ changes the weight is automatic (and in fact already holds for mixed sheaves).

	Next, by \cref{prop:t-exactness}, $\oblv_{\gr}$ is a $t$-exact functor. It is also faithful, by \cref{rmk:unmixed_Hom_direct_sum_gr_Hom}. Moreover, by construction, $\Shv_{\infty, c}(\mathcal{Y})$ is generated by the essential image of $\Shv_{\gr, c}(\mathcal{Y})$ as a triangulated category. This gives us the second condition of~\cite[Defn. 4.2]{rider_formality_2013}.
	
	For the third condition, observe that we have a natural equivalence by construction
	\[
		\oblv_\gr(\mathcal{F}\lrangle{1}) \xrightarrow[\simeq]{\varepsilon} \oblv_\gr(\mathcal{F}), \qquad\text{for all } \mathcal{F} \in \Shv_\gr(\mathcal{Y}). 
	\]
	
	Lastly, \cref{rmk:unmixed_Hom_direct_sum_gr_Hom} gives us a direct sum decomposition required in~\cite[Eq. (4.3)]{rider_formality_2013}.
\end{proof}

\begin{rmk}
The results above have natural variations (with exactly the same proofs) where instead of $\Shv_{\gr, c}(\mathcal{Y})$ and $\Shv_{\infty, c}(\mathcal{Y})$, we consider the smallest full $\DG$-subcategory of $\Shv_{\gr,c}(\mathcal{Y})$ and $\Shv_c(\mathcal{Y})$ generated by a certain collection of objects coming from $\Shv_{\mixed, c}(\mathcal{Y}_m)$ for some $m$'s as long as the resulting categories are closed under the perverse truncation and Tate twist.
\end{rmk}

\subsubsection{}
Let $\Perv_{\infty, c}(\mathcal{Y}) \defeq \Shv_{\infty, c}(\mathcal{Y})^{\theart}$ denote the perverse $t$-heart. Then, $\oblv_\gr$ restricts to an exact functor
\[
	\oblv_\gr: \Perv_{\gr, c}(\mathcal{Y}) \to \Perv_{\infty, c}(\mathcal{Y})
\]
which is still faithful.

\begin{prop}
\label{prop:gr_shv_grading_ala_BGS}
The functor $\oblv_\gr: \Perv_{\gr, c}(\mathcal{Y}) \to \Perv_{\infty, c}(\mathcal{Y})$ together with $\varepsilon$ realizes $\Perv_{\gr, c}(\mathcal{Y})$ as a grading on $\Perv_{\infty, c}(\mathcal{Y})$ with a degree $1$ Tate twist $\lrangle{1}$ in the sense of~\cite[Defn. 4.3.1.(1)]{beilinson_koszul_1996}.
\end{prop}

This result is slightly more subtle than \cref{prop:gr_shv_mixed_version_ala_Rider}. Before proving it, we need some preparation. 

\subsubsection{}
By $\Ind$-extending, we obtain perverse $t$-structures on $\Shv_{\gr}(\mathcal{Y})^{\ren}$ and
\[
	\Shv_{\infty}(\mathcal{Y})^{\ren} \defeq \Ind(\Shv_{\infty, c}(\mathcal{Y})) \simeq \Shv_{\gr}(\mathcal{Y})^{\ren} \otimes_{\Vect^\gr} \Vect,
\]
where the last equivalence is due to \cref{prop:2nd_characterization_Shv_infty}. The hearts of these $t$-structures are compactly generated abelian categories $\Perv_{\gr}(\mathcal{Y}) \defeq \Ind(\Perv_{\gr,c}(\mathcal{Y}))$ and $\Perv_{\infty}(\mathcal{Y}) \defeq \Ind(\Perv_{\infty,c}(\mathcal{Y}))$, respectively. Moreover, by construction, these $t$-structures are compatible with filtered colimits. In particular, filtered colimits are exact in $\Perv_\gr(\mathcal{Y})$ and $\Perv_\infty(\mathcal{Y})$.

\subsubsection{}
Consider the continuous $\Vect^\gr$-linear adjoint pair
\[
	\oblv_\gr: \Vect^\gr \rightleftarrows \Vect: \const
\]
where $\const$ sends an object $V\in \Vect$ to the constant graded object where each graded component is $V$. Namely,
\[
	\const(V) \defeq \bigoplus_k V\lrangle{k} \in \Vect^\gr.
\]
This induces an eponymous continuous adjoint pair
\[
	\oblv_{\gr} :\Shv_\gr(\mathcal{Y})^{\ren} \simeq \Shv_{\gr}(\mathcal{Y})^{\ren} \otimes_{\Vect^{\gr}} \Vect^{\gr} \rightleftarrows \Shv_{\gr}(\mathcal{Y})^{\ren} \otimes_{\Vect^{\gr}} \Vect \simeq \Shv_{\infty}(\mathcal{Y})^{\ren}: \const,
\]
By construction,
\[
	\const(\mathcal{F}\boxtimes V) \simeq \bigoplus_k \mathcal{F} \otimes V\lrangle{k}, \qquad \text{for all } \mathcal{F} \in \Shv_{\gr}(\mathcal{Y})^{\ren}, V \in \Vect.
\]
In particular,
\[
	\const(\oblv_\gr(\mathcal{F})) \simeq \bigoplus_k \mathcal{F}\lrangle{k}, \qquad \text{for all } \mathcal{F} \in \Shv_{\gr}(\mathcal{Y})^{\ren}. \teq\label{eq:const_oblv_gr}
\]

By construction, $\oblv_\gr$ is $t$-exact and hence, induces a functor
\[
	\oblv_\gr: \Perv_\gr(\mathcal{Y}) \to \Perv_\infty(\mathcal{Y}).
\]

\begin{lem} \label{lem:oblv_gr_const_adjoint_perv}
$\const$ is also $t$-exact, and hence, induces a functor
\[
	\const: \Perv_\infty(\mathcal{Y}) \to \Perv_\gr(\mathcal{Y}).
\]
In particular, the adjoint pair $\oblv_\gr \dashv \const$ restricts to an exact adjoint pair
\[
	\oblv_\gr: \Perv_\gr(\mathcal{Y}) \rightleftarrows \Perv_\infty(\mathcal{Y}): \const.
\]
\end{lem}
\begin{proof}
	We will show that for any $\mathcal{F} \in \Perv_\infty(\mathcal{Y}) \subseteq \Shv_{\infty}(\mathcal{Y})^{\ren}$, $\const(\mathcal{F}) \in \Perv_\gr(\mathcal{Y})$. Let $\mathcal{C}$ be the full subcategory of $\Perv_\infty(\mathcal{Y})$ spanned by objects $\mathcal{F}$ such that $\const(\mathcal{F}) \in \Perv_\gr(\mathcal{Y})$. We will show that $\mathcal{C} = \Perv_\infty(\mathcal{Y})$.
	
	By \cref{eq:const_oblv_gr}, $\mathcal{C}$ contains the essential image of $\Perv_\gr(\mathcal{Y})$ under $\oblv_\gr$. By general nonsense, $\mathcal{C}$ is closed under extensions. Moreover, since $\const$ is continuous and the $t$-structures involved are all compatible with filtered colimits, $\mathcal{C}$ is closed under filtered colimits. But by \cref{prop:2nd_characterization_Shv_infty}, any object of $\Perv_\infty(\mathcal{Y})$ is a filtered colimit of extensions of objects in the essential image of $\Perv_{\gr,c}(\mathcal{Y})$. Thus, $\mathcal{C} = \Perv_\infty(\mathcal{Y})$ and we are done.
\end{proof}

We are now ready to prove \cref{prop:gr_shv_grading_ala_BGS}.

\begin{proof}[Proof of \cref{prop:gr_shv_grading_ala_BGS}] \label{proof:prop:gr_shv_grading_ala_BGS}
Taking \cref{prop:gr_shv_mixed_version_ala_Rider} into account, it remains to show that for $\mathcal{F}, \mathcal{G} \in \Perv_{\gr, c}(\mathcal{Y})$ and all $i\in \mathbb{N}$, the natural map $\oblv_{\gr}(\mathcal{G}\lrangle{1}) \xrightarrow[\simeq]{\varepsilon} \oblv_{\gr}(\mathcal{G})$ induces an isomorphism
\[
	\bigoplus_{k\in \mathbb{Z}} \Ext^i_{\Perv_{\gr,c}(\mathcal{Y})}(\mathcal{F}, \mathcal{G}\lrangle{k}) \simeq \Ext^i_{\Perv_{\infty, c}(\mathcal{Y})}(\oblv_\gr(\mathcal{F}), \oblv_\gr(\mathcal{G})). \teq\label{eq:unmixed_exts}
\]
This will be the result of a certain adjoint pair between the derived categories of $\Perv_{\gr}(\mathcal{Y})$ and $\Perv_{\infty}(\mathcal{Y})$ along with a certain almost compactness property.

The exact adjoint pair of \cref{lem:oblv_gr_const_adjoint_perv} induces a $t$-exact adjoint pair 
\[
	\oblv_{\gr}^D: D(\Perv_{\gr}(\mathcal{Y})) \rightleftarrows D(\Perv_{\infty}(\mathcal{Y})): \const^D,
\]
whose restriction to the hearts recovers the original adjoint pair. Here, $D(-)$ denotes the unbounded derived category as constructed in, for example,~\cite[\S1.3.5]{lurie_higher_2017}. By construction, for any $\mathcal{G} \in \Perv_\gr(\mathcal{Y}) \subset D(\Perv_\gr(\mathcal{Y}))$, 
\[
	\const^D(\oblv_\gr^D(\mathcal{G}[i])) \simeq \const^D(\oblv_\gr^D(\mathcal{G}))[i] \simeq \bigoplus_k \mathcal{G}\lrangle{k}[i] \in \Perv_\infty(\mathcal{Y}) \subseteq D(\Perv_{\infty}(\mathcal{Y})).
\]
Thus, for any $\mathcal{F} \in \Perv_{\gr, c}(\mathcal{Y})$ and $\mathcal{G} \in \Perv_\infty(\mathcal{Y})$, we have
\begin{align*}
	&\alignsep\Ext^i_{\Perv_\infty(\mathcal{Y})}(\oblv_\gr(\mathcal{F}), \oblv_\gr(\mathcal{G})) \\
	&\simeq \Ho^i(\cHom_{D(\Perv_\infty(\mathcal{Y}))}(\oblv_\gr^D(\mathcal{F}), \oblv_\gr^D(\mathcal{G}))) \\
	&\simeq \Ho^i(\cHom_{D(\Perv_\gr(\mathcal{Y}))}(\mathcal{F}, \const^D\oblv_\gr^D(\mathcal{G}))) \\
	&\simeq \Ho^i(\cHom_{D(\Perv_\gr(\mathcal{Y}))}(\mathcal{F}, \bigoplus_k \mathcal{G}\lrangle{k})) \\
	&\simeq \bigoplus_k \Ho^i(\cHom_{D(\Perv_\gr(\mathcal{Y}))}(\mathcal{F}, \mathcal{G}\lrangle{k})) \\
	&\simeq \bigoplus_k \Ext^i_{\Perv_\gr(\mathcal{Y})}(\mathcal{F}, \mathcal{G}\lrangle{k}).
	\end{align*}
	Here, the second to last equivalence is due to the fact that $\mathcal{F}$ is almost compact in $D(\Perv_\gr(\mathcal{Y}))^{\pervt\leq 0}$. Indeed, by \cite[Example C.6.9.5, Remark C.6.8.10]{lurie_spectral_2018} and \cref{prop:artinian_simple_objects}, $D(\Perv_\gr(\mathcal{Y}))^{\pervt\leq 0}$ is locally Noetherian. Then, the fact that $\mathcal{F}$ is almost compact follows from \cite[Prop. C.6.9.3]{lurie_spectral_2018}. The proof thus concludes.
\end{proof}

\subsection{$\Hom$-purity and the usual construction via $\Ch^b$} 
In the above, we discussed how our construction provides the correct answer in the sense of~\cites{beilinson_koszul_1996,rider_formality_2013}. We will now explain how our construction in fact also recovers, and hence, generalizes classical constructions.

Before the current paper (with the exception of~\cites{soergel_perverse_2018,soergel_equivariant_2018}), mixed versions are constructed by hand as the $\DG$-/homotopy category of bounded chain complexes in a chosen collection of semi-simple objects. In our language, the mixed version was defined to be $\Ch^b(\ho\Pur_{\infty, c}(\mathcal{Y}))$ where $\ho\Pur_{\infty, c}(\mathcal{Y})$ is the underlying homotopy category of $\Pur_{\infty, c}(\mathcal{Y})$. However, this is only sensible when the mixed $\cHom$ complexes between those objects satisfy some purity condition. What we did above amounts to saying that a mixed version in fact always exists, free from any extra purity condition. 

However, as we will see in \cref{prop:criterion_weight_heart_gr_sheaves_classical} below, when a purity condition is satisfied, the weight heart $\Shv_{\gr, c}(\mathcal{Y})^{\weightheart}$ is classical and hence, the weight complex functor of \cref{subsubsec:weight_complex_functor} yields an equivalence of (weight) $\DG$-categories $\Shv_{\gr, c}(\mathcal{Y}) \simeq \Ch^b(\Shv_{\gr, c}(\mathcal{Y})^{\weightheart})$. Moreover, in this case, $\Shv_{\gr, c}(\mathcal{Y})^{\weightheart} \simeq \ho\Pur_{\infty, c}(\mathcal{Y})$ and hence, $\Shv_{\gr, c}(\mathcal{Y}) \simeq \Ch^b(\ho\Pur_{\infty, c}(\mathcal{Y}))$, recovering classical constructions.

\subsubsection{}
Before stating our result, we note the following fact. By the decomposition theorem for graded sheaves, \cref{thm:decomposition_theorem}, we have the following factorization
\[
\begin{tikzcd}[sep=small]
	\Shv_{\gr, c}(\mathcal{Y})^{\weightheart} \ar{dr}[swap]{\oblv_\gr^{\weightheart}} \ar{rr} &&  \Shv_{\infty, c}(\mathcal{Y}) \\
	& \Pur_{\infty, c}(\mathcal{Y}) \ar[hookrightarrow]{ur}
\end{tikzcd}
\]
Moreover, from the definition of $\Pur_{\infty, c}(\mathcal{Y})$ and the fact that any simple mixed perverse sheaf on $\mathcal{Y}_n$ is necessarily pure, we know that $\oblv_\gr^{\weightheart}$ is essentially surjective. This justifies the use of the notation $\Pur_{\infty, c}$.

\begin{prop}
\label{prop:criterion_weight_heart_gr_sheaves_classical}
Let $\mathcal{Y} \in \Stk_k$ and $\mathcal{Y}_n \in \Stk_{k_n}$ be any $k_n$-form of $\mathcal{Y}$. Then, the following are equivalent:
\begin{myenum}{(\roman*)}
	\item \label{enum:prop_crit_classical_weightheart:first} $\Shv_{\gr, c}(\mathcal{Y})^{\weightheart}$ is classical.
	\item \label{enum:prop_crit_classical_weightheart:defn} For any $\mathcal{F}^\gr, \mathcal{G}^\gr \in \Shv_{\gr, c}(\mathcal{Y})^{\weightheart}$, $\cHom_{\Shv_\gr(\mathcal{Y})^\ren}(\mathcal{F}^\gr, \mathcal{G}^\gr) \in \Vect$ concentrates in cohomological degree $0$.
	\item \label{enum:prop_crit_classical_weightheart:pure_grHom_anything} For any $\mathcal{F}^\gr, \mathcal{G}^\gr \in \Shv_{\gr, c}(\mathcal{Y})^{\weightheart}$, $\cHom_{\Shv_\gr(\mathcal{Y})^\ren}^\gr(\mathcal{F}^\gr, \mathcal{G}^\gr) \in \Vect^\gr$ is pure of weight $0$, i.e. it concentrates in the diagonal or equivalently, the $k$-graded component concentrates in cohomological degree $k$. 
	\item \label{enum:prop_crit_classical_weightheart:comparing_with_hPur} The natural functor $\ho(\oblv_\gr^{\weightheart}): \ho \Shv_{\gr, c}(\mathcal{Y})^{\weightheart} \to \ho\Pur_{\infty, c}(\mathcal{Y})$ is an equivalence of categories.
	\item \label{enum:prop_crit_classical_weightheart:pure_grHom_simple_perv} For simple $\IC^\gr_1, \IC^\gr_2 \in \Perv_{\gr, c}(\mathcal{Y})^{\weightheart} = \Perv_{\gr, c}(\mathcal{Y})^{w=0}$, the object $\cHom^\gr_{\Shv_\gr(\mathcal{Y})^\ren}(\IC_1^\gr, \IC_2^\gr)$ is pure of weight $0$ as an object in $\Vect^\gr$.
	\item \label{enum:prop_crit_classical_weightheart:pure_mHom_anything} For any $\mathcal{F}_n, \mathcal{G}_n \in \Shv_{\mixed, c}(\mathcal{Y}_n)$ that is pure of weight $0$, $\cHom^\mixed_{\Shv_\mixed(\mathcal{Y}_n)^\ren}(\mathcal{F}_n, \mathcal{G}_n) \in \Shv_\mixed(\pt_n)$ is pure of weight $0$.
	\item \label{enum:prop_crit_classical_weightheart:pure_mHom_simple_perv} For any $\IC_1, \IC_2 \in \Perv_{\mixed, c}(\mathcal{Y}_n)^{w=0}$ that are simple, $\cHom^\mixed_{\Shv_\mixed(\mathcal{Y}_n)^\ren}(\IC_1, \IC_2) \in \Shv_\mixed(\pt_n)$ is pure of weight $0$.
\end{myenum}
\end{prop}
\begin{proof}
By definition and the fact that $\Vect$-enriched $\cHom$'s between elements in the weight heart can only concentrate in non-positive degrees (see \cref{subsubsec:weight_heart}), we have \ref{enum:prop_crit_classical_weightheart:first} $\Leftrightarrow$ \ref{enum:prop_crit_classical_weightheart:defn}. Moreover, it is easy to see that  
\ref{enum:prop_crit_classical_weightheart:pure_grHom_anything} $\Rightarrow$ 
\ref{enum:prop_crit_classical_weightheart:defn},
\ref{enum:prop_crit_classical_weightheart:pure_grHom_simple_perv}, and 
\ref{enum:prop_crit_classical_weightheart:pure_mHom_anything} (the last one, for example, follows from \cref{prop:hom_graded_sheaves}); and 
\ref{enum:prop_crit_classical_weightheart:pure_mHom_anything}
$\Rightarrow$
\ref{enum:prop_crit_classical_weightheart:pure_mHom_simple_perv}. 
It remains to prove that 
\ref{enum:prop_crit_classical_weightheart:defn}
$\Rightarrow$
\ref{enum:prop_crit_classical_weightheart:pure_grHom_anything},
\ref{enum:prop_crit_classical_weightheart:pure_grHom_simple_perv}
$\Rightarrow$
\ref{enum:prop_crit_classical_weightheart:pure_grHom_anything}, 
\ref{enum:prop_crit_classical_weightheart:pure_mHom_simple_perv} 
$\Rightarrow$
\ref{enum:prop_crit_classical_weightheart:pure_grHom_simple_perv}, and
\ref{enum:prop_crit_classical_weightheart:pure_grHom_anything}
$\Leftrightarrow$
\ref{enum:prop_crit_classical_weightheart:comparing_with_hPur}.

For \ref{enum:prop_crit_classical_weightheart:defn} $\Rightarrow$ \ref{enum:prop_crit_classical_weightheart:pure_grHom_anything}, we note that if $\mathcal{G}^\gr \in \Shv_{\gr, c}(\mathcal{Y})^{\weightheart}$ then so is $\mathcal{G}^\gr \lrangle{k}[k]$ for any $k\in \mathbb{Z}$. Thus, for any $\mathcal{F}^\gr, \mathcal{G}^\gr \in \Shv_{\gr, c}(\mathcal{Y})^{\weightheart}$, assuming \ref{enum:prop_crit_classical_weightheart:defn}, we know that
\[
	\cHom_{\Shv_\gr(\mathcal{Y})^\ren}(\mathcal{F}^\gr, \mathcal{G}^\gr\lrangle{k}[k]) \simeq \cHom^\gr_{\Shv_\gr(\mathcal{Y})^\ren}(\mathcal{F}^\gr, \mathcal{G}^\gr)_k[k] \in \Vect
\]
concentrates in cohomological degree $0$ for any $k$. \ref{enum:prop_crit_classical_weightheart:pure_grHom_anything} thus follows.

We will now prove that \ref{enum:prop_crit_classical_weightheart:pure_grHom_simple_perv} $\Rightarrow$ \ref{enum:prop_crit_classical_weightheart:pure_grHom_anything}. By the decomposition theorem for graded sheaves, \cref{thm:decomposition_theorem}, it suffices to prove \ref{enum:prop_crit_classical_weightheart:pure_grHom_anything} when $\mathcal{F}^\gr$ and $\mathcal{G}^\gr$ are of the form $\IC_i^\gr\lrangle{k_i}[k_i]$ where $\IC^\gr_i \in \Perv_{\gr, c}(\mathcal{Y})^{\weightheart}$ are simple, $i\in \{1, 2\}$. But by the hypothesis \ref{enum:prop_crit_classical_weightheart:pure_grHom_simple_perv}, $\cHom^\gr_{\Shv_\gr(\mathcal{Y})^\ren}(\IC_1^\gr, \IC_2^\gr) \in \Vect^\gr$ is pure of weight $0$, and hence, so is 
\begin{align*}
	&\alignsep\cHom^\gr_{\Shv_\gr(\mathcal{Y})^\ren}(\IC_1^\gr\lrangle{k_1}[k_1], \IC_2^\gr\lrangle{k_2}[k_2]) \\
	&\simeq \cHom^\gr_{\Shv_\gr(\mathcal{Y})^\ren}(\IC_1^\gr, \IC_2^\gr)\lrangle{k}[k] \in \Vect^\gr, \quad k = k_2 - k_1.
\end{align*}
We thus obtain \ref{enum:prop_crit_classical_weightheart:pure_grHom_anything}.

We will show that \ref{enum:prop_crit_classical_weightheart:pure_mHom_simple_perv} $\Rightarrow $\ref{enum:prop_crit_classical_weightheart:pure_grHom_simple_perv}. Assuming \ref{enum:prop_crit_classical_weightheart:pure_mHom_simple_perv}, observe that for any $\IC_1, \IC_2$ as in \ref{enum:prop_crit_classical_weightheart:pure_mHom_simple_perv}, by \cref{prop:hom_graded_sheaves}, $\cHom^\gr_{\Shv_\gr(\mathcal{Y})^\ren}(\gr(\IC_1), \gr(\IC_2)) \in \Vect^\gr$ is pure of weight $0$. But now, by \cref{eq:defn_pure_gr_perv_sheaves,thm:explicit_description_Perv_gr_c_w=nu}, any simple object in $\Perv_{\gr, c}(\mathcal{Y})^{\weightheart}$ is a direct summand of an object of the form $\gr(\IC)$ where $\IC \in \Perv_{\mixed, c}(\mathcal{Y})^{w=0}$ is simple. The proof thus concludes.

Next, we will show that \ref{enum:prop_crit_classical_weightheart:pure_grHom_anything}$\Leftrightarrow$ \ref{enum:prop_crit_classical_weightheart:comparing_with_hPur}. We already saw above that $\ho(\oblv_\gr^{\weightheart})$ is essentially surjective. Thus, \ref{enum:prop_crit_classical_weightheart:comparing_with_hPur} is equivalent to the fact that $\ho(\oblv_\gr^{\weightheart})$ is fully faithful. Let $\mathcal{F}^\gr, \mathcal{G}^\gr \in \Shv_{\gr, c}(\mathcal{Y})^{\weightheart}$ and $\mathcal{F}, \mathcal{G} \in \Pur_{\infty, c}(\mathcal{Y})$ their images. We have,
\begin{align*}
	\Hom_{\ho\Shv_{\gr, c}(\mathcal{Y})^{\weightheart}}(\mathcal{F}^\gr, \mathcal{G}^\gr)
	&\simeq \Ho^0(\cHom_{\Shv_\gr(\mathcal{Y})^\ren}(\mathcal{F}^\gr, \mathcal{G}^\gr)) \\
	&\to \bigoplus_k \Ho^0(\cHom^\gr_{\Shv_\gr(\mathcal{Y})^\ren}(\mathcal{F}^\gr, \mathcal{G}^\gr)_k) \teq\label{eq:embed_degree_0} \\
	&\simeq \Ho^0(\cHom_{\Shv(\mathcal{Y})^\ren}(\mathcal{F}, \mathcal{G})) \tag{\cref{rmk:unmixed_Hom_direct_sum_gr_Hom}} \\
	&\simeq \Hom_{\ho\Pur_{\infty, c}(\mathcal{Y})}(\mathcal{F}, \mathcal{G}),
\end{align*}
where \cref{eq:embed_degree_0} is the embedding to the graded degree $0$ part. Now, \ref{enum:prop_crit_classical_weightheart:comparing_with_hPur} is equivalent to \cref{eq:embed_degree_0} being an equivalence for all $\mathcal{F}^\gr, \mathcal{G}^\gr \in \Shv_{\gr, c}(\mathcal{Y})^{\weightheart}$. Consider the above with $\mathcal{G}^\gr$ replaced by $\mathcal{G}^\gr\lrangle{l}[l]$ for all $l\in \mathbb{Z}$, we see that \cref{eq:embed_degree_0} being an equivalence is equivalent to \ref{enum:prop_crit_classical_weightheart:pure_grHom_anything}.
\end{proof}

\begin{rmk}
Note that $\Pur_{\infty, c}(\mathcal{Y})$, and hence also $\Ch^b(\ho\Pur_{\infty, c}(\mathcal{Y}))$, is endowed with a homological shift functor $[k]$ for any $k \in \mathbb{Z}$. These \emph{inner} homological shifts are not to be confused with the \emph{outer} homological shifts coming from $\Ch^b$. When $\Shv_{\gr, c}(\mathcal{Y})^{\weightheart}$ is classical, \cref{prop:criterion_weight_heart_gr_sheaves_classical} provides an equivalence of weight categories $\Shv_{\gr,c}(\mathcal{Y}) \simeq \Ch^b(\ho\Pur_{\infty, c}(\mathcal{Y}))$. Under this equivalence, simultaneous graded-degree and homological shifts $[k]\lrangle{k}$ on the left-hand side get translated to the \emph{inner} homological shifts $[k]$ on the right-hand side.
\end{rmk}

\section{Hecke categories}
\label{sec:Hecke_categories}
We will now apply the theory developed above to obtain a geometric realization of the $\DG$-category of bounded chain complexes of Soergel bimodules. More precisely, fixing a reductive group $G$ over $k$ and subgroups $T \subset B \subset G$ where $T$ is a maximal torus and $B$ a Borel subgroup, the main result of this section, \cref{thm:geometric_incarnation_Hecke}, states that we have an equivalence of monoidal categories
\[
	\Shv_{\gr, c}(B\backslash G/B) \simeq \Shv_{\gr, c}(BB \times_{BG} BB) \simeq \Ch^b(\SBim_W),
\]
compatible with the weight and monoidal structures on both sides, where the weight structure on $\Ch^b(\SBim_W)$ is given by the ``stupid'' truncation. In particular,
\[
	\Shv_{\gr, c}(B\backslash G/B)^{\weightheart} \simeq \SBim_W.
\]
Here, $\SBim_W$ is the category of Soergel bimodules attached to the Coxeter system coming from $G$ and $W$ is the Weyl group of $G$. The main feature of this section is that known statements about sheaves on $B\backslash G/B$ can be readily applied to yield the desired result.

We note that thanks to \cref{prop:criterion_weight_heart_gr_sheaves_classical}, known results about $\Shv_{\mixed, c}(B_n\backslash G_n/B_n)$ such as those proved in~\cite{bezrukavnikov_koszul_2013} can be applied directly to deduce that $\Shv_{\gr, c}(B\backslash G/B)^{\weightheart}$ is classical and hence
\[
	\Shv_{\gr, c}(B\backslash G/B) \simeq \Ch^b(\ho\Pur_c(B\backslash G/B)) \simeq \Ch^b(\SBim_W),
\]
where the last equivalence has already been proved originally by Soergel in~\cite{soergel_kategorie_1990}. Note that here, $G_n$ is a split form of $G$ over $\Fqn$ and moreover, $\Pur_c$ and $\Pur_{\infty, c}$ coincide since $B\backslash G/B$ is a finite orbit stack. In what follows, however, we take a slightly more direct approach, highlighting which computational input is necessary.

\subsection{Geometric Hecke categories}
\label{sec:geometric_hecke_categories}
We will now study the category of constructible graded sheaves on $B\backslash G/B \simeq BB \times_{BG} BB$.

\subsubsection{Monoidal structure via convolution}
Since $BB\times_{BG} BB$ is of the form $X\times_Y X$, it is naturally an algebra object in $\Corr(\Stk_k)$ where the multiplication map is given by the standard convolution correspondence
\[
\begin{tikzcd}
	BB \times_{BG} BB \times_{BG} BB \ar{d} \ar{r} & (BB\times_{BG} BB) \times (BB\times_{BG} BB) \\
	BB\times_{BG} BB
\end{tikzcd}
\]
Applying $\Shv_{\gr, !}^{\ren, *}$ of \cref{thm:graded_sheaves_correspondence}, we obtain a monoidal structure on $\Shv_\gr(B\backslash G/B)^\ren$. More precisely, $\Shv_\gr(B\backslash G/B)^\ren \in \Alg(\Mod_{\Vect^\gr})$. Note that since the vertical map is proper and the horizontal map is smooth, we can, equivalently, apply $\Shv^{\ren, *}_{\gr, *}$ of \cref{eq:shv_corr_all_sm} to obtain the same monoidal structure on $\Shv_\gr(B\backslash G/B)^\ren$. 

For any $\mathcal{F}_1, \mathcal{F}_2 \in \Shv_\gr(B\backslash G/B)^\ren$, we will use $\mathcal{F}_1 * \mathcal{F}_2 \in \Shv_\gr(B\backslash G/B)^\ren$ to denote the convolution of $\mathcal{F}_1$ and $\mathcal{F}_2$.

\subsubsection{Compatibility with weight structures}
In forming the monoidal structure, we pull along smooth (and representable) morphisms and push along proper morphisms. Thus, the monoidal structure on $\Shv_\gr(B\backslash G/B)^\ren$ restricts to one on $\Shv_{\gr, c}(B\backslash G/B)$ which is also compatible with the weight structure. By \cref{thm:aoki_monoidal_wt_complex_functor}, we obtain a monoidal functor
\[
	\wt: \Shv_{\gr, c}(B\backslash G/B) \to \Ch^b(\ho\Shv_{\gr, c}(B\backslash G/B)^{\weightheart}), \teq\label{eq:wt_complex_functor_for_finite_Hecke_stack}
\]
compatible with the action of $\Vect^{\gr, c}$, i.e., a morphism in $\Alg(\Mod_{\Vect^\gr})$.

\subsubsection{Purity}
We will now show that the weight complex functor $\wt$ of \cref{eq:wt_complex_functor_for_finite_Hecke_stack} is an equivalence of categories.

\begin{prop}
\label{prop:shv_gr_Hecke_stack_classical}
The category $\Shv_{\gr, c}(B\backslash G/B)^{\weightheart}$ is classical, i.e., $\Shv_{\gr, c}(B\backslash G/B)^{\weightheart}$ is equivalent to $\ho\Shv_{\gr, c}(B\backslash G/B)^{\weightheart}$. As a result, the weight complex functor gives an equivalence
\[
	\wt: \Shv_{\gr, c}(B\backslash G/B) \xrightarrow{\simeq} \Ch^b(\Shv_{\gr, c}(B\backslash G/B)^{\weightheart})
\]
as objects in $\Alg(\Mod_{\Vect^{\gr, c}})$, i.e., they are equivalent as monoidal categories and the equivalence is compatible with $\Vect^{\gr, c}$-actions.
\end{prop}
\begin{proof}
The second part follows from the first by \cref{subsubsec:weight_complex_functor}. Now, the hypothesis of \cref{prop:criterion_weight_heart_gr_sheaves_classical}.\ref{enum:prop_crit_classical_weightheart:pure_mHom_anything} is satisfied by~\cite[Lem. 3.1.3 and Lem. 3.1.5.(2)]{bezrukavnikov_koszul_2013}. The proof thus concludes.
\end{proof}

\subsection{Soergel bimodules}
We will now quickly recall the definition of the category of Soergel bimodules. We will make use of the following notation: for any $\mathcal{Y} \in \Stk_k$ with the structure map $\pi: \mathcal{Y} \to \pt$, we use
\[
	\Co^*_\gr(\mathcal{Y}) \defeq \pi_{*, \ren} \Qlbar \in \ComAlg(\Vect^\gr) \teq\label{eq:graded_cohomology}
\]
to denote the \emph{graded cohomology} of $\mathcal{Y}$, where $\pi_{*, \ren}$ is the pushforward functor of graded sheaves. Note that this is just the usual cohomology of $\mathcal{Y}$, except that we turn Frobenius weights to the grading in $\Vect^\gr$.

\subsubsection{}
We let $R^\Rightarrow \defeq \Co^*_\gr(BT) \simeq \Co^*_\gr(BB)$. Explicitly,
\[
	R^\Rightarrow \simeq \Sym(X^*(T)\otimes_{\mathbb{Z}} \Qlbar[-2]\lrangle{-2}),
\]
where $X^*(T)$ is the co-character lattice of $T$. Namely, $R^\Rightarrow$ is a polynomial algebra generated by $\rank X^*(T)$ many variables, put in graded degree $2$ and cohomological degree $2$.

Let $\sh^\Rightarrow$ and $\sh^\Leftarrow$ be monoidal auto-equivalences of $\Vect^\gr$ given by\footnote{Note that the shear functors $\sh^\Rightarrow$ and $\sh^\Leftarrow$ are only monoidal rather than symmetric monoidal.}
\[
	\sh^\Rightarrow(V)_i = V_i[-i] \qquad\text{and}\qquad \sh^\Leftarrow(V)_i = V_i[i], \qquad\text{for all } V \in \Vect^\gr,
\]
Let $R \defeq \sh^{\Leftarrow}(R^\Rightarrow)$. Then, $\sh^{\Leftarrow}$ and $\sh^\Rightarrow$ induce equivalences of monoidal categories
\[
	\sh^\Leftarrow: \BiMod_{R^\Rightarrow}(\Vect^\gr) \rightleftarrows \BiMod_{R}(\Vect^\gr): \sh^\Rightarrow
\]
where $\BiMod$ denotes the category of bimodules.

\subsubsection{}
Let $W$ denote the Weyl group of $G$. Then, $W$ acts on $T$, and hence, on $R^\Rightarrow$ and $R$. The category of Soergel bimodules $\SBim_W$ is a full subcategory of $\BiMod_R(\Vect^{\gr, \heartsuit})$ (and hence, also of $\BiMod_R(\Vect^\gr)$) spanned by finite direct sums of graded degree shifts of direct summands of objects of the form
\[
	R \otimes_{R^{s_1}} R \otimes_{R^{s_2}} \cdots \otimes_{R^{s_k}} R
\]
for any sequence of simple reflections $s_i\in W$. Equivalently, $\SBim_W$ is the smallest full subcategory of $\BiMod_R(\Vect^{\gr, \heartsuit})$ containing $R \otimes_{R^s} R$ for all simple reflections $s\in W$ and is closed under taking tensor products, graded degree shifts, finite direct sums, and direct summands.

We define $\SBim_W^\Rightarrow \defeq \sh^\Rightarrow(\SBim_W)$, a full subcategory of $\BiMod_{R^\Rightarrow}(\Vect^\gr)$. As above, we have mutually inverse equivalences of monoidal categories
\[
	\sh^\Leftarrow: \SBim_W^\Rightarrow \rightleftarrows \SBim_W: \sh^\Rightarrow.
\]
Note also that this implies that $\SBim_W^\Rightarrow$ is also a classical category which is closed under simultaneous shifts $[n]\lrangle{n}$ for any $n\in \mathbb{Z}$.

\subsubsection{}
Finally, we let $\Ch^b(\SBim_W)$ and $\Ch^b(\SBim_W^\Rightarrow)$ denote the corresponding monoidal $\DG$-categories of bounded chain complexes of Soergel bimodules. We have mutually inverse equivalences of monoidal categories
\[
	\sh^\Leftarrow: \Ch^b(\SBim_W^\Rightarrow) \rightleftarrows \Ch^b(\SBim_W): \sh^\Rightarrow.
\]
These functors are easily seen to be compatible with $\Vect^{\gr, c}$-actions.

\subsection{\texorpdfstring{$\Shv_\gr(BG)^\ren$}{Renormalized graded sheaves over BG} as a category of modules}
We will now relate the category $\Shv_{\gr}(BG)^\ren$ and the category $\Mod_{\Co^*_\gr(BG)}(\Vect^\gr)$ of graded modules over its graded cohomology $\Co^*_\gr(BG)$. The material is standard and well-known. In the ungraded setting, a version of this also appeared in~\cite{webster_geometric_2008}.

\subsubsection{Identification of categories}

\begin{prop}
\label{prop:Shv_gr_BG_as_category_of_modules}
Let $G$ be any connected algebraic group and $\pi: BG \to \pt$ the structure map. Then, the functor of taking global sections $\pi_{*, \ren}$ induces an equivalence of symmetric monoidal $\Vect^\gr$-module categories $\pi_{*, \ren}^\enh$ fitting into the following commutative diagram
\[
\begin{tikzcd}[sep=large]
	\Shv_\gr(BG)^\ren \ar[shift left=\arrdisp]{dr}{\pi_{*, \ren}^\enh} \ar[shift left=\arrdisp]{d}{\pi_{*, \ren}} \\
	\Vect^\gr \ar[shift left=\arrdisp]{r}{\Co^*_\gr(BG)\otimes-} \ar[shift left=\arrdisp]{u}{\pi^*_\ren} & \Mod_{\Co^*_\gr(BG)}(\Vect^\gr) \ar[shift left=\arrdisp]{ul} \ar[shift left=\arrdisp]{l}{\oblv_{\Co^*_\gr(BG)}}
\end{tikzcd}
\]
\end{prop}
\begin{proof}
We have a pair of adjoint functors
\[
	\pi^*_\ren: \Vect^\gr \rightleftarrows \Shv_\gr(BG)^\ren: \pi_{*, \ren} \simeq \cHom^\gr_{\Shv_\gr(BG)^\ren}(\Qlbar, -),
\]
where $\pi^*_\ren$ is symmetric monoidal. We know that $\pi^*_\ren$ preserves compactness, and hence, $\pi_{*, \ren} \simeq \cHom^\gr_{\Shv_\gr(BG)^\ren}(\Qlbar, -)$ is also continuous. Moreover, since $G$ is connected, $\Shv_\gr(BG)^\ren$ is compactly generated by the constant sheaf (along with graded degree shifts). As a result, $\pi^*_\ren$ generates the target and hence, by Barr--Beck--Lurie theorem,~\cite[Thm. 4.7.4.5]{lurie_higher_2017}, the functor $\pi_{*, \ren}$ upgrades to an equivalence of $\Vect^\gr$-module categories
\[
	\pi_{*, \ren}^\enh: \Shv_\gr(BG)^\ren \xrightarrow{\simeq} \Mod_{\pi_{*, \ren} \pi^*_\ren}(\Vect^\gr)
\]
where $\pi_{*, \ren} \pi^*_\ren$ is a monad. Since all functors are strict functors of $\Vect^\gr$-modules, we have an equivalence of functors
\[
	\pi_{*, \ren}\pi^*_\ren(-) \simeq \pi_{*, \ren}\pi^*_\ren(\Qlbar) \otimes - \simeq \Co^*_\gr(BG)\otimes -.
\]
where the monad structure on $\pi_{*, \ren} \pi^*_\ren$ induces an algebra structure on $\Co^*_\gr(BG)$. We thus get an equivalence of $\Vect^\gr$-module categories
\[
	\pi_{*, \ren}^\enh: \Shv_\gr(BG)^\ren \xrightarrow{\simeq} \Mod_{\Co^*_\gr(BG)}(\Vect^\gr). \teq\label{eq:equivalence_of_categories_Shvgr_Mod}
\]

The symmetric monoidal structure on $\Shv_\gr(BG)^\ren$ induces one on $\Mod_{\Co^*_\gr(BG)}$ and hence, by applying~\cite[Cor. 4.8.5.20]{lurie_higher_2017}, we can upgrade the algebra structure on $\Co^*_\gr(BG)$ to a commutative algebra structure such that the tensor product on $\Mod_{\Co^*_\gr(BG)}(\Vect^\gr)$ is just the relative tensor product over $\Co^*_\gr(BG)$. The equivalence \cref{eq:equivalence_of_categories_Shvgr_Mod} thus upgrades to a symmetric monoidal equivalence. It remains to identify the commutative algebra structure on $\Co^*_\gr(BG)$ obtained above with the usual one from cup products. 

By construction, the right-lax symmetric monoidal structures on $\pi_{*, \ren}$ and $\oblv_{\Co^*_\gr(BG)}$ are compatible under the identification \cref{eq:equivalence_of_categories_Shvgr_Mod}. In particular, we have an equivalence of commutative algebras
\[
	\oblv_{\Co^*_\gr(BG)} (\Co^*_\gr(BG)) \simeq \pi_{*, \ren} \Qlbar.
\]
But now, the commutative algebra structure on the right-hand side is precisely given by cup product. Thus, we are done.
\end{proof}

\subsubsection{Functoriality} \label{subsubsec:functoriality_sheaves_BG}
From the construction, suppose $h: G \to H$ is a homomorphism of connected algebraic groups, which induces a morphism $\tilde{h}: BG \to BH$ at the level of classifying stacks. Then, we have a morphism of objects in $\ComAlg(\Vect^\gr)$
\[
	\Co^*_\gr(BH) \to \Co^*_\gr(BG).
\]

A similar argument to \cref{prop:Shv_gr_BG_as_category_of_modules} above yields the following commutative diagram
\[
\begin{tikzcd}
	\Shv_\gr(BG)^\ren \ar{d}{\tilde{h}_{*, \ren}} \ar{r}{\simeq} & \Mod_{\Co^*_\gr(BG)}(\Vect^\gr) \ar{d}{\res_{\Co^*_\gr(BH)}^{\Co^*_\gr(BG)}} \\
	\Shv_\gr(BH)^\ren \ar{r}{\simeq} & \Mod_{\Co^*_\gr(BH)}(\Vect^\gr)
\end{tikzcd}
\]
In other words, pushing forward is identified with restriction of scalar $\res_{\Co^*_\gr(BH)}^{\Co^*_\gr(BG)}$ along $\Co^*_\gr(BH) \to \Co^*_\gr(BG)$. Passing to left adjoints, we see that the pullback functor $\tilde{h}^*_\ren$ is identified with induction $-\otimes_{\Co^*_\gr(BH)} \Co^*_\gr(BG)$.

\begin{prop} \label{prop:bimod_vs_graded_sheaves_BG_x_BG}
In the situation of \cref{prop:Shv_gr_BG_as_category_of_modules}, we have an equivalence of categories
\[
	\Shv_\gr(BG \times BG)^\ren \simeq \BiMod_{\Co^*_\gr(BG)}(\Vect^\gr).
\]
Moreover, the monoidal product on $\BiMod_{C^*_\gr(\Vect^\gr)}$ is identified with the usual convolution monoidal structure on $\Shv_\gr(BG\times BG)^\ren$.
\end{prop}
\begin{proof}
This follows from the discussion above. Indeed, the monoidal structure on $\Shv_\gr(BG \times BG)^\ren$ is given by applying \cref{eq:shv_corr_all_sm} to the following the correspondence below, where we pull along the horizontal map and push along the vertical map
\[
\begin{tikzcd}
	BG \times BG \times BG \ar{d}{p_{13}} \ar{r}{\id \times \Delta \times \id} & BG \times BG \times BG \times BG \\
	BG \times BG
\end{tikzcd}
\]
By the discussion above, in terms of $\BiMod_{\Co^*_\gr(BG)}(\Vect^\gr)$, this is identified with the relative tensor $- \otimes_{\Co^*_\gr(BG)} -$ and the proof concludes.
\end{proof}

\subsection{A geometric realization of the graded finite Hecke category}
As promised, we will prove the following result.

\begin{thm} \label{thm:geometric_incarnation_Hecke}
We have an equivalence of monoidal categories
\[
	\Shv_{\gr, c}(B\backslash G/B) \simeq \Ch^b(\SBim_W^\Rightarrow) \simeq \Ch^b(\SBim_W).	
\]
\end{thm}

Due to \cref{prop:shv_gr_Hecke_stack_classical}, it suffices to show that we have an equivalence $\Shv_{\gr, c}(B\backslash G/B)^{\weightheart} \simeq \SBim_W^\Rightarrow$. Note that, in particular, this implies that the equivalence stated in \cref{thm:geometric_incarnation_Hecke} is compatible with the weight structures.

The rest of this subsection is devoted to the proof of this statement. 

\subsubsection{Construction of functor}
Consider the following commutative diagram (along with its higher analog, where we consider higher powers of $BB\times BB$ and $BB\times_{BG} BB$)
\[
\begin{tikzcd}
	BB \times_{BG} BB \times BB \times_{BG} BB \ar{d} & \ar{l} BB \times_{BG} BB \times_{BG} BB \ar{d} \ar{r} & BB \times_{BG} BB \ar{d} \\
	BB \times BB \times BB \times BB & \ar{l} BB \times BB \times BB \ar{r} & BB \times BB
\end{tikzcd} \teq\label{eq:from_Hecke_stack_to_BB}
\]
where the first square is Cartesian. Pushing forward, we obtain a monoidal functor $\tilde{S}$ given by the composition
\[
	\Shv_{\gr,c}(B\backslash G/B)^{\weightheart} \hookrightarrow \Shv_\gr(B\backslash G/B)^\ren \to \Shv_\gr(BB \times BB)^\ren \simeq \BiMod_{R^\Rightarrow}(\Vect^\gr).
\]
We will show that $\tilde{S}$ factors through $S$ in the following diagram
\[
\begin{tikzcd}
	\Shv_{\gr, c}(B\backslash G/B)^{\weightheart} \ar{r}{S} \ar[bend right=10]{rr}[swap]{\tilde{S}} & \SBim_W^\Rightarrow \ar{r} & \BiMod_{R^\Rightarrow}(\Vect^\gr)
\end{tikzcd} \teq\label{eq:factorization_comparison_functor_SBim}
\]
and moreover, the functor $S$ is an equivalence.

\subsubsection{Factoring the functor}
For each $w \in W$, we let $X_w = B\backslash BwB/B$ denote the Schubert stack associated to $w$ and $\overline{X}_w$ its closure in $B\backslash G/B$. By the decomposition theorem for graded sheaves, \cref{thm:decomposition_theorem}, for any $\mathcal{F} \in \Shv_{\gr, c}(B\backslash G/B)^{\weightheart}$, we have a finite decomposition
\[
	\mathcal{F} \simeq \bigoplus_{w, i} \IC_{\overline{X}_w}[n_i]\lrangle{n_i}.
\]
Thus, to show the factorization \cref{eq:factorization_comparison_functor_SBim}, it suffices to show that $\tilde{S}$ sends $\IC_{\overline{X}_w}$ to an object in $\SBim_W^\Rightarrow$.

When $w_1, w_2 \in W$ such that $\ell(w_1w_2) = \ell(w_1) + \ell(w_2)$, the morphism $\overline{X}_{w_1} * \overline{X}_{w_2} \to \overline{X}_{w_1 w_2}$ is birational, where $\overline{X}_{w_1} * \overline{X}_{w_2} \subset BB\times_{BG} BB \times_{BG} BB$ is the closed substack corresponding to the preimage of $\overline{X}_{w_1} \times \overline{X}_{w_2} \subset BB\times_{BG} BB \times BB \times_{BG} BB$ under the top left map of \cref{eq:from_Hecke_stack_to_BB}. Thus, $\IC_{\overline{X}_{w_1w_2}}$ is a direct summand of $\IC_{\overline{X}_{w_1}} * \IC_{\overline{X}_{w_2}}$. Note that comparing to~\cite[Prop. 3.2.5]{bezrukavnikov_koszul_2013}, we do not need to invoke Frobenius-semisimplicity.

Since $\tilde{S}$ is monoidal, the discussion above implies that it suffices to show that $\tilde{S}$ sends $\IC_{\overline{X}_s}$ to an object in $\SBim_W^\Rightarrow$ for any simple reflection $s\in W$. In this case, $\overline{X}_s$ is smooth. Moreover, it is a classical computation that the cohomology of $\overline{X}_s$ is simply $(R \otimes_{R^s} R)^\Rightarrow$ which belongs to $\SBim_W^\Rightarrow$ by definition. We thus obtain the factorization \cref{eq:factorization_comparison_functor_SBim}.

Note that this discussion also implies that the functor $S$ is essentially surjective.

\subsubsection{Fully-faithfulness}
To show that $S$ is an equivalence of categories, it remains to show that $S$ is fully-faithful. However, this is a consequence of \cite[Prop. 3.1.6]{bezrukavnikov_koszul_2013} and how $\Hom$ is computed in the category of graded sheaves, see \cref{prop:hom_graded_sheaves}. The proof of \cref{thm:geometric_incarnation_Hecke} concludes.

\appendix

\section{Generalities on \texorpdfstring{$\DG$}{DG}-categories}
\label{sec:generalities_DGCats}

$\DG$-categories play an instrumental role in the paper. While there are many different ways to think about $\DG$-categories, we find the one worked out in~\cite[Vol. I, Chap. 1, \S10]{gaitsgory_study_2017} the most well-documented and most convenient to use, despite the fact that it is based on formidable machinery of $\infty$-categories developed in~\cites{lurie_higher_2017, lurie_higher_2017-1}. For example, as far as we know, relative tensors of categories, which already appear in the definition of the category of graded sheaves, are only developed in this context. More generally, constructions involving limits and colimits can be performed in a much more streamlined way in this setting.

Having said that, since the machinery is packaged in such a convenient and intuitive way, most of it can be taken as a black box. This section thus provides a quick review of the important aspects of the theory that are used throughout the paper, focusing on $\DG$-categories, module categories, and relative tensors thereof. We do not aim to review all the concepts used in the paper as it is impossible to do so. Rather, the main goal is to introduce the notations, provide intuition, and familiarize the readers with the language and references of the subject. We do, however, hope that even for readers who are unfamiliar with the theory, this section still provides sufficient background to follow the main ideas of the paper. We note that most results we recall below hold true in more generality than the way we phrase them. We choose to forfeit generality to keep things simple; interested readers can consult the accompanying references.

Most of the materials written here are either developed in~\cites{lurie_higher_2017, gaitsgory_study_2017} or are direct consequences of the results therein. Because of that, no proof will be given here. This is with the exception of~\cref{subsec:induction_formula_enriched_hom} where we prove an induction formula for the enriched $\Hom$-spaces, crucial in our study of graded sheaves. Although the result seems to be well-known among experts, we cannot find a proof in the literature.

Throughout this section, we fix a field $\Lambda$ of characteristics $0$. For the purpose of constructing the theory of graded sheaves, $\Lambda=\Qlbar$. However, results stated in this section hold for any field $\Lambda$ of characteristic $0$. Thus, except for this section, all $\DG$-categories appearing in this paper are linear over $\Qlbar$.

\subsection{Stable presentable categories}
We will now quickly review the main features of stable presentable ($\infty$-)categories.

\subsubsection{$\infty$-categories}
The theory of $\infty$-category is indispensable for our purposes. The theory was developed in great detail in~\cites{lurie_higher_2017-1,lurie_higher_2017}, but shorter accounts exist, see~\cite[Vol. I, Chap. 1]{gaitsgory_study_2017} and~\cite{cisinski_higher_2019}. One virtue of the theory is that, in the words of~\cite{cisinski_higher_2019}, $\infty$-categories allow for a union between category theory and homotopical algebra. For example, (homotopical) limits and colimits, which are ubiquitous in homotopical algebra, form a natural part of the theory of $\infty$-categories.

We let $\Spc$ denote the $\infty$-category of spaces, or $\infty$-groupoids. For any $\infty$-category $\mathcal{C}$, and $c_1, c_2 \in \mathcal{C}$, we use $\Hom_\mathcal{C}(c_1, c_2)$ to denote the \emph{space} of homomorphisms between $c_1$ and $c_2$ in $\mathcal{C}$. $\Hom$ being a space, rather than just a set, is one of the main distinguishing features of the theory of $\infty$-categories compared to classical category theory.

In this paper, unless otherwise specified, by a category, we always mean an $\infty$-category. A classical category, i.e., a $1$-category, can be viewed as an $\infty$-category where the $\Hom$-space is now just a discrete set.

\subsubsection{Presentable categories}
Roughly speaking, a \emph{presentable category} is a ``large'' category (i.e., the collection of objects forms a class rather than a set) but is, in a precise sense, ``generated'' by a set of objects. The theory is worked out in detail in~\cite[\S5.5]{lurie_higher_2017-1}. We let $\Pr^L$ denote the category of presentable categories where morphisms are continuous functors, i.e., they commute with colimits. For $\mathcal{C}, \mathcal{D} \in \Pr^L$, we use $\Fun_{\cont}(\mathcal{C}, \mathcal{D}) \subseteq \Fun(\mathcal{C}, \mathcal{D})$ to denote the full subcategory consisting of \emph{continuous functors}, i.e., those that commute with colimits.\footnote{Note that this differs from the convention used in~\cite{gaitsgory_study_2017} where continuous functors are required to commute with only filtered colimits. However, when working in the setting of stable categories and exact functors, which is where all our categorical constructions happen, there is no difference between the two~\cite[Vol. I, Chap. 1, \S5.1.6]{gaitsgory_study_2017}.}

Presentable categories enjoy many nice features. For example, they admit all (small) limits and colimits~\cite[Def. 5.5.0.1 and Cor. 5.5.2.4]{lurie_higher_2017-1} and they satisfy the adjoint functor theorem~\cite[Cor. 5.5.2.9]{lurie_higher_2017-1}. Moreover, $\Pr^L$ itself admits small limits and colimits~\cite[\S5.5.3]{lurie_higher_2017-1} as well as a symmetric monoidal structure (also known as the Lurie tensor product),~\cite[Prop. 4.8.1.15]{lurie_higher_2017} (see also~\cite[Vol. I, Chap. 1, Thm. 6.1.2]{gaitsgory_study_2017}). The category $\Spc$ is the unit of the symmetric monoidal structure on $\Pr^{L}$.

\subsubsection{}
Due to the importance of this tensor product, let us quickly recall its characterization. Let $\mathcal{C}$ and $\mathcal{D}$ be presentable categories. Then, $\mathcal{C}\otimes \mathcal{D}$ is initial among presentable categories $\mathcal{E}$ which receive a functor from $\mathcal{C} \times \mathcal{D}$ that is continuous in each variable. In particular, for any such $\mathcal{E}$, the category of continuous functors from $\mathcal{C}\otimes \mathcal{D}$ to $\mathcal{E}$ is identified with the category of functors $\mathcal{C} \times \mathcal{D} \to \mathcal{E}$ that is continuous in each variable. We thus have a canonical functor $\mathcal{C} \times \mathcal{D} \to \mathcal{C} \otimes \mathcal{D}$ that is continuous in each variable. For $c\in \mathcal{C}$ and $d\in \mathcal{D}$, we write $c\boxtimes d$ to denote the image of $(c, d) \in \mathcal{C} \times \mathcal{D}$ under this functor.

\subsubsection{Stable categories}
The language of $\infty$-category allows for an elegant formulation of triangulated categories which are called \emph{stable} ($\infty$-)categories,~\cite[Defn. 1.1.1.9]{lurie_higher_2017}. More precisely, an $\infty$-category $\mathcal{C}$ is stable if
\begin{myenum}{--}
\item It is pointed, i.e., it has a $0$ object (which is both final and initial).
\item Every morphism $f: X\to Y$ in $\mathcal{C}$ can be completed into a pullback and a pushout square, respectively, as follows
\[
\begin{tikzcd}
	F \ar{d} \ar{r} & X \ar{d}{f} \\
	0 \ar{r} & Y
\end{tikzcd} \qquad\text{and}\qquad
\begin{tikzcd}
	X \ar{d}{f} \ar{r} & 0 \ar{d} \\
	Y \ar{r} & C
\end{tikzcd} \teq\label{eq:fiber_cofiber_squares}
\]
We call the first (resp. second) square a fiber (resp. co-fiber) sequence. Moreover, $F$ and $C$ are referred to as the fiber and co-fiber of $f$, respectively.
\item The squares in~\cref{eq:fiber_cofiber_squares} above are simultaneously pullback and pushout squares.\footnote{More generally, pullback squares and pushout squares are the same in a stable category~\cite[Prop. 1.1.3.4]{lurie_higher_2017}.}
\end{myenum}

\subsubsection{}
Despite the simple formulation, the homotopy category $\ho\mathcal{C}$ of a stable category $\mathcal{C}$ is a triangulated category~\cite[Thm. 1.1.2.14 and Rmk. 1.1.2.15]{lurie_higher_2017}. Moreover, in this language,
\[
	F \simeq \coCone(X\to Y) \simeq \Cone(X\to Y)[-1] \quad\text{and}\quad  C \simeq \Cone(X\to Y)
\]
where $F$ and $C$ are as in~\cref{eq:fiber_cofiber_squares}. In other words, these squares become distinguished triangles in the resulting triangulated category.

\subsubsection{}
Let $\mathcal{C}$ and $\mathcal{D}$ be stable categories. We use $\Fun_{\ex}(\mathcal{C}, \mathcal{D})$ to denote the full subcategory of $\Fun(\mathcal{C}, \mathcal{D})$ consisting of functors preserving pullbacks, or equivalently, pushouts. It is easy to see that these functors also preserve finite limits and colimits.  

We let $\Pr^{L, \stab}$ denote the full subcategory of $\Pr^L$ consisting of stable presentable categories. It is closed under limits, colimits, and tensor products in $\Pr^L$. The category $\Sptr$ of spectra is the unit of the symmetric monoidal structure on $\Pr^{L, \stab}$. For $\mathcal{C}, \mathcal{D} \in \Pr^{L, \stab}$, $\Fun_{\cont}(\mathcal{C}, \mathcal{D})$ consists of continuous functors that also preserve small colimits.

\subsubsection{Compact generation}
\label{subsubsec:compact_generation}
Among the presentable stable categories, compactly generated ones are particularly manageable. Fortunately, all of the stable categories appearing in this paper are of this type. We start by reviewing the notion of compactness and generation.

An object $c$ in a category $\mathcal{C}$ is said to be \emph{compact} if
\[
	\Hom_\mathcal{C}(c, -): \mathcal{C} \to \Spc
\]
commutes with filtered colimits. We let $\mathcal{C}^c$ denote the full subcategory of $\mathcal{C}$ containing compact objects.

Let $\mathcal{C} \in \Pr^{L, \stab}$. A collection of objects $\{c_\alpha\}$ in $\mathcal{C}$ is said to \emph{generate} if
\[
	\Hom_\mathcal{C}(c_\alpha, c[i]) = 0, \forall \alpha, \forall i=0, 1, \dots \quad\Rightarrow\quad c \simeq 0.
\]

A category $\mathcal{C} \in \Pr^{L, \stab}$ is said to be \emph{compactly generated} if it admits a set of compact generators.

\subsubsection{Ind-completion}
Given any small stable category $\mathcal{C}_0$, we let $\Ind(\mathcal{C}_0)$ denote the smallest full-subcategory of $\Fun(\mathcal{C}_0^{\opp}, \Spc)$ that contains the image of $\mathcal{C}_0$ under the Yoneda embedding and is closed under all filtered colimits. Equivalently, we can define $\Ind(\mathcal{C}_0)$ to be the full subcategory containing all functors that preserve finite limits (or equivalently, fiber products)~\cite[Cor. 5.3.5.4]{lurie_higher_2017-1} or~\cite[Vol. I, Chap. 1, Def. 7.1.3]{gaitsgory_study_2017}. 

We call $\Ind(\mathcal{C}_0)$ the \emph{ind-completion} of $\mathcal{C}_0$. It has the following universal property: for any stable category $\mathcal{D}$ that admits all (small) colimits (for example, when $\mathcal{D}$ is stable and presentable),
\[
	\Fun_{\cont}(\Ind(\mathcal{C}_0), \mathcal{D}) \simeq \Fun_{\ex}(\mathcal{C}_0, \mathcal{D}).
\]
In other words, continuous functors out of $\Ind(\mathcal{C}_0)$, a \emph{large} category, is determined by its restriction to $\mathcal{C}_0$, a \emph{small} category.

\subsubsection{}
\label{subsubsec:Ind_compactly-generated}
$\Ind(\mathcal{C}_0)$ is stable, presentable, and compactly generated. Moreover, the Yoneda embedding $\mathcal{C}_0 \to \Ind(\mathcal{C}_0)$ factors through $\Ind(\mathcal{C}_0)^c$ and the essential image of $\mathcal{C}_0$ forms a collection of compact generators of $\Ind(\mathcal{C}_0)$. Furthermore, all elements in $\Ind(\mathcal{C}_0)^c$ are direct summands of (the Yoneda embeddings of) objects in $\mathcal{C}_0$~\cite[Vol. I, Chap. 1, Lem. 7.2.4]{gaitsgory_study_2017}. In other words, $\Ind(\mathcal{C}_0)^c$ is the \emph{idempotent completion} of $\mathcal{C}_0$.

By~\cite[Vol. I, Chap. 1, Lem. 7.2.4.(3')]{gaitsgory_study_2017}, all compactly generated categories $\mathcal{C}$ are of the form $\Ind(\mathcal{C}_0)$ for some $\mathcal{C}_0$.

We record the following relation between the continuity of functors and compactness.

\begin{lem}[{\cite[Vol. I, Chap. 1, Lem. 7.1.5]{gaitsgory_study_2017}}]
\label{lem:continuous_right_adjoint_vs_compactness}
Let $F: \mathcal{C}\to \mathcal{D}$ be a functor in $\Pr^{L, \stab}$ and $F^R$ its right adjoint (which exists by the adjoint functor theorem). Suppose that $\mathcal{C}$ is compactly generated. Then $F^R$ is continuous if and only if $F$ preserves compactness, i.e., $F$ sends $\mathcal{C}^c$ to $\mathcal{D}^c$.
\end{lem}

\subsubsection{Compact generation of tensors}
Tensor products of compactly generated categories are themselves compactly generated.

\begin{prop}[{\cite[Vol. I, Chap. 1, Prop. 7.4.2]{gaitsgory_study_2017}}] 
\label{prop:compact_generation_tensors}
Let $\mathcal{C}, \mathcal{D} \in \Pr^{L, \stab}$ such that they are both compactly generated. Then, $\mathcal{C} \otimes \mathcal{D}$ is compactly generated by objects of the form $c_0 \boxtimes d_0$ where $c_0\in \mathcal{C}^c$ and $d_0\in \mathcal{D}^c$.
\end{prop}

\subsection{Module categories and enriched \texorpdfstring{$\Hom$}{Hom}-spaces}
\label{subsec:module_categories}

\subsubsection{Algebra and module objects}
\label{subsubsec:alg_mod_DGCat}
For any symmetric monoidal category $\mathcal{A}$, we use $\ComAlg(\mathcal{A})$ to denote the category of commutative algebra objects in $\mathcal{A}$. For any such commutative algebra $A \in \ComAlg(\mathcal{A})$, we use $\Mod_A(\mathcal{A})$ to denote the category of $A$-module objects in $\mathcal{A}$. See~\cite[Vol. I, Chap. 1, \S3]{gaitsgory_study_2017} for a more detailed discussion. 

The Lurie tensor product on $\Pr^{L, \stab}$ allows us to talk about commutative algebra objects in $\Pr^{L, \stab}$ and modules over such an object. In other words, it makes sense to talk about $\ComAlg(\Pr^{L, \stab})$ and for any $\mathcal{A} \in \ComAlg(\Pr^{L, \stab})$, we can talk about $\mathcal{A}$-module categories, which form a category $\Mod_\mathcal{A}(\Pr^{L, \stab})$. For brevity's sake, unless confusion is likely to happen, we will write $\Mod_\mathcal{A}$ to denote $\Mod_\mathcal{A}(\Pr^{L, \stab})$ in the situation above.

\subsubsection{}
Let us quickly unwind the definitions. A commutative algebra object $\mathcal{A} \in \ComAlg(\Pr^{L, \stab})$ is a symmetric monoidal stable presentable category such that the tensor product $\mathcal{A} \times \mathcal{A} \to \mathcal{A}$ is continuous in each variable (which induces a continuous functor $\mathcal{A} \otimes \mathcal{A} \to \mathcal{A}$). We will use $\mult_\mathcal{A}: \mathcal{A} \otimes \mathcal{A} \to \mathcal{A}$ (and sometimes, also $\mult_A: A \times \mathcal{A} \to \mathcal{A}$) to denote the operation of taking the tensor product. In other words, for $a, b \in \mathcal{A}$, $a\otimes b \defeq \mult_\mathcal{A}(a, b) = \mult_{\mathcal{A}}(a\boxtimes b)$.

\subsubsection{} \label{subsubsec:unwind_module_category}
Similarly, an $\mathcal{M} \in \Mod_\mathcal{A}$ is equipped with a ``multiplication'' map $\mathcal{A} \times \mathcal{M} \to \mathcal{M}$ that is continuous on each variable (and hence, induces a continuous functor $\mathcal{A} \otimes \mathcal{M} \to \mathcal{M}$) along with higher compatibilities. We will use $\act_{\mathcal{A}, \mathcal{M}}: \mathcal{A}\otimes \mathcal{M} \to \mathcal{M}$ (and sometimes, also $\act_{\mathcal{A}, \mathcal{M}}: \mathcal{A} \times \mathcal{M} \to \mathcal{M}$) to denote the action of $\mathcal{A}$ on $\mathcal{M}$ given by the module structure. For $a\in \mathcal{A}$ and $m\in \mathcal{M}$, we will also use $a\otimes m$ to denote $a\otimes m \defeq \act_{\mathcal{A}, \mathcal{M}}(a, m) = \act_{\mathcal{A}, \mathcal{M}}(a\boxtimes m)$. 

Since $\mathcal{A}$ is symmetric monoidal, a left-module structure is the same as a right-module structure. Thus, in the above, we also sometimes use $m\otimes a$ to denote the action.

When $F: \mathcal{A} \to \mathcal{B}$ is a symmetric monoidal functor between symmetric monoidal category. Then $\mathcal{B}$ obtains the structure of an $\mathcal{A}$-module. In particular, for $a\in \mathcal{A}$ and $b \in \mathcal{B}$, we can talk about $a \otimes b = F(a) \otimes b$, where the first and second tensor products come from the $\mathcal{A}$-module structure and the monoidal structure on $\mathcal{B}$, respectively. Similarly to the above, since the monoidal structures are \emph{symmetric}, we will also write $b\otimes a = b \otimes F(a)$.

\subsubsection{Lax functors}
Let $\mathcal{A} \in \ComAlg(\Pr^{L, \stab})$ and $\mathcal{M}, \mathcal{N} \in \Mod_\mathcal{A}$. Then one can talk about left/right-lax $\mathcal{A}$-module functors $F: \mathcal{M} \to \mathcal{N}$~\cite[Vol. I, Chap. 1, \S3.5.1]{gaitsgory_study_2017}. Roughly speaking, if $F$ is left-lax, then for any $a\in \mathcal{A}, m\in \mathcal{M}$, we have a natural map
\[
	F(a\otimes m) \to a\otimes F(m).
\]
Similarly, if $F$ is right-lax, we have a map in the opposite direction. These maps are equivalences if $F$ is a strict (rather than lax) functor of $\mathcal{A}$-modules.

\begin{lem}[{\cite[Cor. 7.3.2.7]{lurie_higher_2017}} or~{\cite[Vol. I, Chap. 1, Lem. 3.5.3]{gaitsgory_study_2017}}] \label{lem:adjoint_laxness_mod}
Let $\mathcal{A}, \mathcal{M}, \mathcal{N}$ be as above and $F: \mathcal{M} \rightleftarrows \mathcal{N}: G$ a pair of adjoint functors. Then, the structure on $F$ of a left-lax functor of $\mathcal{A}$-modules is equivalent to the structure on $G$ of a right-lax functor of $\mathcal{A}$-modules.
\end{lem}

\subsubsection{Enriched $\Hom$-spaces}
\label{subsubsec:enriched_Hom}
For any ($\infty$-)category $\mathcal{C}$ and two objects $c_1, c_2 \in \mathcal{C}$, recall that $\Hom_\mathcal{C}(c_1, c_2) \in \Spc$ denotes the the \emph{space} of maps between $c_1$ and $c_2$. The theory of module categories allows for a richer notion of $\Hom$-space that we will now turn to.

Let $\mathcal{A} \in \ComAlg(\Pr^{L, \stab})$ and $\mathcal{M} \in \Mod_\mathcal{A}$. For any two objects $m, n \in \mathcal{M}$, we consider the following functor
\begin{align*}
	\mathcal{A}^{\opp} &\to \Spc \\
	a &\mapsto \Hom_\mathcal{M}(a \otimes m, n).
\end{align*}
By assumption, this functor preserves limits and hence, by presentability of all categories involved, we know that this functor is representable~\cite[Prop. 5.5.2.2]{lurie_higher_2017-1}. We denote the representing object $\cuHom_\mathcal{M}^\mathcal{A}(m, n)$, the \emph{$\mathcal{A}$-enriched $\Hom$-space} between $m$ and $n$. This object is also called the \emph{relative inner $\Hom$} in~\cite[Vol. I, Chap. 1, \S3.6]{gaitsgory_study_2017}.

\subsubsection{}
By definition, we have
\[
	\Hom_\mathcal{M}(a\otimes m, n) \simeq \Hom_\mathcal{A}(a, \cuHom_\mathcal{M}^\mathcal{A}(m, n)). \teq\label{eq:defining_formula_enriched_Hom}
\]
In particular, evaluating at $a = 1_\mathcal{A}$, the monoidal unit of $\mathcal{A}$, we recover the usual $\Hom$-space from the $\mathcal{A}$-enriched one
\[
	\Hom_\mathcal{M}(m, n) \simeq \Hom_\mathcal{A}(1_\mathcal{A}, \cuHom_\mathcal{M}^\mathcal{A}(m, n)).
\]

\subsubsection{} \label{subsubsec:enriched_hom_as_adjoint}
Equation \cref{eq:defining_formula_enriched_Hom} also implies that for any $m \in \mathcal{M}$, we have a pair of adjoint functors
\[
\begin{tikzcd}[column sep=huge]
	\mathcal{A} \ar[shift left=\arrdisp]{r}{-\otimes m} & \ar[shift left=\arrdisp]{l}{\cuHom_\mathcal{M}^\mathcal{A}(m, -)} \mathcal{M}.
\end{tikzcd}
\]

\subsubsection{}
When $\mathcal{M} = \mathcal{A}$, then for any $a_1, a_2 \in \mathcal{A}$, we obtain the inner-$\Hom$: $\cuHom_{\mathcal{A}}^{\mathcal{A}}(a_1, a_2) \in \mathcal{A}$. For brevity's sake, we will omit the superscript $\mathcal{A}$ and simply write $\cuHom_\mathcal{A}(a_1, a_2)$.

\subsubsection{Relative tensors of module categories}
Let $\mathcal{A} \in \ComAlg(\Pr^{L, \stab})$ and $\mathcal{M}, \mathcal{N} \in \Mod_\mathcal{A}$. Then, one can form the relative tensor (\cite[\S4.4]{lurie_higher_2017} and~\cite[Vol. I, Chap. 1, \S4.2.1]{gaitsgory_study_2017})
\[
	\mathcal{M}\otimes_\mathcal{A} \mathcal{N} = |\mathcal{M} \otimes \mathcal{A}^{\otimes \bullet} \otimes \mathcal{N}| \in \Mod_\mathcal{A}.
\]
Here, $\mathcal{M} \otimes \mathcal{A}^{\otimes \bullet} \otimes \mathcal{N}$ denotes the simplicial object obtained by the two-sided bar construction, and $|-|$ denotes the geometric realization of a simplicial object (i.e., taking colimit). This construction equips $\Mod_\mathcal{A}$ with the structure of a symmetric monoidal category.

We have canonically defined functors
\[
	\mathcal{M} \times \mathcal{N} \to \mathcal{M}\otimes \mathcal{N} \to \mathcal{M} \otimes_{\mathcal{A}} \mathcal{N}.	
\]
For $m \in \mathcal{M}$ and $n \in \mathcal{N}$, we write $m \boxtimes_\mathcal{A} n \in \mathcal{M}\otimes_\mathcal{A} \mathcal{N}$ to denote the image of $(m, n) \in \mathcal{M}\times \mathcal{N}$ under this functor. When no confusion is likely to occur, we will simply write $m\boxtimes n \in \mathcal{M} \otimes_\mathcal{A} \mathcal{N}$ for $m \boxtimes_\mathcal{A} n$.

\subsubsection{Compact generation of relative tensors}
We have the following generalization of \cref{prop:compact_generation_tensors}.

\begin{prop}[{\cite[Vol. I, Chap. 1, Prop. 8.7.4]{gaitsgory_study_2017}}] 
\label{prop:compact_generation_rel_tensors}
Let $\mathcal{A} \in \ComAlg(\Pr^{L, \stab})$ and $\mathcal{M}, \mathcal{N} \in \Mod_\mathcal{A}$ such that $\mathcal{A}, \mathcal{M}, \mathcal{N}$ are all compactly generated. Suppose that
\[
	\mathcal{A} \otimes \mathcal{A} \to \mathcal{A}, \quad \mathcal{A} \otimes \mathcal{N} \to \mathcal{N}, \quad \mathcal{M}\otimes \mathcal{A} \to \mathcal{M}
\]
preserve compactness. Then, the functor $\mathcal{M} \otimes \mathcal{N} \to \mathcal{M}\otimes_\mathcal{A} \mathcal{N}$ also preserves compactness. Moreover, $\mathcal{M}\otimes_\mathcal{A} \mathcal{N}$ is compactly generated by objects of the form $m_0 \boxtimes n_0$ where $m_0 \in \mathcal{M}^c$ and $n_0 \in \mathcal{N}^c$.
\end{prop}

\subsection{Duality}
We will review basic general patterns of duality. The most important concept is that of a dualizable object. This is an important finiteness condition in a monoidal category which frequently appears in knot theory and topological quantum field theory. The materials presented in this section come from~\cite[Vol. I, Chap. 1, \S4]{gaitsgory_study_2017}.

\subsubsection{Dualizability}
Let $\mathcal{A} \in \ComAlg(\Pr^{L, \stab})$ and $a\in \mathcal{A}$. Then, we obtain an endofunctor $a \otimes -: \mathcal{A} \to \mathcal{A}$. We say that $a$ admits a \emph{dual} $a^\vee$ if $a^\vee \otimes -$ is adjoint to $a \otimes -$. In this case, we say that $a$ is \emph{dualizable}. 

Note that most of~\cite[Vol. I, Chap. 1, \S4]{gaitsgory_study_2017} distinguishes between left and right duals (or dualizability), which correspond to specifying precisely which adjoint (i.e., left vs. right) we have. However, in a symmetric monoidal category, the two notions coincide. Indeed, given a dualizable object $a\in \mathcal{A}$ with dual $a^\vee$, for any $b, c\in \mathcal{A}$, we have
\[
	\Hom_\mathcal{A}(a \otimes b, c) \simeq \Hom_\mathcal{A}(b, a^\vee \otimes c) \teq\label{eq:defining_eq_dual}
\]
and
\[
	\Hom_\mathcal{A}(b, a\otimes c) \simeq \Hom_\mathcal{A}(a^\vee \otimes b, c).
\]

\subsubsection{Dual object and internal-$\Hom$}
Just as in linear algebra, dual objects, in general, can be expressed as the (internal) $\Hom$ into the monoidal unit object. Let $a \in \mathcal{A}$ be a dualizable object as above. By \cref{subsubsec:enriched_Hom}, for any $b, c \in \mathcal{A}$, we have
\[
	\Hom_\mathcal{A}(a \otimes b, c) \simeq \Hom_\mathcal{A}(b, \cuHom_\mathcal{A}(a, c)). \teq\label{eq:internal_Hom_universal}
\]
Comparing with~\cref{eq:defining_eq_dual}, we get
\[
	\cuHom_\mathcal{A}(a, c) \simeq a^\vee \otimes c.
\]
In particular, taken $c = 1_\mathcal{A}$, we get
\[
	a^\vee \simeq \cuHom_\mathcal{A}(a, 1_\mathcal{A}).
\]
Moreover, the functor $\cuHom_\mathcal{A}(a, -) \simeq a^\vee \otimes -$ commutes with both limits and colimits. 

\subsection{Rigidity}
Rigid monoidal categories are stable monoidal categories that exhibit strong finiteness properties and behave extremely nicely with respect to duality and module category structures. In the stable categorical context, they were introduced in~\cite[Vol. I, Chap. 1, \S9]{gaitsgory_study_2017} which is the reference for this subsection.

\subsubsection{Compactly generated rigid symmetric monoidal categories}
\begin{defn}
Let $\mathcal{A} \in \ComAlg(\Pr^{L, \stab})$ such that $\mathcal{A}$ is compactly generated. We say that $\mathcal{A}$ is \emph{rigid} if
\begin{myenum}{--}
	\item The monoidal unit $1_\mathcal{A}$ is compact.
	\item The functor $\mult_\mathcal{A}$ sends $\mathcal{A}^c\times \mathcal{A}^c$ to $\mathcal{A}^c$, i.e., it preserves compactness.
	\item Every compact object of $\mathcal{A}$ is dualizable.
\end{myenum}
\end{defn}

We note that rigidity, in general, does \emph{not} require the category involved to be compactly generated nor the monoidal structure to be symmetric. The general definition can be found at~\cite[Vol. I, Chap. 1, Def. 9.1.2]{gaitsgory_study_2017}. However, we do not need this generality in the paper. The general definition is equivalent to the one given above due to~\cite[Vol. I, Chap. 1, Lem. 9.1.5]{gaitsgory_study_2017}.

\begin{expl}
We let $\Vect$ denote the ($\infty$-)derived category of chain complexes of vector spaces over a field $\Lambda$. The category $\Vect$ has a natural symmetric monoidal structure. It is compactly generated with compact objects being perfect complexes. In fact, it is compactly generated by $\Lambda$. From this description, we see that $\Vect$ is a compactly generated rigid symmetric monoidal category.
\end{expl}

\begin{expl}
\label{expl:Vectgr_rigid}
Let $\Vect^\gr$ denote the category of graded chain complexes. Formally speaking, let $\mathbb{Z}$ be the discrete category with the underlying set of objects given by the set of integers. Then, $\Vect^\gr = \Fun(\mathbb{Z}, \Vect)$. We think of objects in $\Vect^\gr$ as graded chain complexes. Informally, any $V\in \Vect^\gr$ is of the form $V = \bigoplus_{i\in \mathbb{Z}} V_i$ where $V_i$ is the $i$-th graded component of $V$. 

$\Vect^\gr$ has a natural symmetric monoidal structure given by Day convolution~\cite[\S2.2.6]{lurie_higher_2017}. More informally, given $V, W \in \Vect^\gr$,
\[
	(V \otimes W)_k = \bigoplus_{i + j = k} V_i \otimes W_j.
\]

$\Vect^\gr$ is compactly generated with compact objects being graded chain complexes supported on finitely many degrees and each graded degree is perfect. Any compact object $V\in (\Vect^\gr)^c$ is dualizable with dual given by $(V^\vee)_i = (V_{-i})^\vee$. From this description, we see that $\Vect^\gr$ is a compactly generated rigid symmetric monoidal category. 

For $V \in \Vect$, we will also abuse notation and view it as an object in $\Vect^\gr$ where $V$ is placed in degree $0$. For any $V\in \Vect^\gr$ and $k\in \mathbb{Z}$, we write $V\lrangle{k}$ to denote a grading shift of $V$, i.e., $(V\lrangle{k})_n = V_{n+k}$. In particular, for $V\in \Vect$, $V\lrangle{k}$ denotes the object in $\Vect^\gr$ obtained by putting $V$ put in graded degree $-k$. The category $\Vect^\gr$ is compactly generated by $\{\Lambda\lrangle{k}\}_{k\in \mathbb{Z}}$.
\end{expl}

\subsubsection{Interaction with lax functors} 
Lax functors between module categories over a rigid monoidal category are automatically strict.

\begin{lem}[{\cite[Vol. I, Chap. 1, Lem. 9.3.6]{gaitsgory_study_2017}}] \label{lem:lax_module_functor_is_strict_rigid}
Let $\mathcal{A}$ be a rigid monoidal category. Any continuous right-lax or (left-lax) functor between $\mathcal{A}$-module categories is strict.
\end{lem}

Combining this result and \cref{lem:adjoint_laxness_mod}, we obtain the following result.
\begin{cor} \label{cor:lax_implies_strict_rigid}
Let $\mathcal{A}$ be a rigid monoidal category, $\mathcal{M}, \mathcal{N} \in \Mod_\mathcal{A}$. Let $F: \mathcal{M} \rightleftarrows \mathcal{N}: G$ be a pair of adjoint functors between the underlying categories. Then, the following are equivalent
\begin{myenum}{(\roman*)}
	\item $F$ is a left-lax functor of $\mathcal{A}$-modules.
	\item $G$ is a right-lax functor of $\mathcal{A}$-modules.
	\item $F$ is a strict functor of $\mathcal{A}$-modules.
	\item $G$ is a strict functor of $\mathcal{A}$-modules.
\end{myenum}
\end{cor}

\subsubsection{Interaction with compactness}
Let $\mathcal{A} \in \ComAlg(\Pr^{L, \stab})$ and $\mathcal{M} \in \Mod_\mathcal{A}$. Then, an object $m\in \mathcal{M}$ is said to be \emph{compact relative to $\mathcal{A}$} if
\[
	\cuHom_\mathcal{M}^\mathcal{A}(m, -): \mathcal{M} \to \mathcal{A}
\]
commutes with filtered colimits (equivalently, all colimits)~\cite[Vol. I, Chap. 1, Def. 8.8.2]{gaitsgory_study_2017}.

The situation is especially nice when $\mathcal{A}$ is rigid.
\begin{lem}[{\cite[Vol. I, Chap. 1, Lem. 9.3.4]{gaitsgory_study_2017}}] \label{lem:relative_comp_vs_comp_rigid}
Let $\mathcal{A}$ be a rigid monoidal category and $\mathcal{M} \in\Mod_\mathcal{A}$. Then, $m\in \mathcal{M}$ is compact if and only if it is compact relative to $\mathcal{A}$.
\end{lem}

\subsubsection{Compact generation of relative tensors}
When $\mathcal{A}$ is rigid, relative tensors of compactly generated $\mathcal{A}$-modules are always compactly generated.

\begin{lem}[{\cite[Vol. I, Chap. 1, Lem. 9.3.2]{gaitsgory_study_2017}}]
\label{lem:act^R_mod_over_rigid}
Let $\mathcal{A}$ be a rigid monoidal category and $\mathcal{M} \in \Mod_\mathcal{A}$. Then, $\act_{\mathcal{A}, \mathcal{M}}: \mathcal{A}\otimes \mathcal{M} \to \mathcal{M}$ admits a continuous right adjoint.
\end{lem}

Combining with \cref{lem:continuous_right_adjoint_vs_compactness,prop:compact_generation_tensors}, we obtain the following
\begin{cor}
\label{cor:action_map_rigid_preserves_compactness}
Let $\mathcal{A}$ and $\mathcal{M}$ be as in \cref{lem:act^R_mod_over_rigid}. Suppose that $\mathcal{A}$ and $\mathcal{M}$ are compactly generated. Then $\mathcal{A} \otimes \mathcal{M} \to \mathcal{M}$ preserves compactness.
\end{cor}

Combining with \cref{prop:compact_generation_rel_tensors}, we get
\begin{cor} \label{cor:compact_generators_relative_tensors}
Let $\mathcal{A} \in \ComAlg(\Pr^{L, \stab})$ that is rigid and compactly generated. Let $\mathcal{M}, \mathcal{N} \in \Mod_\mathcal{A}$ be compactly generated as well. Then, $\mathcal{M} \otimes \mathcal{N} \to \mathcal{M} \otimes_\mathcal{A} \mathcal{N}$ preserves compactness. Moreover, $\mathcal{M} \otimes_\mathcal{A} \mathcal{N}$ is compactly generated by objects of the form $m_0 \boxtimes n_0$ where $m_0 \in \mathcal{M}^c$ and $n_0 \in \mathcal{N}^c$.
\end{cor}

\subsection{An induction formula for enriched \texorpdfstring{$\Hom$}{Hom}-spaces}
\label{subsec:induction_formula_enriched_hom}

Relative tensor products over a compactly generated rigid symmetric monoidal category are quite explicit. \cref{cor:compact_generators_relative_tensors} gives us a set of compact generators. But in fact, enriched $\Hom$-spaces are also explicit.

\begin{prop}
\label{prop:Hom_in_rel_tensor_of_cats}
Let $\mathcal{A}$ be a compactly generated rigid symmetric monoidal category, $\mathcal{M}, \mathcal{N} \in \Mod_\mathcal{A}$ that are both compactly generated. Then, $\mathcal{M} \otimes_{\mathcal{A}} \mathcal{N}$ is compactly generated. Moreover, for $(m_0, n_0) \in \mathcal{M}^c \times \mathcal{N}^c$ and $(m, n) \in \mathcal{M} \times \mathcal{N}$,
\[
	\cuHom_{\mathcal{M} \otimes_{\mathcal{A}} \mathcal{N}}^{\mathcal{A}}(m_0 \boxtimes n_0, m\boxtimes n) \simeq \cuHom_\mathcal{M}^\mathcal{A}(m_0, m) \otimes \cuHom_\mathcal{N}^\mathcal{A}(n_0, n).
\]
\end{prop}
\begin{proof}
From \cref{subsubsec:enriched_hom_as_adjoint}, we have a pair of adjoint functors
\[
\begin{tikzcd}[column sep=huge]
	\mathcal{A} \ar[shift left=\arrdisp]{r}{-\otimes n_0} & \ar[shift left=\arrdisp]{l}{\cuHom_\mathcal{N}^\mathcal{A}(n_0, -)} \mathcal{N}.
\end{tikzcd} \teq\label{eq:proof_hom_rel_tensor_first_adjunction}
\]
By \cref{lem:relative_comp_vs_comp_rigid}, we know that the right adjoint $\cuHom_\mathcal{N}^\mathcal{A}(n_0, -)$ is continuous. Moreover, since $-\otimes n_0$ is compatible with $\mathcal{A}$-module structures on both sides, so is $\cuHom_\mathcal{N}^\mathcal{A}(n_0, -)$, by \cref{cor:lax_implies_strict_rigid}. Thus, applying $\mathcal{M}\otimes_\mathcal{A} - $ to the above, we obtain the following pair of adjoint functors
\[
\begin{tikzcd}[column sep=15ex]
	\mathcal{M} \ar[shift left=\arrdisp]{r}{-\boxtimes n_0} & \ar[shift left=\arrdisp]{l}{\id_\mathcal{M} \otimes \cuHom_\mathcal{N}^\mathcal{A}(n_0, -)} \mathcal{M} \otimes_\mathcal{A} \mathcal{N}.
\end{tikzcd} \teq\label{eq:proof_hom_rel_tensor_first_adjunction_after_tensor}
\]
Here, $\id_\mathcal{M} \otimes \cuHom_\mathcal{N}^\mathcal{A}(n_0, -)$ is the functor given by
\[
	m\boxtimes n \mapsto m \otimes \cuHom_\mathcal{N}^\mathcal{A}(n_0, n),
\]
where $\cuHom_\mathcal{N}^\mathcal{A}(n_0, n) \in \mathcal{A}$ and we have used the $\mathcal{A}$-module structure on $\mathcal{M}$ to form the tensor, see \cref{subsubsec:unwind_module_category}.

Similarly to \cref{eq:proof_hom_rel_tensor_first_adjunction}, we have the following pair of adjoint functors
\[
\begin{tikzcd}[column sep=huge]
	\mathcal{A} \ar[shift left=\arrdisp]{r}{m_0 \otimes -} & \ar[shift left=\arrdisp]{l}{\cuHom_\mathcal{M}^\mathcal{A}(m_0, -)} \mathcal{M}.
\end{tikzcd}
\]
Composing this with \cref{eq:proof_hom_rel_tensor_first_adjunction_after_tensor}, we obtain a pair of adjoint functors
\[
\begin{tikzcd}[column sep=20ex]
	\mathcal{A} \ar[shift left=\arrdisp]{r}{- \otimes (m_0 \boxtimes n_0)} & \ar[shift left=\arrdisp]{l}{\cuHom_\mathcal{M}^\mathcal{A}(m_0, - \otimes \cuHom_\mathcal{N}^\mathcal{A}(n_0, -))} \mathcal{M} \otimes_\mathcal{A} \mathcal{N}.
\end{tikzcd}
\]

Now, note that
\[
	\cuHom_\mathcal{M}^\mathcal{A}(m_0, - \otimes \cuHom_\mathcal{N}^\mathcal{A}(n_0, -)) \simeq \cuHom_\mathcal{M}^\mathcal{A}(m_0, - )\otimes \cuHom_\mathcal{N}^\mathcal{A}(n_0, -) \teq\label{eq:taking_out_tensor_from_Hom}
\]
since $\cuHom_\mathcal{M}^\mathcal{A}(m_0, -)$ is compatible with $\mathcal{A}$-module structures by \cref{cor:lax_implies_strict_rigid}. On the other hand, $-\otimes(m_0\boxtimes n_0)$ also admits a right adjoint given by $\cuHom_{\mathcal{M} \otimes_\mathcal{A} \mathcal{N}}^\mathcal{A}(m_0 \boxtimes n_0, -)$, again by \cref{subsubsec:enriched_hom_as_adjoint}. This implies that for all $m\boxtimes n \in \mathcal{M} \otimes_\mathcal{A} \mathcal{N}$, we have
\[
	\cuHom_{\mathcal{M} \otimes_\mathcal{A} \mathcal{N}}^\mathcal{A}(m_0 \boxtimes n_0, m\boxtimes n) \simeq \cuHom_\mathcal{M}^\mathcal{A}(m_0, m)\otimes \cuHom_\mathcal{N}^\mathcal{A}(n_0, n)
\]
as desired.
\end{proof}

We have the following variant of the proposition above.

\begin{prop}\label{prop:induction_Hom_tensor}
Let $F: \mathcal{A} \to \mathcal{B}$ be a symmetric monoidal functor between compactly generated rigid symmetric monoidal categories. Let $\mathcal{M} \in \Mod_\mathcal{A}$. Then, for $(b_0, m_0) \in \mathcal{B}^c \times \mathcal{M}^c$ and $(b, m) \in \mathcal{B} \times \mathcal{M}$, we have (see the end of \cref{subsubsec:unwind_module_category} for the notation)
\begin{align*}
	\cuHom_{\mathcal{B} \otimes_\mathcal{A} \mathcal{M}}^\mathcal{B}(b_0 \boxtimes m_0, b \boxtimes m)
	&\simeq \cuHom_{\mathcal{B}}(b_0, b) \otimes \cuHom_\mathcal{M}^\mathcal{A}(m_0, m)\\
	&\simeq \cuHom_\mathcal{B}(b_0, b) \otimes F(\cuHom_\mathcal{M}^\mathcal{A}(m_0, m)).
\end{align*}
In particular, if we let $F_\mathcal{M}: \mathcal{M} \to \mathcal{B} \otimes_\mathcal{A} \mathcal{M}$ be the natural functor, i.e., $F_\mathcal{M}(m) = 1_\mathcal{B} \boxtimes m$, then
\[
	\cuHom_{\mathcal{B} \otimes_\mathcal{A} \mathcal{M}}^\mathcal{B}(F_\mathcal{M}(m_0), F_\mathcal{M}(m)) \simeq F(\cuHom_\mathcal{M}^\mathcal{A}(m_0, m)).
\]
\end{prop}
\begin{proof}
The second part follows from the first part by noticing that $\cuHom_\mathcal{B}(1_\mathcal{B}, 1_\mathcal{B}) \simeq 1_\mathcal{B}$.

The proof of the first part is similar to the one above. Namely, we have the following pair of adjoint functors
\[
\begin{tikzcd}[column sep=20ex]
	\mathcal{B} \ar[shift left=\arrdisp]{r}{- \otimes (b_0 \boxtimes m_0)} & \ar[shift left=\arrdisp]{l}{\cuHom_\mathcal{B}(b_0, - )\otimes \cuHom_\mathcal{M}^\mathcal{A}(m_0, -)} \mathcal{B} \otimes_\mathcal{A} \mathcal{M}.
\end{tikzcd}
\]
On the other hand, $- \otimes (b_0 \boxtimes m_0)$ also admits a right adjoint given by $\cuHom_{\mathcal{B} \otimes_\mathcal{A} \mathcal{M}}^\mathcal{B}(b_0 \boxtimes m_0, -)$, by \cref{subsubsec:enriched_hom_as_adjoint}. Thus, for all $(b, m) \in \mathcal{B} \times \mathcal{M}$, we have
\[
	\cuHom_{\mathcal{B} \otimes_\mathcal{A} \mathcal{M}}^\mathcal{B}(b_0 \boxtimes m_0, b \boxtimes m) \simeq \cuHom_\mathcal{B}(b_0, b) \otimes \cuHom_\mathcal{M}^\mathcal{A} (m_0, m)
\]
as desired.
\end{proof}

We end this subsection with a useful observation.

\begin{prop}
\label{prop:F_M_is_conservative_when_F_is}
In the situation of \cref{prop:induction_Hom_tensor}, suppose further that $F$ is conservative. Then, $F_\mathcal{M}$ is also conservative.
\end{prop}
\begin{proof}
Since $\mathcal{M}$ is compactly generated, it is dualizable, by~\cite[Vol. I, Chap. 1, Prop. 7.3.2]{gaitsgory_study_2017}. Moreover, since $\mathcal{A}$ is rigid, $\mathcal{M}$ is also dualizable as an $\mathcal{A}$-module category, by~\cite[Vol. I, Chap. 1, Prop. 9.4.4]{gaitsgory_study_2017}. Now, the claim follows from \cref{lem:F_M_is_conservative_when_F_is}.

\end{proof}

We thank the anonymous referee for the following useful observation.

\begin{lem}
	\label{lem:F_M_is_conservative_when_F_is}
	Let $\mathcal{A}$ be an object in $\ComAlg(\DGCatprescont)$, $\mathcal{C}, \mathcal{D}$, and $\mathcal{M}$ objects in $\Mod_{\mathcal{A}}$ such that $\mathcal{M}$ is a dualizable $\mathcal{A}$-module category, whose dual is denoted by $\mathcal{M}^{\vee_A}$. Let $F: \mathcal{C} \to \mathcal{D}$ be a conservative $\mathcal{A}$-linear functor. Then, the induced map
	\[
		F_\mathcal{M}: \mathcal{M} \otimes_\mathcal{A} \mathcal{C} \to \mathcal{M} \otimes_\mathcal{A} \mathcal{D}
	\]
	is also conservative.
\end{lem}
\begin{proof}
	By the assumption on $\mathcal{M}$, we have $\mathcal{M} \otimes_\mathcal{A} - \simeq \Hom_\mathcal{A}(\mathcal{M}^{\vee_A}, -)$ and under this identification, $F_M$ is the functor
	\[
		\Hom_\mathcal{A}(\mathcal{M}^{\vee_A}, \mathcal{C}) \to \Hom_\mathcal{A}(\mathcal{M}^{\vee_A}, \mathcal{D}),
	\]
	induced by compositing with $F$. But now, the conservativity of $F_\mathcal{M}$ follows from that of $F$.
\end{proof}

\subsection{\texorpdfstring{$\DG$}{DG}-categories}
\label{subsec:DG-cats}
The theory of $\DG$-categories fits nicely in the framework of stable categories reviewed above. The materials in this subsection are from~\cite[Vol. I, Chap. 1, \S10]{gaitsgory_study_2017}.

A $\DG$-category is a stable presentable category equipped with a module structure over $\Vect$, i.e., it is an object in $\Mod_{\Vect} = \Mod_{\Vect}(\Pr^{L, \stab})$. We let $\DGCatprescont = \Mod_{\Vect}$ denote the category of $\DG$-categories with morphisms being continuous functors (compatible with $\Vect$-module structures).

As already seen above, $\DGCatprescont$ is a symmetric monoidal category where the tensor structure is given by relative tensor over $\Vect$. In this paper, given $\mathcal{C}, \mathcal{D} \in \DGCatprescont$, we will never take the ``absolute tensor'' $\mathcal{C} \otimes \mathcal{D}$. On the other hand, we will make extensive use of the relative tensors $\mathcal{C}\otimes_{\Vect} \mathcal{D}$. Because of that, from this point forward, following~\cite[Vol. I, Chap. 1, \S10.6]{gaitsgory_study_2017}, for we will adopt the following convention
\begin{myenum}{--}
	\item When $\mathcal{C}, \mathcal{D} \in \DGCatprescont$, $\mathcal{C} \otimes \mathcal{D}$ is used to denote $\mathcal{C}\otimes_{\Vect} \mathcal{D}$. Moreover, for $c \in \mathcal{C}$ and $d\in \mathcal{D}$, $c\boxtimes d$ is used to denote the element in $\mathcal{C} \otimes_{\Vect} \mathcal{D}$.
	\item For $c_1, c_2 \in \mathcal{C}$ where $\mathcal{C} \in \DGCatprescont$, $\cHom_\mathcal{C}(c_1, c_2)$ is used to denote $\cuHom_\mathcal{C}^{\Vect}(c_1, c_2)$. In particular, $\Hom_\mathcal{C}(c_1, c_2) = \Hom_{\Vect}(\Lambda, \cHom_\mathcal{C}(c_1, c_2))$.
	\item All functors between $\DG$-categories are, unless otherwise specified, compatible with the $\Vect$-module structures.
\end{myenum}

\subsection{Large vs. small categories}
\label{subsec:large_vs_small_cats}
Up to now, we have been working with \emph{large} categories, in the sense that the collection of objects is not a set. As mentioned above, working in this context allows many convenient theorems regarding the existence of adjoints and limits/colimits to hold true. On the other hand, working with small categories is usually more concrete. In fact, familiar constructions in representation theory usually involve small categories rather than large ones. In this sense, large categories could be viewed as a technical device: we usually perform our construction in the large category setting whenever it is convenient to do so, and then extract the small category out of that.

We will now review constructions involving small categories and how they are related to those in the large category setting.

\subsubsection{Small $\DG$-categories}
Let $\DGCatidemex$ denote the category whose objects are idempotent complete stable infinity categories, equipped with an action of $\Vect^c$, and whose morphisms are given by exact functors (i.e., functors that commute with finite limits and colimits) compatible with the action of $\Vect^c$. Unless otherwise specified, all small $\DG$-categories in the current paper belong to $\DGCatidemex$.

The procedure of taking $\Ind$-completion provides a functor
\[
	\Ind: \DGCatidemex \to \DGCatprescont.
\]
The essential image of $\Ind$ is precisely those $\DG$-categories that are compactly generated. Indeed, given such a category $\mathcal{C}$, the full subcategory $\mathcal{C}^c$ spanned by the set of compact objects is an object of $\DGCatidemex$. Moreover, $\Ind(\mathcal{C}^c) \simeq \mathcal{C}$ when $\mathcal{C}$ is compactly generated, see also \cref{subsubsec:Ind_compactly-generated}.

\subsubsection{Tensors of small $\DG$-categories}
Similarly to~\cite[Prop. 4.4]{ben-zvi_integral_2010}, $\DGCatidemex$ admits a symmetric monoidal structure given by
\[
	\mathcal{C}_1 \otimes \mathcal{C}_2 \defeq (\Ind(\mathcal{C}_1) \otimes \Ind(\mathcal{C}_2))^c, \qquad \mathcal{C}_1, \mathcal{C}_2 \in \DGCatidemex.
\]
Moreover, the functor $\Ind$ is symmetric monoidal. In particular, when $\mathcal{C}_1, \mathcal{C}_2 \in \DGCatprescont$ are compactly generated,
\[
	(\mathcal{C}_1 \otimes \mathcal{C}_2)^c = \mathcal{C}_1^c \otimes \mathcal{C}_2^c.
\]

Using this monoidal structure, it is possible to make sense of $\ComAlg(\DGCatidemex)$. Namely, $\mathcal{A} \in \ComAlg(\DGCatidemex)$ is a symmetric monoidal small $\DG$-category such that the tensor product is compatible with the $\Vect^c$-action and with finite colimits in each variable. For such $\mathcal{A}$, it is also possible to make sense of $\Mod_\mathcal{A}(\DGCatidemex)$. Given $\mathcal{A} \in \ComAlg(\DGCatidemex)$, we will also use the notation $\Mod_\mathcal{A} \defeq \Mod_\mathcal{A}(\DGCatidemex)$ when no confusion is likely to occur. Note that this notation is similar to the one introduced in \cref{subsubsec:alg_mod_DGCat}; however, the fact that $\mathcal{A}$ is \emph{small} should be clear from the context.

\subsubsection{Module structures on large vs. small categories}
\label{subsubsec:module_structures_large_vs_small}
For $\mathcal{A} \in \ComAlg(\DGCatprescont)$ and $\mathcal{M} \in \Mod_\mathcal{A}$ such that $\mathcal{A}$ is rigid compactly generated and $\mathcal{M}$ is compactly generated, \cref{cor:action_map_rigid_preserves_compactness} implies that $\mathcal{M}^c \in \Mod_{\mathcal{A}^c}$. Moreover, by continuity, the $\mathcal{A}$-module structure on $\mathcal{M}$ is obtained by ind-extending the $\mathcal{A}^c$-module structure on $\mathcal{M}^c$.

\subsubsection{Relative tensors of small $\DG$-categories}
Much as above, but in the relative setting, let $\mathcal{A} \in \ComAlg(\DGCatprescont)$ be compactly generated and rigid, and $\mathcal{M}, \mathcal{N} \in \Mod_\mathcal{A}$ be compactly generated. Then, we define
\[
	\mathcal{M}^c \otimes_{\mathcal{A}^c} \mathcal{N}^c = (\mathcal{M} \otimes_\mathcal{A} \mathcal{N})^c.
\]
Here, the LHS is performed in $\DGCatidemex$ using the tensor product described above.

\subsubsection{Enriched $\Hom$}
Let $\mathcal{A}_0$ be in $\ComAlg(\DGCatidemex)$ and $\mathcal{M}_0$ in $\Mod_{\mathcal{A}_0}(\DGCatidemex)$, then we can make sense of $\mathcal{A} \defeq \Ind(\mathcal{A}_0)$-enriched $\Hom$-spaces between objects in $\mathcal{M}_0$ by embedding $\mathcal{M}_0 \hookrightarrow \mathcal{M} \defeq \Ind(\mathcal{M}_0)$ and use the discussion above for presentable categories. Here, the $\mathcal{A}$-module structure on $\mathcal{M}$ is obtained by left Kan extending $\mathcal{A}_0 \otimes \mathcal{M}_0 \to \mathcal{M}_0$ along $\mathcal{A}_0 \otimes \mathcal{M}_0 \to \mathcal{A} \otimes \mathcal{M}$.


\section*{Acknowledgements}
Q. Ho would like to thank Minh-Tam Quang Trinh for asking him many questions about the chromatographic construction many years ago and for many stimulating conversations about knot invariants. Q. Ho would also like to thank the organizers of the American Institute of Mathematics's Link Homology Research Community for providing a stimulating (online) research environment. P. Li would like to thank Peng Shan and Yin Tian for many helpful conversations on Soergel bimodules and Hilbert schemes. 

We thank Jens Eberhardt, Kaif Hilman, Yifeng Liu, Andrew Macpherson, and Zhiwei Yun for many helpful discussions and email exchanges. We also thank Sasha Beilinson and David Ben-Zvi for their interest and questions, which helped clarify various parts of the paper. Finally, we thank the anonymous referee for carefully reading the manuscript and providing extensive feedback.

This work started when Q. Ho was a postdoc in Hausel group at IST Austria, partially supported by the Lise Meitner fellowship, Austrian Science Fund (FWF): M 2751. Q. Ho is partially supported by the Hong Kong RGC ECS grant 26305322 and the RGC GRF grant 16301324. P. Li is partially supported by the National Natural Science Foundation of China (Grant No. 12101348).

\printbibliography
\end{document}